\documentclass[12pt]{amsart}

\newtheorem{theorem}{Theorem}[section]

\newtheorem{lemma}[theorem]{Lemma}
\newtheorem{proposition}[theorem]{Proposition}

\theoremstyle{definition}
\newtheorem{definition}[theorem]{Definition}

\theoremstyle{remark}
\newtheorem{remark}[theorem]{Remark}

\numberwithin{equation}{section}

\newcommand{\R}{\mathbb{R}}

\newcommand{\p}{\partial}

\newcommand{\sig}{\Sigma}
\newcommand{\s}{\sigma}

\newcommand{\n}{\nabla}
\newcommand{\z}{\bar{z}}

\newcommand{\la}{\langle}
\newcommand{\ra}{\rangle}

\newcommand{\der}[1]{\frac{\partial}{\partial #1}}
\newcommand{\fder}[2]{\frac{\partial #1}{\partial #2}}

\newcommand{\vp}{\varphi}

\newcommand{\mC}{\mathcal{C}}

\newcommand{\V}{\mathcal{V}}

\begin{document}
\title[Sharp eigenvalue bounds and minimal surfaces in the ball]
{Sharp eigenvalue bounds and minimal surfaces in the ball}
\author{Ailana Fraser}
\address{Department of Mathematics \\
                 University of British Columbia \\
                 Vancouver, BC V6T 1Z2}
\email{afraser@math.ubc.ca}
\author{Richard Schoen}
\address{Department of Mathematics \\
                 Stanford University \\
                 Stanford, CA 94305 \\
                 \& Department of Mathematics \\
                 University of California \\
                 Irvine, CA 92617}
\thanks{
A. Fraser was partially supported by  the 
Natural Sciences and Engineering Research Council of Canada and R. Schoen
was partially supported by the National Science Foundation [DMS-1105323] \& [DMS-1404966].}
\email{rschoen@math.uci.edu}

\begin{abstract} We prove existence and regularity of metrics on a surface with
boundary which maximize $\sigma_1 L$ where $\sigma_1$ is the first nonzero Steklov
eigenvalue and $L$ the boundary length. We show that such metrics arise as the induced
metrics on free boundary minimal surfaces in the unit ball $B^n$ for some $n$. In the case
of the annulus we prove that the unique solution to this problem is the induced metric on
the critical catenoid, the unique free boundary surface of revolution in $B^3$. We also show that the unique solution on the M\"obius band is achieved by an explicit $S^1$ invariant embedding
in $B^4$ as a free boundary surface, the critical M\"obius band. For oriented surfaces of genus
$0$ with arbitrarily many boundary components we prove the existence of maximizers which
are given by minimal embeddings in $B^3$. We characterize the limit as the number of boundary
components tends to infinity to give the asymptotically sharp upper bound of $4\pi$. We 
also prove multiplicity bounds on $\sigma_1$ in terms of the topology, and we give a lower
bound on the Morse index for the area functional for free boundary surfaces in the ball.
\end{abstract}

\maketitle

\section{Introduction}
If we fix a smooth compact surface $M$ with boundary, we can consider all Riemannian
metrics on $M$ with fixed boundary length, say $L(\p M)=1$. We can then hope to find
a canonical metric by maximizing a first eigenvalue. The eigenvalue problem
which turns out to lead to geometrically interesting maximizing metrics is the Steklov
eigenvalue; that is, the first eigenvalue $\sigma_1$ of the Dirichlet-to-Neumann map on $\p M$. In
our earlier paper \cite{FS} we made a connection of this problem with minimal surfaces 
$\Sigma$ in
a euclidean ball which are proper in the ball and which meet the boundary of the ball
orthogonally. We refer to such minimal surfaces as {\it free boundary} surfaces since they
arise variationally as critical points of the area among surfaces in the ball whose boundaries
lie on $\p B$ but are free to vary on $\p B$. The orthogonality at $\p B$ makes the area
critical for variations that are tangent to $\p B$ but do not necessarily fix $\p \Sigma$.
Given an oriented surface $M$ of genus $\gamma$ with $k\geq 1$ boundary components, we let
$\sigma^*(\gamma,k)$ denote the supremum of $\sigma_1L$ taken over all smooth metrics
on $M$. Given a non-orientable surface $M$ we let $k$ denote the number of boundary
components and $\gamma$ denote the genus of its oriented double covering. Thus the M\"obius band has $\gamma=0$ and $k=1$ while the Klein bottle with a disk removed has $\gamma=1$
and $k=1$. We then let $\sigma^\#(\gamma,k)$ denote the supremum of $\sigma_1L$ taken over all smooth metrics on $M$.

In this paper we develop the theory in several new directions. First we develop a
general existence and regularity theory for the problem. We show in Section \ref{section:extremal}
that for any compact surface $M$ with boundary, a smooth maximizing metric $g$ exists on $M$
provided the conformal structure is controlled for any metric near the maximum. Furthermore
we show that the multiplicity of the eigenvalue is at least two and there are independent eigenfunctions which define a conformal harmonic map to the unit ball whose image is a
free boundary surface such that the induced metric from the ball is equal to a constant times
$g$ on the boundary. Our proof involves a canonical regularization procedure which produces
a special maximizing sequence for which a carefully chosen set of eigenfunctions converges 
strongly in $H^1$ to a limit. We then show that the limit defines a continuous map which is stationary for the free boundary problem. The higher regularity then follows from minimal surface
theory. 

In Section \ref{section:properties} we prove boundedness of the conformal structure for
nearly maximizing metrics in the
genus $0$ case with arbitrarily many boundary components. This proof involves showing
that the supremum value for $\sigma_1L$ strictly increases when a boundary component is
added. It is then shown by delicate constructions of comparison functions that if the conformal
structure degenerates for a sequence, then $\sigma_1$ for the sequence must be
asymptotically bounded above by the supremum for surfaces with fewer boundary components.
We also show that the conformal structure is controlled for nearly maximizing metrics on
the M\"obius band. Combining this with the work of Section \ref{section:extremal} we obtain the following main
theorem. 
\begin{theorem} Let $M$ be either an oriented surface of genus $0$ with $k\geq 2$
boundary components or a M\"obius band. There exists on $M$ a smooth metric $g$ which
maximizes $\sigma_1L$ over all metrics on $M$. Moreover there is a branched conformal minimal 
immersion $\varphi:(M,g)\to B^n$ for some $n\geq 3$ by first eigenfunctions so that $\varphi$ is a 
$\sigma$-homothety from $g$ to the induced metric $\varphi^*(\delta)$ where $\delta$ is
the euclidean metric on $B^n$.
\end{theorem}
The terminology $\sigma$-homothetic means that $g$ is a constant times the induced metric on
the boundary.

We are also able to uniquely characterize the maximizing metrics for the annulus and the 
M\"obius band. In the case of the annulus we show that the unique maximizer is the `critical
catenoid', the unique portion of a suitably scaled catenoid which defines a free boundary
surface in $B^3$. We note that in \cite{FS} this result was proven by a more elementary
argument for some conformal structures on the annulus (super-critical). The proof combines the
existence results with the uniqueness theorem of Section \ref{section:unique_cc}.
\begin{theorem} If $\sig$ is a free boundary minimal surface in $B^n$ which is homeomorphic
to the annulus and such that the coordinate functions are first eigenfunctions, then $n=3$ and 
$\sig$ is congruent to the critical catenoid.
\end{theorem}
We note that there are multiplicity bounds proven in Section \ref{section:prelim} which imply that the maximizing metric in the genus $0$ case necessarily lies in $B^3$. The main idea of the proof is to show that such a surface is $S^1$ invariant, and this case was analyzed in detail in \cite{FS}. The proof of the uniqueness of the critical catenoid uses several ingredients including an analysis of the second variation of energy and area for free boundary surfaces. We show in 
Section \ref{section:index} that certain canonical vector fields reduce the area up to second
order; in fact, we show that the Morse index (for area) of any free boundary surface in $B^n$,
which does not split as a product with a line, is at least $n$. We note that this same class of variations was used by J. D. Moore and T. Schulte \cite{MS} to show that any free boundary
submanifold in a convex body is unstable. Our proof requires that the minimal surface is in the ball 
$B^n$. In the uniqueness proof it is necessary to show that the energy has high enough index, so in order to do this we solve a Cauchy-Riemann equation on $\Sigma$ to 
add a tangential component to certain normal variations to make them conformal. We note that
a systematic study of the relationship between second variation of energy and area was
done by N. Ejiri and M. Micallef \cite{EM}.

By combining the existence and uniqueness theorems we obtain the following. 
\begin{theorem}
For any metric on the annulus $M$ we have 
\[
         \sigma_1L\leq (\sigma_1L)_{cc}
\]
with equality if and only if $M$ is $\sigma$-homothetic to the critical catenoid.
In particular,
\[
     \sigma^*(0,2)=(\sigma_1L)_{cc}\approx 4\pi/1.2.
\]
\end{theorem}
See Definition \ref{def:sigma-homothetic} for the definition of $\sigma$-homothetic.

In Section \ref{section:band} we prove the analogous theorem for the M\"obius band. 
We explicitly construct
a minimal embedding by first eigenfunctions of a M\"obius band into $B^4$ which defines a
free boundary surface. We refer to this surface as the {\em critical M\"obius band}. We show that this is the only such surface which has an $S^1$ symmetry group. The following theorem characterizes this metric.
\begin{theorem}
Assume that $\Sigma$ is a free boundary minimal M\"obius band in $B^n$ such that the coordinate functions are first eigenfunctions. Then $n=4$ and  $\Sigma$ is the critical M\"obius band.
\end{theorem}
We prove this by extending the argument for the annulus case to show that such a M\"obius band
is invariant under an $S^1$ group of rotations of $B^n$. We then show that $n=4$ and the
surface is the critical M\"obius band.

Combining this with the existence theorem we obtain the following.
\begin{theorem}
For any metric on the M\"obius band $M$ we have 
\[
        \sigma_1L\leq (\sigma_1L)_{cmb} =2\pi\sqrt{3}
\]
with equality if and only if $M$ is $\sigma$-homothetic to the critical M\"obius band. In particular,
\[ \sigma^\#(0,1)=(\sigma_1L)_{cmb} =2\pi\sqrt{3}.
\]
\end{theorem}

For a surface of genus $0$ and $k\geq 3$ boundary components we are not able to explicitly
characterize an extremal metric, but we show in Section \ref{section:asymptotics} that the metric arises from  a free boundary surface in $B^3$ which is embedded and star-shaped with respect to the origin. We then analyze the limit as $k$ goes to infinity. 

\begin{theorem} The sequence $\sigma^*(0,k)$ is strictly increasing in $k$
and converges to $4\pi$ as $k$ tends to infinity. For each $k$ a maximizing metric is achieved
by a free boundary minimal surface $\Sigma_k$ in $B^3$ of area less than $2\pi$. The limit of these minimal surfaces as $k$ tends to infinity is a double disk.
\end{theorem}
The proof of this theorem uses properties of the nodal sets of first eigenfunctions (in this case,
intersection curves of planes through the origin with the surface) together with a variety
of minimal surface methods. The maximizing property of the metrics is used in an essential
way to identify the limit as a double plane. 

For compact, closed surfaces $M$ there is a question which in certain respects has parallels
to the problem we study. This is the question of maximizing $\lambda_1(g)A(g)$ over all
smooth metrics on $M$ where $\lambda_1$ is the first nonzero eigenvalue. This problem has been resolved for surfaces with non-negative Euler characteristic with contributions by several authors. First, J. Hersch \cite{H} proved that the constant curvature metrics on $\mathbb S^2$ uniquely maximize. The connection of this problem with minimal surfaces in spheres was made by
P. Li and S. T. Yau \cite{LY} who succeeded in showing that the constant curvature metric on
$\mathbb{RP}^2$ uniquely maximizes $\lambda_1A$ over all smooth metrics on 
$\mathbb{RP}^2$. It was shown by N. Nadirashvili \cite{N} that the maximizing metric on the
torus is uniquely achieved by the flat metric on the $60^0$ rhombic torus. Finally, the case of the 
Klein bottle was handled in a series of papers by Nadirashvili \cite{N} who proved the existence
of a maximizing metric, by D. Jakobson, Nadirashvili, and I. Polterovich \cite{JNP} who
constructed the maximizing metric (which is $\mathbb S^1$ invariant but not flat), and
by A. El Soufi, H. Giacomini, and M. Jazar \cite{EGJ} who proved it is the unique maximizer.
The cases of the torus and the Klein bottle are much more difficult technically than the two
sphere and the projective plane because they require a difficult theorem of \cite{N} which 
asserts the existence of a maximizing metric which is the induced metric on a branched 
conformal minimal immersion into $\mathbb S^n$ by first eigenfunctions. As in our existence
theorem there are two steps, the first being to control the conformal structure for metrics which
are near the maximum, and the second being to prove existence and regularity of the conformal metric. The first part of this theorem has been simplified by  A. Girouard \cite{G}. Recently
G. Kokarev \cite{K} has undertaken a study of existence and regularity of maximizing measures
for $\lambda_1A$ in a conformal class. He has obtained existence and partial regularity of
such measures.

We point out that the case of closed surfaces has to do with minimal immersions into spheres while our theory has to do with minimal immersions into balls with the free boundary condition.  A basic result which is important for the theory of closed surfaces is that the conformal transformations of the sphere reduce the area of minimal surfaces. This was shown by Li and Yau \cite{LY} for two
dimensional surfaces and extended by A. El Soufi and S. Ilias \cite{EI1} to higher dimensions.
The conformal transformations of the ball do not seem to have this property for free boundary minimal surfaces, and this makes our theory, particularly the uniqueness theorems, more difficult.
We note that in \cite{FS} it was shown that the conformal transformations of the ball do reduce
the boundary length for two dimensional free boundary surfaces, but we do not know if they
also reduce the area.

{\em Acknowledgements}. The authors would like to thank the referees for several valuable comments which greatly improved the exposition and clarified the content.

\section{Notation and preliminaries} \label{section:prelim}

Let $(M,g)$ be a compact $k$-dimensional Riemannian manifold with boundary 
$\p M \neq \emptyset$ and Laplacian $\Delta_g$. Given a function $u \in C^\infty(\p M)$, let $\hat{u}$ be the harmonic extension of $u$:
\[
    \begin{cases}
    \Delta_g \hat{u} =0  & \text{on }\; M, \\   
    \hat{u} =u &  \text{on } \p M.
    \end{cases}
\]
 Let $\eta$ be the outward unit conormal along $\p M$. The Dirichlet-to-Neumann map is the map
 \[
             L: C^\infty(\p M) \rightarrow C^\infty(\p M)
 \]
 given by 
 \[
           Lu=\fder{\hat{u}}{\eta}.
 \]         
$L$ is a nonnegative self-adjoint operator with discrete spectrum
$\sigma_0 < \sigma_1 \leq \sigma_2 \leq \cdots$ tending to infinity.
We will refer to these as the Steklov eigenvalues.

Since the constant functions are in the kernel of $L$, the lowest eigenvalue $\sigma_0$ of $L$ is zero. The first nonzero eigenvalue $\sigma_1$ of $L$ can be characterized variationally as follows:
\[
       \sigma_1=\inf_{u \in C^1(\p M),\; \int_{\p M} u=0}
       \frac{\int_M |\n \hat{u}|^2 \,da}{\int_{\p M} u^2 \,ds}.
\]

The results of this paper concern two dimensional surfaces, and in this case the Steklov
problem has a certain conformal invariance which we now elucidate with some terminology.

\begin{definition} \label{def:sigma-homothetic}
If we have two surfaces $(M_1,g_1)$ and $(M_2,g_2)$, we say that $M_1$ and $M_2$
are {\it $\sigma$-isometric} (resp. {\it $\sigma$-homothetic}) if there is a conformal diffeomorphism
$\varphi:M_1\to M_2$ such that the pullback metric $\varphi^*(g_2)=\lambda^2 g_1$
with $\lambda=1$ (resp. $\lambda=c$ for a constant $c$) on $\p M_1$. 
\end{definition}

It is clear that if
two surfaces are $\sigma$-isometric then their Steklov eigenvalues coincide, while if
surfaces are $\sigma$-homothetic then the normalized eigenvalues $L(\p M_j)\sigma_i(M_j)$
coincide for $j=1,2$ and for all $i$. We can only hope to characterize surfaces up to 
$\sigma$-homothety by conditions on the Steklov spectrum.

We will need the following coarse upper bound which combines Theorem 2.3 of \cite{FS}
with a bound of G. Kokarev \cite{K}.
\begin{theorem} \label{theorem:yy} 
Let $M$ be a compact oriented surface of genus $\gamma$ with $k$ boundary components. Let 
$\sigma_1$ be the first non-zero eigenvalue of the Dirichlet-to-Neumann operator on $M$ with metric $g$. Then
\[
       \sigma_1 L(\p M) \leq \min\{2 (\gamma+k) \pi, 8\pi[(\gamma+3)/2]\}
\]
where $[x]$ denotes the greatest integer less than or equal to $x$.
\end{theorem}

We will also need the following multiplicity bounds for the first eigenvalue on surfaces. These results appeared almost simultaneously in two other papers, \cite{KKP} and \cite{J1}, with some improvements in \cite{KKP}. The approaches behind all the proofs go back to the ideas of Cheng and Besson on multiplicity bounds of the eigenvalues of the Laplacian on closed surfaces. Some further results have also been obtained in \cite{J2}.
\begin{theorem} \label{multiplicity} 
For a compact connected oriented surface $M$ of genus $\gamma$ with $k$ boundary components, the multiplicity of $\s_i$ is at most $4\gamma+2i+1$.
\end{theorem}

\begin{proof} We first show that for any eigenfunction $u$ on $M$, the set of points 
$S=\{p\in\bar{M}:\ u(p)=0,\ \nabla u(p)=0\}$ is finite. It is clear that this set is discrete
in the interior of $M$, and standard unique continuation results imply that if there is point 
$p\in S\cap\p M$, then $u$ vanishes to a finite order at $p$, the point $p$ is isolated in $S$,
and the zero set of $u$ near $p$ consists of a finite number of arcs from $p$ which meet the boundary transversely. To see this we choose conformal coordinates $(x,y)$ centered at $p$ so that $M$ is the upper half plane near $(0,0)$. The function $u$ is then harmonic and satisfies
the boundary condition $u_y=-\sigma u\lambda$ for $y=0$ where $\sigma$ is the eigenvalue and
$\lambda$ is a smooth positive function of $x$. The function $v=e^{\sigma\lambda y}u$ then satisfies the equation $\Delta(e^{-\sigma\lambda y}v)=0$ together with the homogeneous Neumann boundary condition. This may be written in the form
\[ \Delta v+b\cdot\nabla v+cv=0,\ v_y(x,0)=0 
\]
where the coefficient functions are smooth and bounded up to the boundary. We then extend
$v$ by even reflection to a full neighborhood of $(0,0)$ and observe that the extended function
$\hat{v}$ satisfies
\[ \Delta \hat{v}+\hat{b}\cdot\nabla\hat{v}+\hat{c}\hat{v}=0
\]
where $\hat{b}$ is the odd extension of $b$ and $\hat{c}$ the even extension of $c$. 
The function $\hat{v}$ is at least $C^{1,1}$, because it satisfies the Neumann condition.
The
coefficients $\hat{b}$ and $\hat{c}$ are bounded and therefore we may apply unique continuation results (see \cite{A}) to assert that $v$ and hence $u$ vanishes to finite order at $(0,0)$. From
the boundary condition on $u$ it follows that the leading order homogenous harmonic polynomial
$P(x,y)$ in the Taylor expansion of $u$ at $(0,0)$ is even under reflection across the $x$-axis and satisfies $P_y(x,0)=0$. It follows that $P_x(x,0)\neq 0$ for $x\neq 0$ and the zero set of $P$
consists of lines through the origin which are symmetric under reflection across the $x$-axis
and which do not contain the $x$-axis. The conclusion for the zero set of $u$ then follows.

Let $\vp$ be an $i$-th eigenfunction. Let $p$ be a point in the interior of $M$ with $\vp(p)=0$. First we show that the order of vanishing of $\vp$ at $p$ is less than or equal to $2\gamma +i$. To see this, suppose the order of vanishing of $\vp$ is $d$.  By \cite{Ch} (Theorem 2.5), the nodal set of $\vp$ consists of a number of $C^2$-immersed one dimensional closed submanifolds, which meet at a finite number of points. Therefore, the nodal set of $\vp$ consists of a finite number of immersed circles and arcs between boundary components. Since the sign of an eigenfunction changes around any of its zeros, there are an even number of arcs meeting any given boundary curve of $M$. By gluing a disk on each boundary component of $M$ and deforming the disk to a point, we can view $M$ as a compact surface $S$ of genus $\gamma$ with $k$ points $p_1, \ldots, p_k$ removed, where each point is identified with a  boundary curve of $M$. Since the nodal set of a spherical harmonic of order $d$ in $\mathbb{R}^2$ consists of $d$ lines passing through the origin, there exist injective piecewise $C^1$ maps $\Phi_j: S^1 \rightarrow S$, $j=1, \ldots, d$ such that $\Phi_j(S^1) \cap \Phi_l(S^1)$, $j \neq l$, consists of a finite number of points, and $\Phi_j(S^1) \subset \vp^{-1}(0)\cup \{p_1, \ldots, p_k \}$, $j=1, \ldots, d$. If $d > 2\gamma + i$, then there exist $n_1, \ldots, n_{2\gamma+i+1} \in \mathbb{Z}$ not all zero such that the homology class of $S$ represented by $\sum_{j=1}^{2\gamma+i+1} n_j \Phi_j$ is zero. It follows from the argument of \cite{Ch} (Lemma 3.1) that $S \setminus \Phi_1(S^1) \cup \ldots \cup \Phi_d(S^1)$ has at least $i+2$ connected components. Therefore, $M \setminus \vp^{-1}(0)$ has at least $i+2$ components. But by the nodal domain theorem for Steklov eigenfunctions (see \cite{KS}) there can be at most $i+1$ connected components of $M \setminus \vp^{-1}(0)$. Therefore the order of vanishing of $\vp$ is less than or equal to $2\gamma +i$.

Now, arguing as in \cite{Be} and \cite{Ch}, we show that the multiplicity $m_i$ of $\s_i$ is at most $4\gamma+2i+1$.  
Let $\vp_1, \ldots , \vp_{m_i}$ be a basis for the $i$-th eigenspace. Let $(x,y)$ be local conformal coordinates near an interior point $p$ of $M$. Consider the system of equations
\begin{equation} \label{eq:system}
        \sum_{j=1}^{m_i} a_j \frac{\p^s \vp_j}{\p x^t \p y^{s-t}} (p)=0,
        \quad s=0, \ldots , 2\gamma +i, \;\; 0 \leq t \leq s
\end{equation}
This is a system of $(2\gamma+i+1)(2 \gamma +i+2)/2$ equations in $m_i$ unknowns, $a_1, \ldots, a_{m_i}$. Since $\Delta \vp_j=0$, we have for any integer $q$, and $0 \leq t \leq q$:
\[
        \frac{\p^{q+2}\vp_j}{\p x^{l+2}\p y^{q-l}}(p)+ \frac{\p^{q+2}\vp_j}{\p x^{l}\p y^{q-l+2}}(p)=0.
\]
Using this, it can then be shown that $(2\gamma+i-1)(2\gamma +i)/2$ of the equations (\ref{eq:system}) follow from the other ones. Therefore, to solve (\ref{eq:system}), it suffices to solve the remaining $4\gamma+2i+1$ equations in $m_i$ unknowns. If $m_i >4\gamma+2i+1$, then this homogeneous system of linear equations has a nontrivial solution. But then $\vp=\sum_{j=1}^{m_i} a_j \vp_j$ would be an $i$-th eigenfunction that vanishes to order $2\gamma+i+1$, a contradiction. Therefore, $m_i \leq 4\gamma+2i+1$.
\end{proof}

For non-orientable surfaces, we have:

\begin{theorem} \label{unoriented_mult} For a compact connected non-orientable surface $M$ with $k$ boundary components and Euler characteristic $\chi(M)$, the multiplicity of $\s_i$ is at most $4(1-\chi(M)-k)+4i+3$.
\end{theorem}

\begin{proof}
We can view $M$ as a domain in a compact surface $S$ of Euler characteristic $\chi(M)+k$. Then, by an argument similar to \cite{Be}, Theorem 2.2, the multiplicity of $\s_i$ is at most $4\tilde{\gamma} +4i+3$, where $\tilde{\gamma}=1-\chi(M)-k$ is the genus of the orientable double cover of $S$.
\end{proof}

\section{The Morse index of free boundary solutions in the ball} \label{section:index}

The result of this section is used in section \ref{section:unique_cc}, but is of interest in its own right, so we include  it here as a general property of free boundary solutions.
Following \cite{FS} we will say that a minimal submanifold $\sig$, properly immersed in a domain 
$\Omega\subset\R^n$, is a {\it free boundary solution} if the outward unit normal vector of $\p\sig$ (the conormal vector) agrees with the outward unit normal to $\p\Omega$ at each point of $\p\sig$. 
If $\vp$ is an isometric minimal immersion of $M$ into the unit ball $B^n$ such that $\sig=\vp(M)$
is a free boundary solution, then the coordinate functions $\vp^1,\ldots, \vp^n$ are Steklov
eigenfunctions with eigenvalue $1$. In this section we show that certain normal deformations
decrease volume to second order. For a free boundary solution there is a Morse index for
deformations which preserve the domain $\Omega$ but not necessarily the boundary of $\sig$.
In general this is larger than the Morse index for deformations which fix the boundary. For example,
the Morse index of the critical catenoid is at least three for deformations which fix the ball, while
it is zero for deformations which fix the boundary of $\sig$. 

We consider a free boundary submanifold $\sig^{k}$ in the ball $B^n$. For a normal variation $W$ we have the index form for area given by
\[ S(W,W)=\int_\sig (|D W|^2-|A^W|^2)\ da-\int_{\p\sig} |W|^2 \ ds
\]
where $A$ denotes the second fundamental form of $\sig$.

\begin{theorem}   \label{theorem:index}
If $\sig^{k}$ is a free boundary solution in $B^n$ and $v\in\R^n$, then
we have
\[ S(v^\perp,v^\perp)=-k\int_{\sig} |v^\perp|^2\ da.
\]
If $\sig$ is not contained in a product $\sig_0 \times \mathbb{R}$ where $\sig_0$ is a free boundary solution, then the Morse index of $\sig$ is at least $n$.
In particular, if $k=2$ and $\sig$ is not a plane disk, its index is at least $n$.
\end{theorem}
\begin{proof} 
Fix a point $x_0 \in \p \Sigma$ and choose local orthonormal frames $e_1,\ldots,e_{k}$ tangent to $\Sigma^k$, where $e_{k}=x$ along $\p\sig$, and $\nu_1, \ldots, \nu_{n-k}$ normal to $\Sigma^k$ such that $(D \nu_\alpha)^\perp=0$ at $x_0$. If $h_{ij}^\alpha=(D_{e_i}e_j)\cdot\nu_\alpha$ is the second fundamental form in this basis, then we have $h_{ik}^\alpha=(D_{e_i} x) \cdot \nu_\alpha=e_i \cdot \nu_\alpha=0$ for $i<k$. Therefore,
\[
       D_x \nu_\alpha
       =\sum_{i=1}^k(D_x \nu_\alpha \cdot e_i)e_i 
       =-h_{kk}^\alpha x.
\]

Without loss of generality we may assume $|v|=1$, and we compute $S(v^\perp,v^\perp)$. Observing that $v^\perp$ is a Jacobi field we have
\[ 
    S(v^\perp,v^\perp)=\int_{\p\sig}[ v^\perp \cdot D_x v^\perp- |v^\perp|^2]\ ds.
\]
We calculate the first term using the observation above and the minimality of $\sig$ 
\begin{align*}
     v^\perp \cdot D_x v^\perp
     &= v^\perp \cdot D_x \left[ \sum_{\alpha=1}^{n-k} (v \cdot \nu_\alpha) \nu_\alpha \right]  \\
     &= - v^\perp \cdot \left[ \sum_{\alpha=1}^{n-k} \left( h_{kk}^\alpha (v \cdot x) \nu_\alpha
              + h_{kk}^\alpha (v \cdot \nu_\alpha) x \right) \right] \\
    &=  \sum_{\alpha=1}^{n-k}\sum_{i=1}^{k-1}h_{ii}^\alpha (v\cdot x)(\nu_\alpha \cdot v^\perp).
\end{align*}
Now if we let $v_0=v-(v\cdot x)x$ be the component of $v$ tangent to $S^{n-1}$, then we have 
$div_{\p\sig}(v_0)=-(k-1)v\cdot x$. On the other hand 
$v_0=v_1 + \sum_{\alpha=1}^{n-k} (v \cdot \nu_\alpha) \nu_\alpha$ where $v_1$ is the component of $v$ tangent to $\p\sig$, and so 
$div_{\p\sig}(v_0)=div_{\p\sig}(v_1)-\sum_{i=1}^{k-1}\sum_{\alpha=1}^{n-k}(v \cdot \nu_\alpha) h_{ii}^\alpha$.  Putting this information into the index form we find
\[ 
      S(v^\perp,v^\perp)=\int_{\p\sig}[(v\cdot x)(div_{\p\sig}(v_1)+(k-1)(v\cdot x))-|v^\perp|^2]\ ds.
\]
We apply the divergence theorem to the first term to get
\[ 
      S(v^\perp,v^\perp)=\int_{\p\sig}[-|v_1|^2+(k-1)(v\cdot x)^2-|v^\perp|^2]\ ds.
\]
Since $|v_1|^2=1-(v\cdot x)^2-|v^\perp|^2$ we have
\[ 
     S(v^\perp,v^\perp)=\int_{\p\sig}[-1+k(v\cdot x)^2]\ ds.
\]

Now we consider the vector field $V=x-k(v\cdot x)v$ and apply the first variation formula
on $\sig$
\[ 
     \int_\sig\ div_\sig V\ da=\int_{\p\sig}\ V\cdot x\ ds.
\]
Direct computation gives $div_\sig V=k(1-|v^t|^2)=k|v^\perp|^2$ where $v^t$ denotes
the tangential part of $v$. Putting this into the formula above we get
\[ 
     S(v^\perp,v^\perp)=-\int_{\p\sig}\ V\cdot x\ ds=-k\int_{\sig}|v^\perp|^2\ da
\]
as desired.

If there is a $v\neq 0$ such that $v^\perp \equiv 0$ on $\sig$, then $v$ lies in the tangent plane to $\Sigma$ at each point, and $\sig$ is contained in the product $\sig_0^{k-1} \times \mathbb{R}$. 
$\sig_0$ is the intersection of $\sig$ with the hyperplane through the origin orthogonal to $v$, and hence is a free boundary solution.  Therefore, if $\sig$ is not contained in a product $\sig_0 \times \mathbb{R}$ where $\sig_0$ is a free boundary solution, then the Morse index of $\sig$ is at least $n$.
In particular, if $k=2$ and $\sig$ is not a plane disk, its index is at least $n$.
\end{proof}

\section{Properties of metrics with lower bounds on $\sigma_1L$} \label{section:properties}

In this section and the next we take up the existence and regularity question for metrics which maximize $\sigma_1 L$. For an oriented surface $M$ with genus $\gamma$ and $k\geq 1$ boundary components we define
the number $\sigma^*(\gamma,k)$ to be the supremum of $\sigma_1(g)L_g(\p M)$ taken
over all smooth metrics on $M$. Theorem \ref{theorem:yy} tells us that 
\[ \sigma^*(\gamma,k)\leq \min\{2 (\gamma+k) \pi, 8\pi[(\gamma+3)/2]\}.
\] 
Weinstock's theorem implies that $\sigma^*(0,1)=2\pi$, and our analysis of rotationally
symmetric metrics on the annulus \cite{FS} implies that $\sigma^*(0,2)$ is at least the
value for the critical catenoid, so $\sigma^*(0,2)>2\pi$. 

There are two general ways in which metrics can degenerate in our problem. The first
is that the conformal structure might degenerate and the second is that the boundary
arclength measure might become singular even though the conformal class is controlled.
We will need the following general result which limits the way in which the boundary measures
can degenerate provided the conformal structures converge and $\sigma_1$ is not too small. If we choose a metric $g$ on $M$ and a measure $\mu$ with a smooth density relative to the arclength
measure of $g$ on $\p M$, then $\sigma_1(g,\mu)$ is defined as the first
Steklov eigenvalue of any metric conformal to $g$ with arclength measure $\mu$ on $\p M$.
To be clear we have
\[ 
  \sigma_1(g,\mu)=\inf\left\{\frac{\int_M|\nabla\hat{u}|^2\ da_g}{\int_{\p M}u^2\ d\mu}:\ \int_{\p M}u\ d\mu=0\right\}.
\]
By the weak* compactness of measures, if we take a sequence of probability measures $\mu_i$
on $\p M$, then there is a subsequence which converges to a limit probability measure $\mu$.
The following result is analogous to Theorem 1.1.3 in \cite{G} and Lemma 3.1 in \cite{K} for closed surfaces.

\begin{proposition}\label{pt.mass} Assume that we have a sequence of metrics $g_i$ converging
in $C^2$ norm to a metric $g$ on $M$, and a sequence of smooth probability measures $\mu_i$ converging in the weak* topology to a measure $\mu$. Assume that there is a $\lambda>2\pi$
so that $\sigma_1(g_i,\mu_i)\geq \lambda$ for each $i$. It then follows that $\mu$ has no 
point masses.  
\end{proposition}

\begin{proof} Assume on the contrary there is a point $p\in\p M$ with $a=\mu(\{p\})>0$. We consider two cases: First assume that $a<1$. In this case we will show that $\sigma_1(g_i,\mu_i)$
converges to $0$. We can find a sufficiently small open interval $I$ in $\p M$ containing $p$
so that $b=\mu(\p M\setminus I)>0$. We then choose a smooth function $u\leq 1$ supported in $I$
with $u(p)=1$ and so that the Dirichlet integral of the harmonic extension $\hat{u}$ is
arbitrarily small (relative to the limit metric $g$). We then consider the function
$\varphi=u-\bar{u}$ where $\bar{u}$ is the average of $u$ with respect to $\mu$. The
harmonic extension $\hat{\varphi}$ has the same Dirichlet integral as that of $\hat{u}$,
and we have $\bar{u}\leq \mu(I)=1-b$, and thus $\varphi(p)\geq b$, and so we have
\[ \int_{\p M}\varphi^2\ d\mu\geq ab^2>0.
\]
It follows that the Rayleigh quotient of $\varphi$ with respect to $(g_i,\mu_i)$
becomes arbitrarily small for large $i$ contradicting the lower bound on $\sigma_1(g_i,\mu_i)$.

The second and more difficult case is when $a=1$, and the support of $\mu$ is $p$. In this case
we show that $\limsup_{i\to\infty} \sigma_1(g_i,\mu_i)\leq 2\pi$ in contradiction to our assumption.
To see this we consider a disk $U$ in $M$ whose boundary intersects $\p M$ in an interval $I$
about $p$. We define a measure $\nu_i$ on $\p U$ to be equal to $\mu_i$ on $I$ and to be zero on
$\p U\setminus I$. We then take a conformal map $F_i$ from $U$ to the unit disk $D$ for which
$\int_{\p U}F_i\ d\nu_i=0$. Since the measures $\nu_i$ are converging to a point mass at $p$,
it follows that the maps $F_i$ outside a neighborhood of $p$ are arbitrarily close to a point of 
$\p D$ for $i$ large. By composing $F_i$ with a rotation we may assume that this point is $(1,0)$,
so we have arranged that the $F_i$ converge to the constant map $(1,0)$ on compact subsets of
$\bar{U}\setminus\{p\}$; in particular, this is true near $\p U\setminus I$. We may then extend $F_i$
to a Lipschitz map $\hat{F}_i:M\to \mathbb R^2$ which maps the complement of a neighborhood
of $\bar{U}$ to the point $(1,0)$ and such that the Dirichlet integral of $\hat{F}_i$ converges to
$2\pi$, the Dirichlet integral of $F_i$. The average of $\hat{F}_i$ on $\p M$ with respect to
$\mu_i$ converges to $0$, so we may use $\hat{F}_i$ minus its average as comparison 
functions as in the Weinstock proof to conclude that 
$\limsup_{i\to\infty}\sigma_1(g_i,\mu_i)\leq 2\pi$. This contradiction completes the proof.
\end{proof}

We will prove a result on the strict monotonicity of $\sigma^*$ when the number of boundary components is increased. We will first need the following lemma concerning approximate
eigenfunctions. We let $(M_1,g)$ be any compact Riemannian manifold with boundary
and let $T$ denote the Dirichlet-Neuman map acting on functions on $\partial M_1$.
We will use $\|\cdot\|$ to denote the $L^2(\partial M_1)$-norm.

\begin{lemma}\label{approx_efcn}
Let $\alpha,\delta\in (0,1)$ and let $\sigma> 0$. Assume that $v\in C^1(\partial M_1)$ with  $\|v\|=1$ and with $\|Tv-\sigma v\|<\delta$. There exist Steklov eigenvalues
$\sigma^{(1)},\ldots, \sigma^{(p)}$ with $|\sigma^{(j)}-\sigma|<\delta^\alpha$ for an integer
$p\geq 1$. Let $V$ be the direct sum of the corresponding eigenspaces and let $w$ be the
orthogonal projection of $v$ into $V$. We then have $\|v-w\|<\delta^{1-\alpha}$.
\end{lemma}

\begin{proof} Let $u_i$ ($i=0,1,2,\ldots$) be a complete orthonormal basis of Steklov eigenfunctions with eigenvalues $\sigma_i$. We expand $v=\sum_{i=0}^\infty a_iu_i$.
We then have
\[ \sum_{i=0}^\infty a_i(\sigma_i-\sigma)u_i=Tv-\sigma v,
\]
so it follows that
\[ \left(\sum_{i=0}^\infty a_i^2(\sigma_i-\sigma)^2\right)^{1/2}=\|Tv-\sigma v\|<\delta.
\]
Thus we have
\[ \delta^\alpha\left(\sum_{\{i:|\sigma_i-\sigma|>\delta^\alpha\}}a_i^2\right)^{1/2}<\delta,
\]
and this implies that $\|v-w\|<\delta^{1-\alpha}$ as claimed. Note that since $\delta<1$
we must have at least one $i$ with $|\sigma_i-\sigma|<\delta^\alpha$.
\end{proof}

We are now in a position to prove the following monotonicity result.

\begin{proposition}\label{sigma*_inc} For any $\gamma$ and $k$ we have 
$\sigma^*(\gamma,k+1)\geq \sigma^*(\gamma,k)$. If $\sigma^*(\gamma,k)$ is achieved by a smooth metric for some $\gamma$ and $k$ it then follows that $\sigma^*(\gamma,k+1)>\sigma^*(\gamma,k)$. 
\end{proposition}

\begin{proof} The weak inequality $\sigma^*(\gamma,k+1)\geq \sigma^*(\gamma,k)$ can be
seen in a direct manner by showing that the surface gotten by removing a small disk from
$M$ has $\sigma_1L$ converging to that of $M$ as the radius of the disk tends to zero. It
is a consequence of the argument given below.

For the strict inequality, let $M$ be a surface of genus $\gamma$ with $k$ boundary components and let $g$ be a metric with $L_g(\p M)=1$ and $\sigma_1(g)=\sigma^*(\gamma,k)$.
Let $p$ be a point of $M$, and without loss of generality assume that $g$ is
flat in a neighborhood of $p$ with euclidean coordinates $x,y$ centered at $p$. 
Let $D_\epsilon$
be the disk of radius $\epsilon$ centered at $p$, and let $M_\epsilon$ be the surface 
$M\setminus D_\epsilon$ which has genus $\gamma$ and $k+1$ boundary components. 
For a positive constant $\lambda$ to be chosen, we let
$g_\lambda$ denote a metric which is equal to $g$ outside a fixed neighborhood of $p$ and which is equal to $\lambda^2(r\log(1/r))^{-2}$ times the euclidean metric near $p$.  
Let $M_{\epsilon,\lambda}$ denote $M\setminus D_\epsilon$ with the 
metric $g_\lambda$, and let $\sigma_1(\epsilon,\lambda)$ be the corresponding first Steklov eigenvalue. Our goal is to show that for $\epsilon$ small there is a choice of $\lambda$
bounded from above and below by positive constants (independent of $\epsilon$) such that
\begin{equation} \label{hole} \sigma_1(\epsilon,\lambda)\geq \sigma_1(g)-o(1)(\log(1/\epsilon))^{-1}
\end{equation}
where $o(1)$ refers to a term which goes to $0$ with $\epsilon$. Since $\sigma_1(g)=\sigma^*(\gamma,k)$ it then follows that 
\[ L_{g_\lambda}(\partial M_\epsilon)\sigma_1(\epsilon,\lambda)=\left(1+\frac{2\pi\lambda}{\log(1/\epsilon)}\right)\sigma_1(\epsilon,\lambda)
\geq \sigma^*(\gamma,k)+\frac{\pi\lambda}{\log(1/\epsilon)}>\sigma^*(\gamma,k)
\]
for $\epsilon$ small. It will then follow that $\sigma^*(\gamma,k+1)>\sigma^*(\gamma,k)$. (Note that the metric $g_\lambda$ is a metric of constant negative curvature $-\lambda^{-2}$ with a cusp.)

We let $\sigma_i(\epsilon,\lambda)$ denote the $i$-th eigenvalue of $M_{\epsilon,\lambda}$. To prove (\ref{hole}) we study the dependence of the low eigenvalues $\sigma_i(\epsilon,\lambda)$ on $\lambda$ for very small $\epsilon$. We show that there are positive constants 
$\lambda_1<\lambda_2$ such that for $\lambda\geq\lambda_2$ we have 
\[ \sigma_1(\epsilon, \lambda)<\frac{1}{2}\sigma_1(g)
\]
for $\epsilon$ small while for $\lambda\leq\lambda_1$ we have
\[ \lim_{\epsilon\to 0}\sigma_1(\epsilon,\lambda)=\sigma_1(g).
\]
We will show that there is a value $\lambda_\epsilon\in (\lambda_1,\lambda_2)$ such 
that the lowest eigenvalue $\sigma_1(\epsilon, \lambda_\epsilon)$ has multiplicity at least two 
and converges as $\epsilon$ goes to $0$ to $\sigma_1(g)$.

Assume that the multiplicity of $\sigma_1(g)$ is $k$. We let $\delta=c\sqrt{\lambda/\log(1/\epsilon)}$ for a suitable constant $c$ and we use Lemma \ref{approx_efcn}
with $\alpha=1/2$ to show that the dimension of the space $V$ which is the direct sum of eigenspaces for eigenvalues $\sigma(\epsilon,\lambda)$ of $M_{\epsilon,\lambda}$ with 
\[ |\sigma(\epsilon,\lambda)-\sigma_1(g)|<\sqrt{\delta}
\]
is at least $k$. (In fact we will show below that it is either $k$ or $k+1$ depending
on the value of $\lambda$ for $\epsilon$ small.) To show that the dimension of $V$
is at least $k$ we note that by elliptic theory there is a fixed constant $c$ so that any
normalized first eigenfunction $u$ is $g$ is bounded by $c$ in $D_\epsilon$. For
such a function $u$ we let $v=u/\|u\|$ where the normalization is done in 
$L^2(\partial M_{\epsilon,\lambda})$. From the discussion above we see that for a constant 
$c$ we have
\[ \|Tv-\sigma_1(g)v\|<\delta.
\]
It then follows from Lemma \ref{approx_efcn} that the dimension of $V$ is at least
$k$ and there are independent functions $w_1,\ldots, w_k$ in $V$ such that
$\|w_i-v_i\|<\sqrt{\delta}$ where $v_i=u_i/\|u_i\|$ for an orthonormal basis $u_1,\ldots, u_k$
of eigenfunctions on $M$.

We now consider a Green's function $G$ on $M$ with a pole at $p$, with $G=0$ on 
$\partial M$,  and such that $G(x,y)-\log(1/r)$ is regular near $p$ (the origin of the
$(x,y)$ coordinates). We then have the outer normal derivative on $\partial D_\epsilon$ with respect to $g_\lambda$ given by
\[ \nabla_\nu G=-\frac{\epsilon\log(1/\epsilon)}{\lambda}\partial_r G=\lambda^{-1}G+O(1)
\]
as $\epsilon$ goes to $0$ with $\lambda$ bounded. Letting $\hat{v}=G/\|G\|$ it
follows that
\[ \|T\hat{v}-\lambda^{-1}\hat{v}\|<\hat{\delta}\equiv c(\log(1/\epsilon))^{-1/2}
\]
since $\|G\|$ is a constant times $\sqrt{\log(1/\epsilon)}$. It then follows from
Lemma \ref{approx_efcn} with $\alpha=1/4$ that there are eigenvalues of 
$M_{\epsilon,\lambda}$
within $\hat{\delta}^{1/4}$ of $\lambda^{-1}$ for $\epsilon$ small. In particular we
can fix $\lambda_2$ so that $\sigma_1(\epsilon,\lambda_2)<\sigma_1(g)/2$. Again from
Lemma \ref{approx_efcn} ($\alpha=1/4$) we can also find $\hat{w}$ in $\hat{V}$ (the direct
sum of the eigenspaces for eigenvalues within $\hat{\delta}^{1/4}$ of $\lambda^{-1}$)
such that $\|\hat{w}-\hat{v}\|<\hat{\delta}^{3/4}$.

We choose $\lambda_1$ so that $\sigma\equiv\lambda_1^{-1}\in (\sigma_1(g),\sigma_2(g))$, and we show that for $\epsilon$ small there are exactly $k+1$ eigenvalues  counted
with multiplicity in $(0,\sigma)$. We consider a sequence $\epsilon_j\to 0$ and
harmonic functions $u_j$ on $M_{\epsilon_j}$ with boundary $L^2$ norm $\|u_j\|=1$.
We assume by extracting a subsequence that the $u_j$ converge uniformly
away from the point $p$ to a harmonic function $u$ on $M$. Under the assumptions:
(i) each $u_j$ is orthogonal in the boundary $L^2$ norm to the function $\hat{w}_j$
constructed in the previous paragraph, and (ii) $\|Tu_j\|\leq c$ for a constant $c$, we 
have $\|u\|=1$ in the $L^2$ norm of $\partial M$. To see this we do a Laurent-type decomposition for $u_j$
\[ u_j(r,\theta)=v_j(r,\theta)+w_j(r,\theta)+a_j\log(r)+b_j
\]
where $v_j$ is smooth and harmonic for $r<r_0$ ($r_0$ fixed) with $v_j(0)=0$, and $w_j$ is smooth and harmonic for $r\geq \epsilon$ with $|w_j(x)|\leq c|x|^{-1}$ for $|x|$ large. We
interpret the orthogonality condition ((i) above) by setting $\alpha_j=a_j\log(\epsilon_j)+b_j$
and observing that 
\[ \int_{\partial D_{\epsilon_j}}u_j\hat{v}_j\frac{\lambda_j}{\epsilon_j\log(1/\epsilon_j)}\ ds
=\langle u_j,(\hat{v}_j-\hat{w}_j)\rangle.
\]
This implies the bound
\[ |\alpha_j|\leq c(\log(1/\epsilon_j))^{1/8}
\]
since we have
\[ \left|\int_{\partial D_{\epsilon_j}}u_j\hat{v}_j\frac{\lambda_j}{\epsilon_j\log(1/\epsilon_j)}\ ds\right|
=\frac{2\pi|\alpha_j|\lambda_j}{\|G_j\|}\geq c|\alpha_j|(\log(1/\epsilon_j))^{-1/2}
\]
and $\|\hat{v}_j-\hat{w}_j\|<\hat{\delta}^{3/4}\leq c(\log(1/\epsilon_j))^{-3/8}$.
Now we have 
\[ \int_{\partial D_j}u_j^2\ ds_j=\int_{\partial D_j}(v_j+w_j)^2\ ds_j+(\alpha_j^2)2\pi 
\frac{\lambda_j}{\log(1/\epsilon_j)},
\]
and the second term goes to zero. To show that the first term goes to zero we
use the boundedness of $Tu_j$ from condition (ii). Recall that the outer unit normal
on $\partial D_\epsilon$ with respect to $g_\lambda$ is given by 
$-\frac{\epsilon\log(1/\epsilon)}{\lambda}\partial_r$. Therefore the bound on $Tu_j$
implies
\[ \int_{\partial D_j}(\partial_r u_j)^2\ ds\leq c(\epsilon_j\log(1/\epsilon_j))^{-1}
\]
where we have used the fact that $\lambda_j$ is bounded. Since $v_j$ is uniformly
bounded, harmonic, and vanishes at the origin, it follows that $(\partial_r v_j)^2\leq c$, and therefore we
have 
\[ \int_{\partial D_j}(\partial_r w_j)^2\ ds\leq c(\epsilon_j\log(1/\epsilon_j))^{-1}.
\]
It is a standard fact for harmonic functions in a planar annulus that the quantity
$\rho\int_{\p D_\rho}((\partial_T w)^2-(\partial_r w)^2)\ ds$ is independent of $\rho$ where $T$ denotes the unit
tangent vector with respect to the euclidean metric to $\p D_r$. In our case this integral must be zero for all $\rho\geq\epsilon$
since the integrand decays at infinity like $\rho^{-4}$. Therefore we have from above
\[ \int_{\p D_\epsilon}(\partial_Tw_j)^2\ ds=\int_{\p D_\epsilon}(\partial_rw_j)^2\ ds\leq 
c(\epsilon_j\log(1/\epsilon_j))^{-1}.
\] 
Since the integral of $w_j$ around $\p D_{\epsilon_j}$ is $0$, we have by the standard
Poincar\'e inequality on a circle
\[ \int_{\p D_{\epsilon_j}}w_j^2\ ds\leq \epsilon_j^2\int_{\p D_{\epsilon_j}}(\partial_Tw_j)^2\ ds.
\]
Combining this with the inequality above we have
\[
     \int_{\p D_{\epsilon_j}} w_j^2 \ ds 
     \leq c \epsilon_j(\log 1/\epsilon_j)^{-1}.
\]
It follows from the discussion above that 
\[ \lim_{j\to\infty}\int_{\partial D_j}u_j^2\ ds_j=0,
\]
and this completes the proof that the limit has norm one since we have uniform convergence
of $u_j$ to $u$ on $\partial M$.

We can now complete the proof that there are exactly $k+1$ eigenvalues in $(0,\sigma)$.
We take an orthonormal basis $u_1,\ldots, u_k$ for the first eigenspace of $\sigma_1(g)$,
and we consider the corresponding projections $w_1,\ldots, w_k$ which are in the finite span
of eigenspaces of $M_{\epsilon,\lambda}$. We take the direct sum of this $k$ dimensional
space with the one dimensional space spanned by $\hat{w}$ which also lies in a finite
direct sum of eigenspaces. We call this $k+1$ dimensional space $W$. If there were an eigenfunction $v$ with eigenvalue in $(0,\sigma)$ which does not lie in $W$, then we can take $u$ to be the component of $v$ orthogonal to $W$, and if we had a sequence $\epsilon_j\to 0$ and functions $u_j$ of this type, we could apply the previous paragraph to assert
that after normalizing and taking a subsequence the $u_j$ converge without loss of norm to a function $u$ which would be a sum of eigenfunctions on $M$ with eigenvalues in $(0,\sigma)$.
But the function is orthogonal to the $u_j$ (and to the constant functions), and this is a contradiction since $\sigma<\sigma_2(g)$. Thus we see that there are exactly $k+1$ eigenvalues in $(0,\sigma)$ for $\epsilon$ sufficiently small.

Finally we argue that for small positive $\epsilon$ there is a $\lambda_\epsilon$
(approximately equal to $\sigma_1(g)^{-1}$) such that $\sigma_1(\epsilon,\lambda_\epsilon)$
has multiplicity at least two. In fact we have shown that if $\lambda$ is slightly larger than
$\sigma_1(g)^{-1}$, then $\hat{w}$ is a first eigenfunction while if $\lambda$
is slightly smaller than $\sigma_1(g)^{-1}$ then $\hat{w}$ is not a first eigenfunction.
We can then take $\lambda_\epsilon$ to be the infimum of $\lambda$ for which $\hat{w}$
is a first eigenfunction. It is clear that the multiplicity cannot be one for this value of 
$\lambda$ since otherwise $\hat{w}$ would span the first eigenspace and this
would continue to hold for $\lambda$ slightly smaller.

To complete the proof, we let $u_\epsilon$ be a first eigenfunction for the problem
with $\lambda=\lambda_\epsilon$. Since the multiplicity is at least two, we can
choose $u_\epsilon$ so that $\int_{\partial D_\epsilon}u_\epsilon\ ds_\epsilon=0$.
We use the Laurent-type decomposition to write $u=v+w$ where $v$ is harmonic
in a fixed neighborhood of the origin with $v(0)=0$ and $w$ is harmonic on the plane
region with
$r\geq \epsilon$ such that $w=O(r^{-1})$ at infinity. The boundary condition then implies
\[ v_r+\sigma_1(g_\epsilon)\lambda(\epsilon\log(1/\epsilon))^{-1}v+w_r+\sigma_1(g_\epsilon)
\lambda(\epsilon\log(1/\epsilon))^{-1}w=0
\]
for $r=\epsilon$. Since $v(0)=0$ and the derivatives of $v$ are bounded near $0$, this implies 
\[ |w_r|\leq c\lambda(\epsilon\log(1/\epsilon))^{-1} |w|+c
\] 
on  $\p D_\epsilon$ where we use the fact that $\log(1/\epsilon)$ is large (since $\epsilon$ is
small). As above we use the fact for harmonic functions in a planar annulus the quantity
$\rho\int_{\p D_\rho}(|w_T|^2-|w_r|^2)\ ds$ is independent of $\rho$ where $T$ denotes the unit
tangent vector with respect to the euclidean metric to $\p D_r$. In our case this integral must be zero for all $\rho\geq\epsilon$
since the integrand decays like $\rho^{-4}$. Therefore we have from above
\[ \int_{\p D_\epsilon}|w_T|^2\ ds=\int_{\p D_\epsilon}|w_r|^2\ ds\leq 
c\lambda^2(\epsilon\log(1/\epsilon))^{-2}\int_{\p D_\epsilon}w^2\ ds+c\epsilon.
\] 
Since the integral of $w$ around $\p D_\epsilon$ is $0$, we have by the standard
Poincar\'e inequality on a circle
\[ \int_{\p D_\epsilon}w^2\ ds\leq \epsilon^2\int_{\p D_\epsilon}|w_T|^2\ ds.
\]
Combining this with the inequality above we have
\[
     \int_{\p D_\epsilon} |w_T|^2 \ ds 
     \leq c \lambda^2(\log 1/\epsilon)^{-2} \int_{\p D_\epsilon} |w_T|^2 \ ds + c\epsilon.
\]
For $\epsilon$ sufficintly small, we may absorb the squared $L^2$ norm of
$w_T$ back to obtain $\int_{\p D_\epsilon} |w_T|^2\ ds \leq c\epsilon$, and hence $\int_{\p D_\epsilon} w^2 \ ds \leq c \epsilon^3$. Since this inequality clearly holds for $v$, we have
\[
   \int_{\p D_\epsilon} u_\epsilon^2 \ ds 
     \leq 4 \int_{D_\epsilon} ( v^2+w^2) \ ds \leq c \epsilon^3.
\]

To complete the proof of (\ref{hole}) we extend $u_\epsilon$ to a function $\hat{u}$ on $M$
by defining $\hat{u}(r,\theta)=\frac{r}{\epsilon}u_\epsilon(\epsilon,\theta)$ on $D_\epsilon$. We then have 
$|\nabla\hat{u}|^2(r,\theta)= \hat{u}_r^2 + r^{-2} \hat{u}_\theta^2=\epsilon^{-2}u_\epsilon^2(\epsilon,\theta)+\epsilon^{-2} (u_\epsilon)_\theta^2(\epsilon,\theta)$. From the bounds we obtained above we can see that 
\begin{align*}
        \int_{D_\epsilon}|\nabla\hat{u}|^2\ da
        &= \frac{1}{\epsilon^2} \int_0^{2\pi} \int_0^\epsilon  [ u_\epsilon^2(\epsilon,\theta) 
                 +(u_\epsilon)_\theta^2(\epsilon,\theta) ] \ r \, dr \, d\theta  \\
        &= \frac{1}{2\epsilon}\int_{\p D_\epsilon}u_\epsilon^2\ ds
          +\frac{1}{2}\int_{\p D_\epsilon}(u_\epsilon)_T^2\ ds. 
 \end{align*}
But from above, since the derivatives of $v$ are bounded near 0,
\[
       \int_{\p D_\epsilon} (u_\epsilon)_T^2 \ ds 
       = \int_{\p D_\epsilon} v_T^2 \ ds + \int_{\p D_\epsilon} w_T^2 \ ds
       \leq c\cdot 2\pi \epsilon + c\epsilon \leq c \epsilon.
\]
Therefore, 
\[
        \int_{D_\epsilon}|\nabla\hat{u}|^2\ da 
        = \frac{1}{2\epsilon}\int_{\p D_\epsilon}u_\epsilon^2\ ds
          +\frac{1}{2}\int_{\p D_\epsilon}(u_\epsilon)_T^2\ ds
        \leq c\epsilon^2 + c \epsilon \leq c \epsilon.
\]
We now use
$\vp=\hat{u}-\int_{\p M}\hat{u}\ ds$ as a comparison function for the eigenvalue problem on $(M,g)$. 
Since $\hat{u}=u_\epsilon$ on $\p M$ and $L(\p M)=1$,
\[
      \int_{\p M} \vp^2 \ ds =\int_{\p M} \left( u_\epsilon - \int_{\p M} u_\epsilon \ ds \right)^2 \ ds
      =\int_{\p M} u_{\epsilon}^2 - \left( \int_{\p M} u_\epsilon \ ds \right)^2.
\]
This gives
\begin{align*} 
         \sigma_1(g) \left[\int_{\p M} u_{\epsilon}^2 \ ds - \left( \int_{\p M} u_\epsilon \ ds \right)^2\right]
         & \leq \int_M |\n \vp|^2 \ da
         =\int_M |\n \hat{u}|^2 \ da  \\
         & =\int_{M_\epsilon} |\n \hat{u}|^2 \ da + \int_{D_\epsilon} |\n \hat{u}|^2 \ da \\
         & \leq \int_{M_\epsilon} | \n u_\epsilon|^2 \ da + c\epsilon \\
         & = \sigma_1(g_\epsilon) + c\epsilon                  
\end{align*}
From the choice of $u_\epsilon$ we have $\int_{\p M}u_\epsilon\ ds=-\int_{\p D_\epsilon}u_\epsilon\ ds_\epsilon=0$. Using the bounds above we see that 
\begin{align*}
    \sigma_1(g)\int_{\p M}u_\epsilon^2 \ ds
    & =\sigma_1(g) \int_{\p M} u_\epsilon^2 \ ds_\epsilon
    =\sigma_1(g) 
       \left[ \int_{\p M_\epsilon} u_\epsilon^2 \ ds_\epsilon - \int_{\p D_\epsilon} u_\epsilon^2 \ ds_\epsilon \right] \\
    & =\sigma_1(g)\left[1-\frac{\lambda}{\epsilon\log(1/\epsilon)}\int_{\p D_\epsilon}u_\epsilon^2\ ds\right]
    \geq \sigma_1(g)-c\frac{\epsilon^2}{\log(1/\epsilon)}.
\end{align*}
Combining these we have
\[ \sigma_1(g_\epsilon)\geq \sigma_1(g)-c\frac{\epsilon^2}{\log(1/\epsilon)}-c\epsilon=\sigma_1(g)-o(1)(\log(1/\epsilon))^{-1}.
\]
This establishes (\ref{hole}) and completes the proof of Proposition \ref{sigma*_inc}.
\end{proof}

We will now prove the main theorem of this section which says roughly that for any metric $g$ on a surface with genus $0$ and $k$ boundary components with $\sigma_1L$ strictly above 
$\sigma^*(0,k-1)$, the conformal structure induced by $g$ lies in a compact subset of
the moduli space of conformal structures. In order to measure this notion, we note that, given a smooth surface $M$ of genus $\gamma$ with $k$ boundary components (except $\gamma=0$, $k=1,2$)
and a metric $g$ on $M$, there is a unique hyperbolic metric $g_0$ in the conformal class of $g$ such that the boundary curves are geodesics. This may be obtained by taking the hyperbolic metric
on the doubled conformal surface $\tilde{M}$ and restricting it to $M$. We can measure the
conformal class by considering the injectivity radius $i(g_0)$ on $\tilde{M}$; that is, compact
subsets of the moduli space are precisely those $g_0$ with $i(g_0)$ bounded below by a positive 
constant.

The case $\gamma=0$ and $k=1$ is handled by Weinstock's theorem, and for the case 
$\gamma=0$ and $k=2$ the doubled surface is a torus, and thus $g_0$ must be taken to be a flat metric which we normalize to have area $1$. Since the proofs for the annulus and M\"obius band  are slightly different from the case $k\geq 3$, we handle them separately in the following Proposition. We fix a topological annulus $M$ and we consider metrics $g$ on $M$. For the M\"obius band we consider the oriented double covering. We let 
$L_1\geq L_2$ denote the lengths of the boundary curves with $L=L_1+L_2$, and we let 
$\alpha=L_1/L_2\geq 1$ denote the ratio. The Riemannian surface $(M,g)$ is conformally equivalent to $[-T,T]\times S^1$ for a unique $T>0$. We have the following result.

\begin{proposition} \label{two_comp} Let $M$ be an annulus or a M\"obius band. Given $\delta>0$, there are positive numbers $\alpha_0$ and $\beta_0$ depending on $\delta$ such that if $g$ is any metric on $M$ with 
$\sigma_1(g)L\geq 2\pi(1+\delta)$, then $\alpha\leq \alpha_0$ and 
$\beta_0^{-1}\leq T\leq \beta_0$.
\end{proposition}
\begin{proof} We first handle the case of the annulus. By scaling the metric we may assume that $L=1$, so we have 
$L_1=\alpha/(1+\alpha)$ and $L_2=(1+\alpha)^{-1}$. We take a conformal diffeomorphism
from $M$ into the unit disk $D$ which takes $\Gamma_1$ to the unit circle and whose image
is a rotationally symmetric annulus. We then compose with a conformal diffeomorphism of
the disk to obtain $\vp:M\to D$ with $\int_{\Gamma_1}\vp\ ds=0$. We let $\bar{\vp}$
denote the mean value of $\vp$ over $\p M$ and observe
\[ |\bar{\vp}|=\left|\int_{\Gamma_2}\vp\ ds\right|<L_2=\frac{1}{1+\alpha}.
\]
We have $\sigma_1\int_{\p M}|\vp-\bar{\vp}|^2\leq \int_M|\nabla \vp|^2\ da<2\pi$.
This implies $\sigma_1(\alpha/(1-\alpha)-|\bar{\vp}|^2)<2\pi$ and using the lower bound on $\sigma_1$
and the upper bound on $\bar{\vp}$ we obtain $(1+\delta)(\alpha/(1+\alpha)-1/(1+\alpha))<1$ which
in turn implies $\alpha\leq \alpha_0$ where $\alpha_0=(2+\delta)/\delta$.

To obtain an upper bound on $T$ we consider the linear function $u=(L_1-L_2)+T^{-1}t$ on 
$[-T,T]\times S^1$, and observe that $\int_{\p M}u\ ds=0$. We thus have
\[ \sigma_1\int_{\p M}u^2\ ds\leq \int_M|\nabla u|^2\ da=4\pi T^{-1}.
\]
Using $\sigma_1\geq 2\pi$ and $\int_{\Gamma_2}u^2\ ds\geq L_2$ (since $u\geq 1$ when
$t=T$) we have $T\leq 2L_2^{-1}\leq \beta_0$ where we may take $\beta_0=2(1+\alpha_0)$.

The lower bound on $T$ is more subtle. We write the boundary measure as 
$\lambda(t,\theta)\ d\theta$ where $\lambda$ is defined at $t=-T$ and $t=T$. Now let us suppose
that we normalize $L=1$ and we define a function $\hat{\lambda}$ on $S^1$ to be
$\hat{\lambda}(\theta)=\lambda(-T,\theta)+\lambda(T,\theta)$. We then have 
$L=1=\int_{S^1}\hat{\lambda}\ d\theta$. There is an interval of length $\pi/2$ which carries
at least $1/4$ of the length, so let's choose the point $\theta=0$ so that
$\int_{-\pi/4}^{\pi/4}\hat{\lambda}\ d\theta\geq 1/4$. We now let $L_0=\int_{\pi/2}^{3\pi/2}\hat{\lambda}\ d\theta$, and we let $u(\theta)$ be a smooth $2\pi$-periodic function which satisfies $u(\theta)=1$ for $|\theta|\leq\pi/4$, $u(\theta)=-1$ for $\pi/2\leq\theta\leq3\pi/2$,
and $\int_{S^1}(u')^2\ d\theta\leq c_1$ for a fixed constant $c_1$. Using the test
function $u$ it is easy to derive the bound $\sigma_1L_0\leq c_2T$ for a fixed constant
$c_2$. We thus obtain a contradiction unless $L_0\leq c_3T$ with $c_3=(2\pi)^{-1}c_2$.

We now consider the case $L_0\leq c_3T$ where $T$ is small. In this case we consider the rectangle 
$R=[-T,T]\times [-3\pi/4,3\pi/4]$. We take the arclength measure to be $0$ on the ends
$\theta=-3\pi/4$ and $\theta=3\pi/4$, and the given arclength measure on the sides $t=\pm T$. We then let $\varphi:R\to D$ be a balanced conformal
diffeomorphism in the sense that $\int_{\p R}\varphi\ d\mu=0$ where $\mu$ is the boundary
measure that we have described. We then have $\int_R|\nabla\varphi|^2\ da=2\pi$, and
by Fubini's theorem there exist $\theta_1\in [\pi/2,5\pi/8]$ and $\theta_2\in [-5\pi/8,-\pi/2]$
such that for $i=1,2$ we have the bound 
\[ \int_{-T}^T|\nabla\varphi|^2(t,\theta_i)\ dt\leq c
\]
for a fixed constant $c$. Let $\alpha_i(t)=\varphi(t,\theta_i)$, and we have the bound
\[ L(\alpha_i)=\int_{-T}^T\left|\frac{\p \varphi(t,\theta_i)}{\p t}\right|\ dt\leq \int_{-T}^T|\nabla\varphi|(t,\theta_i)\ dt
\leq (2T)^{1/2}\left(\int_{-T}^T|\nabla\varphi|^2(t,\theta_i)\ dt\right)^{1/2},
\]
and thus $L(\alpha_i)\leq cT^{1/2}$ for a constant $c$. Now each curve $\alpha_i$ divides
the unit disk into a large region and a small region. Because of the balancing condition and the 
fact that most of the arclength is in the subset of $R$ with $|\theta|\leq \pi/2$, it follows that
the subrectangles $\theta_1\leq \theta\leq 3\pi/4$ and $-3\pi/4\leq \theta\leq \theta_2$ map
into the small region. By the isoperimetric inequality it follows that the areas of these small
regions are bounded by constants times $T$. Therefore we have the bounds
\[ \int_{[-T,T]\times [\theta_1,3\pi/4]}|\nabla\varphi|^2\ da\leq cT,\ \  \int_{[-T,T]\times [-3\pi/4,\theta_2]}|\nabla\varphi|^2\ da\leq cT.
\]
Another application of Fubini's theorem now gives a $\theta_3\in [5\pi/8,3\pi/4]$ and 
$\theta_4\in [-3\pi/4,-5\pi/8]$ with
\[  \int_{-T}^T|\nabla\varphi|^2(t,\theta_i)\ dt\leq cT
\]
for $i=3,4$. We can now define a Lipschitz map $\hat{\varphi}:M\to D$ on $[-T,T]\times [-\pi,\pi]=M$ such that $\hat{\varphi}=\varphi$ for $\theta_4\leq t\leq \theta_3$, $\hat{\varphi}(t,-\pi)=0$, and such that $\hat{\varphi}$ on the intervals $[\theta_3,\pi]$ and 
$[-\pi,\theta_4]$ changes linearly in the $\theta$ variable from $0$ to $\varphi$. By the above bounds, the energy of $\hat{\varphi}$ in these regions is bounded by a constant times $T$, and therefore
\[ \int_M|\nabla\hat{\varphi}|^2\ da\leq 2\pi+cT.
\] 
From the fact that $\hat{\varphi}$ is bounded and the bounds on the arclength in the region
on which $\varphi$ was modified, we see
\[ \left|\int_{\p M}\hat{\varphi}\ ds\right|\leq cT, \quad \left|\int_{\p M}|\hat{\varphi}|^2\ ds-1\right|\leq cT.
\]
We then use the components of $\hat{\varphi}-\int_{\p M}\hat{\varphi}\ ds$ as comparison functions for $\sigma_1(g)$, and we have
\[ \sigma_1(g)\left(\int_{\p M}|\hat{\varphi}|^2\ ds-\left|\int_{\p M}\hat{\varphi}\ ds\right|^2\right)\leq 
\int_M|\nabla\hat{\varphi}|^2\ da.
\]
Therefore we conclude that $\sigma_1(g)\leq 2\pi+cT$ which is a contradiction if $T$ is small.

In case $M$ is a M\"obius band, we let $[-T,T]\times S^1$ be the oriented double covering so
that $M$ is the quotient gotten by identifying $(t,\theta)$ with $(-t,\theta+\pi)$. In this case 
$L_1=L_2=L(\p M)$ which we normalize to be $1$. The arguments follow very much along
the lines above for the annulus. To get an upper bound on $T$ we take a conformal map $\varphi$  from $[T/2,T]\times S^1$ to a concentric annulus in the plane $D\setminus D_r$ noting that 
$r=e^{-T/2}$. We then balance $\varphi$ using a conformal automorphism of $D$
making $\int_{\{T\}\times S^1}\varphi\ ds=0$. This balanced $\varphi$ takes the circle
$\{3T/4\}\times S^1$ to a circle with radius at most $e^{-T/4}$. In particular it follows that the
Dirichlet integral of $\varphi$ between $T/2$ and $3T/4$ is bounded by a constant times
$e^{-T/2}$ (the area enclosed by the circle). Thus by Fubini's theorem we can find 
$T_1\in [T/2,3T/4]$ such that 
\[ \int_{\{T_1\}\times S^1}|\nabla\varphi|^2\ d\theta\leq cT^{-1}
\] 
where we have made a coarse estimate. We can then extend $\varphi$ to a map 
$\hat{\varphi}:[0,T]\times S^1$ by linearly interpolating in the $t$ variable from $\varphi$
to $0$ on the interval $[0,T_1]$ with the properties that $\hat{\varphi}(0,\theta)=0$, $\int_{\{T\}\times S^1}|\hat{\varphi}|^2\ ds=1$,  $\int_{\{T\}\times S^1}\hat{\varphi}\ ds=0$, and
\[ \int_{[0,T]\times S^1}|\nabla\hat{\varphi}|^2\ dtd\theta\leq 2\pi+cT^{-1}.
\]
This gives the upper bound since we may extend $\hat{\varphi}$ to $[-T,T]\times S^1$ so that
it is invariant under the identification hence defined on $M$.

To get the lower bound we use essentially the same argument as for the annulus. We consider
the measure $\hat{\lambda}\ d\theta$ as above. Because of the identification we have that
$\hat{\lambda}(\theta)=\hat{\lambda}(\theta+\pi)$ and the integral of $\hat{\lambda}$ is
equal to $2$. We take the weak limit of these measures for a sequence $T_i\to 0$,
and observe that if the weak limit is not a pair of point masses we have $\sigma_1$ tending
to $0$, while if it is a pair of antipodal point masses then $\sigma$ is bounded above by 
$2\pi$ plus a term which tends to $0$. In either case we get a contradiction and we have
finished the proof.
\end{proof}

We now address the general case and prove the following theorem.
\begin{theorem} \label{k_comp} Let $M$ be a smooth oriented surface of genus $0$ with $k\geq 2$ boundary components. If $\lambda>\sigma^*(0,k-1)$, then there is a $\delta>0$ depending on 
$\lambda$ such that if $g$ is a smooth metric on $M$ with $\sigma_1L\geq\lambda$ then the
injectivity radius of $g_0$ is at least $\delta$ where $g_0$ is the unique constant curvature metric
which parametrizes the conformal class of $g$.
\end{theorem}       
\begin{proof} We have done the case $k=2$ in Proposition \ref{two_comp}, so we now treat
the cases $k\geq 3$. We will prove the theorem by contradiction assuming that we have
a sequence of hyperbolic metrics $g_{0,i}$ on the doubled surface and a sequence of
metrics $g_i$ conformal to $g_{0,i}$ and with $L_{g_i}(\p M)=1$ and $\sigma_1(g_i)\geq \lambda$.
We assume that the injectivity radius $\delta_i$ of $g_{0,i}$ on the doubled surface tends to $0$,
and we take a subsequence so that the metrics $g_{0,i}$ converge in $C^2$ norm on compact
subsets to a complete hyperbolic metric $g_0$ on a surface $\tilde{M}_0$ with finite area. The surface
$\tilde{M}_0$ is the limit of the surfaces $(\tilde{M},g_{0,i})$ after a finite collection of disjoint
simple closed curves have been pinched to curves of $0$ length. The surface $\tilde{M}_0$
is the double of a surface $M_0$ which is a (possibly disconnected) hyperbolic surface with finite area and (possibly noncompact) boundary. Each curve which is pinched corresponds to two
ends of $\tilde{M}_0$. There are two possibilities depending on whether a pinched curve lies in $M$
or crosses the boundary of $M$. The surface $M_0$ is disconnected if and only if there is a
pinched curve in $M$ which is not a boundary component. The hyperbolic metric on an end
corresponding to the pinching of a curve in $M$ is given in suitable coordinates on $(t_0,\infty)\times S^1$ by $dt^2+e^{-2t}d\theta^2$. We refer to such an end as a {\it cusp}. If the pinched curve $\gamma$ crosses $\p M$, then it must be invariant under the reflection of $\tilde{M}$ across 
$\p M$, and it is either a component 
of $\p M$ or the reflection fixes two points of $\gamma$, and thus $\gamma$ intersects exactly
two boundary components. In this case the metric on the end associated with pinching $\gamma$
is the same except that it is defined on $(t_0,\infty)\times (0,\pi)$ where $\theta=0$ and $\theta=\pi$
are portions of two separate components of $\p M_0$. We refer to such an end as a {\it half-cusp}.

We let $M_i=(M,g_{0,i})$ and we recall that near a geodesic $\gamma$ which is being pinched
the metric $g_{0,i}$ in a neighborhood of $\gamma$ is given by 
$g_{0,i}=dt^2+(e^{-t}+\epsilon e^{t})^2d\theta^2$ in suitable coordinates $(t,\theta)$ on an
interval of $\mathbb R\times S^1$. Here $\epsilon=\epsilon_i$ is determined
by the condition that $4\pi\sqrt{\epsilon}=L_i(\gamma)$ and this representation of $g_{0,i}$ is valid
for $|t-t^*|\leq t^*-c$ where $t^*=1/2\log(1/\epsilon)$ is the value of $t$ which represents
$\gamma$ and $c$ is a fixed constant. Thus as the length of $\gamma$ goes to $0$ the metric
converges to the cusp metric $dt^2+e^{-2t}d\theta^2$. For a large $T$, we let $M_{i,T}$ be the subset of $M_i$ with boundary equal to the set $t=T$ on each end. We then have $M_{i,T}$
converging in $C^2$ norm to $M_{0,T}$, the corresponding subset of $M_0$.

We normalize $L_{g_i}(\p M)=1$ and by extracting a subsequence we may assume that the
arclength measures converge to a limiting measure $\mu$ on compact subsets of $\p M_0$.
The outline of the proof of the theorem is as follows: 
For any connected component $M_0'$ of $M_0$ we show that $\mu(\p M_0')$ is either $0$ or $1$.
We then show that there is a unique component $M_0'$ of $M_0$ such that $\mu(\p M_0')=1$. We show that $M_0'$ has no cusps. Finally we show that $M_0'$ has no half-cusps, and is therefore all of $M_0$, and $M_0$ is compact, completing the proof.
\\

\noindent
{\em Step 1.} For any connected component $M_0'$ of $M_0$ we show that $\mu(\p M_0')$ is either $0$ or $1$. \\

To see this we suppose that $a=\mu(\p M_0')\in (0,1)$. There exists a $T$ such that 
$L_{g_i}(\p M'_{i,T})>a/2$ for $i$ sufficiently large, where $M'_i$ is the part of $M_i$ converging to
$M_0'$. For a large number $R$ and $i$ large enough we
let $u$ be the function which is $1$ on $M'_{i,T}$, decays linearly as a function of $t$ to $0$ at $t=T+R$, and is extended to $M_i$ to be $0$ outside $M'_{i,T+R}$. We then have the Dirichlet integral $E(u)\leq c R^{-2}$ since the areas are uniformly bounded. On the other hand, given 
$\epsilon>0$ we have the average $\bar{u}=\int_{\p M_i}u\ ds_i\leq a+\epsilon$ for $i$ sufficiently large (where 
$ds_i$ is the arclength measure for $g_i$). It follows
that $\int_{\p M_i}(u-\bar{u})^2\ ds_i\geq (1-a-\epsilon)^2a/2$, and so
\[ \sigma_1(g_i)(1-a-\epsilon)^2a/2\leq \sigma_1(g_i)\int_{\p M_i}(u-\bar{u})^2\ ds_i
\leq E(u)\leq cR^{-2}.
\] 
This is a contradiction for $R$ sufficiently large. 
\\

\noindent
{\em Step 2.} There is a unique component $M_0'$ of $M_0$ such that $\mu(\p M_0')=1$. \\

We have at most one connected component of $M_0$ with positive boundary measure, and
we now show that there is exactly one such component. If this were not the case then from above
we would have $\mu(\p M_0)=0$ and all of the arclength would concentrate at infinity in the
half-cusps. Suppose we have a half-cusp such that for any large $T$, we have 
$\limsup_i L_i(\{T\leq t\leq 2t^*-T\})=a>0$. We then must have $a=1$ since otherwise we can 
show that $\sigma_1(g_i)$ becomes arbitrarily small as in the previous paragraph. Thus we are left with the situation where $\lim_i L_i(\{T\leq t\leq 2t^*-T\})=1$ for some half-cusp end. In this case
we claim that $\limsup_i \sigma_1(g_i)\leq 2\pi$. This follows from the argument used in
Proposition \ref{two_comp}; that is, we extend the arclength measure $ds_i$ to be zero at 
$t=T$ and $t=2t^*-T$ and we take a balanced conformal map $\varphi$ from the
simply connected region $[T,2t^*-T]\times [0,\pi]$ to the unit disk. Since the total measure is nearly $1$ and it is concentrating away from the ends we can argue that the map $\varphi$ is close to constants near the ends. We may then extend it to a map $\hat{\varphi}:M_i\to D$ which is
zero away from the end, which is nearly balanced, has $L^2$ norm near $1$, and Dirichlet
integral near that of $\varphi$ (which is $2\pi$). Using $\hat{\varphi}$ minus its average as a test function we then can show that $\sigma_1(g_i)$ is bounded above by a number arbitrarily close
to $2\pi$, a contradiction. We have shown that there is a unique component $M_0'$ of $M_0$ such that $\mu(\p M_0')=1$. 
\\

\noindent
{\em Step 3.} We now show that $M_0'$ has no cusps. \\

Suppose $M_0'$ has a cusp, and consider the curve $\gamma$ which is being pinched to form the cusp. For
$i$ sufficiently large we modify $M_i$ by cutting along $\gamma$. If $\gamma$ is an interior
curve of $M$, then we consider the connected component of $M\setminus \gamma$ which
contains most of the boundary measure. We now fill in the curve $\gamma$ with a disk of
circumference $L_i(\gamma)$ and extend the metric $g_{0,i}$ to form a surface $\hat{M}_i$
of genus $0$ with $l$ boundary components where $l<k$. Whether $\gamma$ is an interior
or boundary curve we have $L_i(\p \hat{M}_i)\to 1$ as $i$ tends to infinity. For a small $r>0$ we consider the surface $\hat{M}_i\setminus D_r$ where $D_r$ is a disk in a constant curvature metric concentric with the disk we added to form $\hat{M}_i$. We let $\sigma_1(i,r)$ denote the
infimum of the Steklov Rayleigh quotient for $\hat{M}_i\setminus D_r$ with competitors of zero
average over $\p \hat{M}_i$ with respect to $ds_i$ and which
are $0$ on $\p D_r$. Since for $i$ large, any such function can be extended to $M_i$ to be $0$ 
outside $\p D_r$, we see that $\sigma_1(i,r)\geq \sigma_1(g_i)$ for any $r>0$ and $i$ sufficiently
large. We will show that $\sigma_1(i,r)\leq \sigma^*(0,k-1)+\epsilon$ for any given $\epsilon>0$
if $r$ is chosen small enough and $i$ is chosen large enough. For $\epsilon$ small this
contradicts the assumption $\sigma_1(g_i)\geq \lambda>\sigma^*(0,k-1)$.

First, fix $i$ large and $r>0$ small.
As observed above, there exists $\rho(i)$ such that for $\rho(i) \leq \rho \leq r$ we have
$\sigma_1(i, \rho) \geq \sigma_1(g_i) \geq \lambda$.
We will use this lower bound on $\sigma_1(i,\rho)$ to get the quantitative lower bound
$\sigma_1(i,r)\leq \sigma_1(i,\rho)+c|\log r|^{-1/2}$.
Let $u_{i,r}$ be a first eigenfunction of $\hat{M}_i\setminus D_r$ (with eigenvalue $\sigma_1(i,r)$) normalized to have $L^2$ norm $1$ on $\p \hat{M}_i$. We claim that for any $T>T(\lambda)$, $u_{i,\rho}$ is pointwise uniformly bounded on compact subsets of  $\hat{M}_{i,T}\setminus\p \hat{M}_{i,T}$ for all $\rho(i)\leq \rho\leq r$. In particular, $u_{i,\rho}$ is uniformly bounded near the disk we added to form
$\hat{M}_i$. 
To see this bound, observe that on any component $\Omega$ of the complement of the zero set of $u_{i,\rho}$, the function $u_{i,\rho}$ is a {\it first} eigenfunction for the eigenvalue problem in $\Omega$ with Dirichlet boundary 
condition on $\p \Omega\cap \hat{M}_i$ and Steklov boundary condition on 
$\Omega\cap\p \hat{M}_i$. In other words the function $u_{i,\rho}$ on $\Omega$ minimizes
the Rayleigh quotient
\[ \frac{\int_\Omega|\nabla v|^2}{\int_{(\p \hat{M}_i)\cap \bar{\Omega}}v^2\ ds_i}
\]
over functions $v$ which are $0$ on $\p \Omega\cap\hat{M}_i$. Moreover the minimum
value of this ratio is $\sigma_1(i,\rho)$. This follows directly
from the variational characterization of $\sigma_1(\hat{M}_i)$. From here we see that
no such component can be too large. First, for $T$ large enough, the set $\hat{M}_{i,T}$
cannot be contained in such an $\Omega$. This follows from the fact that for $T$ large,
the length $L_i(\p\hat{M}_{i,T/2})$ is greater than $1/2$ (in fact is near $1$). We can then
choose a function $v$ which is $1$ on $\hat{M}_{i,T/2}$ and zero outside $\hat{M}_{i,T}$
with Dirichlet integral bounded by $cT^{-2}$. This contradicts the fact that 
$\sigma_1(i,\rho) \geq \lambda$. It follows that there are points in 
$\p\hat{M}_{i,T}$ at which $u_{i,\rho}=0$. If we can find a point a fixed distance from 
$\p\hat{M}_{i,T}$ at which $u_{i,\rho}$ is bounded, the result then follows from bounds on
harmonic functions with bounded Dirichlet integral. To find such a point we assert that
there is a point $p\in \p\hat{M}_{i,T}$ and a fixed $\delta>0$ with the property that each half
circle of radius $\tau$ centered at $p$ for $0<\tau\leq\delta$ intersects the zero set of $u_{i,\rho}$.
Indeed if this were not true it would follow that each component of the zero set of $u_{i,\rho}$
is contained in a small disk $D_\delta(p)$ centered at a boundary point. Since the zero set divides
$\hat{M}_i \setminus D_r$ into two components (Courant nodal domain theorem), it follows that there can be at most one such disk, and therefore there is a component $\Omega$ of the complement of the zero set which contains $\hat{M}_{i,T}\setminus D_\delta(p)$. Since for $\delta$ small the arclength of the
interval of radius $\delta$ about $p$ on the boundary is small,
we can construct a function which is $1$ away from $p$, vanishes in $D_\delta(p)$
and has small Rayleigh quotient, contradicting the lower bound on $\sigma_1(i,\rho)$. Therefore
$u_{i,\rho}$ is locally bounded in $\hat{M}_{i,T}\setminus\p \hat{M}_{i,T}$ as claimed.

We now choose a Lipschitz radial function $\zeta$ which is $1$ outside $D_{\sqrt{r}}$ and $0$ inside $D_r$ with Dirichlet integral $E(\zeta)\leq c|\log r|^{-1}$. We then let $v=\zeta u_{i,\rho}$
and estimate the Dirichlet integral
\[ E(v)\leq (1+\epsilon)E(u_{i,\rho})+c(1+\epsilon^{-1})E(\zeta)
\]
for any $\epsilon>0$ where we have used the arithmetic-geometric mean inequality and the bound on $u_{i,\rho}$. We take $\epsilon=|\log r|^{-1/2}$ and conclude 
$E(v)\leq \sigma_1(i,\rho)+c|\log r|^{-1/2}$, and this implies the bound
$\sigma_1(i,r)\leq \sigma_1(i,\rho)+c|\log r|^{-1/2}$ for all $\rho(i) \leq \rho \leq r$.

Now we claim that given $\epsilon>0$, if $i$ is sufficiently large, then $\sigma_1(i,\rho(i)) \leq \sigma_1(M_0')+\epsilon$. This follows since as $i \rightarrow \infty$, $u_{i,\rho(i)}$ converges on compact subsets to a function $u$ on $M_0'\setminus\{p\}$, where $p$ is the cusp point. Since $u$ is a bounded harmonic function, it extends to $M_0'$, and must be a first eigenfunction. Finally, since $\sigma_1(M_0') \leq \sigma^*(0,k-1)$, we obtain the desired bound $\sigma_1(i,r)\leq \sigma_1(i,\rho(i))+c|\log r|^{-1/2}\leq \sigma^*(0,k-1) + \epsilon$ for any given $\epsilon>0$ if $r$ is chosen small enough and $i$ is chosen large enough. Thus we have shown that $M_0$ has no cusps.
\\

\noindent
{\em Step 4.} $M_0'$ has no half-cusps. \\

Finally we finish the proof by showing that $M_0'$ has no half-cusps and is therefore all of $M_0$, and $M_0$ is compact. 
Assume we have a half-cusp in $M_0'$ and consider the surfaces $M_i$ converging to $M_0$.
Note that we have two half-cusps for each pinched curve, one on each side of the pinched
curve. For $i$ large, let $(t,\theta)$ be the coordinates described above on the half-cusp. We have 
shown that the measure concentrates away from the neck, so for any $\epsilon>0$ we may choose $T_0$ large enough that $L_i(\{t\geq T_0\})<\epsilon$. We form a surface $M_1$ with at most $k-1$
boundary components by making a cut along the $t=t^*,\ 0\leq\theta\leq\pi$ curve. We take the
boundary measure $\mu$ on $\p M_1$ to be the original measure $ds_i$ on $\{t\leq T_0\}$ and $0$ on the remainder of the boundary. Thus we have $\mu(\p M_1)\geq 1-\epsilon$. We let 
$\sigma_1(M_1)$ be the corresponding first eigenvalue and let $u_1$ be a first eigenfunction
normalized to have boundary $L^2$-norm equal to $1$ with respect to $\mu$. From above
we know that $u_1$ is bounded by a constant depending on $T_0$ for points
of $\{t=T_0\}$ away from the boundary points. By Fubini's theorem we may then choose 
$T_1\in [T_0,T_0+1]$ so that $\int_{\{t=T_1\}}|\nabla u_1|^2 ds\leq c$, and therefore we have
$|u_1(T_1,\theta)|\leq c$ for $0\leq \theta\leq\pi$. Since $u_1$ satisfies the homogeneous
Neumann boundary condition for $t\geq T_1$, it follows from the maximum principle that
$|u_1|\leq c$ on all of $M_1$. For $T>T_1$ to be chosen sufficiently large we choose a Lipschitz function $\zeta(t)$ which is $1$ for $t\leq T$, $0$ for $t\geq T+1$, and with $|\zeta'|\leq 1$. We then set $v=\zeta u_1$ and we extend $v$ to $M_i$ to be $0$ for $t\geq T+1$ (on both sides of the neck being pinched). Using the bound on $u_1$ we may estimate the Dirchlet integral of $v$ on
$M_i$ by $E(v)\leq (1+\epsilon_1)E(u_1)+c(1+\epsilon_1^{-1})e^{-T}$ for any $\epsilon_1>0$.
Thus we may choose $T$ sufficiently large so that $E(v)\leq E(u_1)+\epsilon$ for the $\epsilon$
chosen above. We let $\bar{v}=\int_{\p M}v\ ds_i$, and we have
\[ \sigma_1(M_i)\left(\int_{\p M}v^2\ ds_i-\bar{v}^2\right)\leq E(v)\leq E(u_1)+\epsilon=\sigma_1(M_1)+\epsilon.
\]
From the choice of $\mu$ and the normalization of $u_1$ we have by the Schwarz inequality
\[ \int_{\p M}v^2\ ds_i=1+\int_{\{t\geq T_0\}}v^2\ ds_i\geq 1+\left(\int_{\{t\geq T_0\}}v\ ds_i\right)^2.
\]
On the other hand we have $\bar{v}=\int_{\{t\geq T_0\}}v\ ds_i$ since $v=u_1$ integrates to $0$
on $\{t\leq T_0\}$. Therefore we have 
\[ \sigma_1(M_i)\leq \sigma_1(M_1)+\epsilon\leq \mu(\p M_1)^{-1}\sigma^*(0,k-1)+\epsilon
\]
since $M_1$ has at most $k-1$ boundary components. This implies
\[ \sigma_1(M_i)\leq (1-\epsilon)^{-1}\sigma^*(0,k-1)+\epsilon,
\]
and this is a contradiction for $\epsilon$ small. We have shown that $M_0$ is compact,
and this is a contradiction to our assumption that the injectivity radius of $g_{0,i}$ was converging
to $0$. We have completed the proof of Theorem \ref{k_comp}.
\end{proof}

\section{The existence and regularity of extremal metrics} \label{section:extremal}

In this section we prove existence of smooth metrics which maximize $\sigma_1 L$ on surfaces provided the conformal structure is bounded for any metric near the maximum. In particular we prove existence of smooth maximizing metrics on oriented surfaces of genus $0$ with $k \geq 2$ boundary components. We also prove existence and regularity of such metrics on the M\"obius band.
 
We first establish the connection between extremal metrics and minimal surfaces that are free boundary solutions in the ball. 
Let $M$ be a compact surface with boundary, and assume $g_0$ is a metric on $M$ with
\[
      \sigma_1(g_0)L_{g_0}(\p M)=\max_g \sigma_1(g) L_g(\p M)
\] 
where the max is over all smooth metrics on $M$. 
Let $g(t)$ be a family of smooth metrics on $M$ with $g(0)=g_0$ and $\frac{d}{dt}g(t) =h(t)$, where  
$h(t)\in S^2(M)$ is a smooth family of symmetric $(0,2)$-tensor fields on $M$.
Denote by $V_{g(t)}$ the eigenspace associated to the first nonzero Steklov eigenvalue $\sigma_1(t)$ of $(M,g(t))$.
Define a quadratic form $Q_{h}$ on smooth functions $u$ on $M$ as follows
\[
      Q_{h}(u)=-\int_M \la \tau(u) , h \ra \; da_t 
         -\frac{\sigma_1(t)}{2} \int_{\p M} u^2 h(T,T) \; ds_t,
\]
where $T$ is the unit tangent to $\p M$ for the metric $g(t)$, and
where $\tau(u)$ is the stress-energy tensor of $u$ with respect to the metric $g(t)$,
\[
    \tau(u)= du \otimes du -\frac{1}{2} |\n u|^2 g.
\]

\begin{lemma}
$\sigma_1(t)$ is a Lipschitz function of $t$, and if $\dot{\sigma}_1(t_0)$ exists, then
\[
       \dot{\sigma}_1(t_0) = Q_{h}(u)
\] 
for any $u \in V_{g(t_0)}$ with $||u||_{L^2(\p M)}=1$.
\end{lemma}
\begin{proof}
To see that $\sigma_1(t)$ is Lipschitz for small t, let $t_1\neq t_2$ and assume without
loss of generality that $\sigma_1(t_1)\leq\sigma_1(t_2)$. Now let $u$ be a first Steklov
eigenfunction for $g(t_1)$ normalized so that $\int_{\partial M}u^2\ ds_{t_1}=1$. It then follows
easily from the fact that the path $g(t)$ is smooth that
\[ \left|\int_M|\nabla^{t_1}u|^2\ da_{t_1}-\int_M|\nabla^{t_2}u|^2\ da_{t_2}\right|\leq c|t_1-t_2|
\]
and
\[ \left|\int_{\partial M}u^2\ ds_{t_1}-\int_{\partial M}u^2\ ds_{t_2}\right|\leq c|t_1-t_2|,\ \ 
|\bar{u}(t_1)-\bar{u}(t_2)|\leq c|t_1-t_2|
\]
where $\bar{u}(t)$ denotes the average of $u$ over $\partial M$ with respect to the metric
$g(t)$. Therefore we have
\[ |\sigma_1(t_1)-\sigma_1(t_2)|=\sigma_1(t_2)-\sigma_1(t_1)\leq \frac{\int_M|\nabla^{t_2}u|^2\ da_{t_2}}{\int_{\partial M}(u-\bar{u}(t_2))^2\ ds_{t_2}}-\int_M|\nabla^{t_1}u|^2\ da_{t_1}\leq c|t_1-t_2|
\]
and $\sigma_1(t)$ is Lipschitz.

Choose $u_0 \in V_{g(t_0)}$ and a family of functions $u(t)$ such that $u(t_0)=u_0$ and 
$\int_{\p M} u(t) \; ds_t=0$; e.g. $u(t)=u_0 - \frac{\int_{\p M} u_0 \; ds_t}{\int_{\p M} \, ds_t}$.
Let
\[
       F(t)=\int_M |\n u(t)|^2 \; da_t -\sigma_1(t) \int_{\p M} u^2(t) \; ds_t.
\]
Then $F(t) \geq 0$, and $F(t_0)=0$, and we have $\dot{F}(t_0)=0$.
Differentiating $F$ with respect to $t$ at $t=t_0$ we therefore obtain
\begin{align*}
    \int_M [\, 2 & \la \n u_0 , \n \dot{u}_0 \ra-\la  du_0 \otimes du_0 
    -\frac{1}{2} |\n u_0|^2 g , h \ra \,] \; da_{t_0} \\
    & = \dot{\sigma_1}(t_0) \int_{\p M}  u_0^2 \; ds_{t_0} 
    + \sigma_1(t_0) \int_{\p M} [\,2u_0  \dot{u}_0 + \frac{1}{2}u_0^2 h(T,T) \,] \; ds_{t_0}.
\end{align*}
Since $u_0$ is a first Steklov eigenfunction, we have
\[
     \int_M \la \n u_0, \n \dot{u}_0 \ra \; da_{t_0}  = \sigma_1(t_0) \int_{\p M} u_0 \, \dot u_0    \; ds_{t_0}.    
\]
Using this, and if we normalize $u_0$ so that $||u_0||_{L^2(\p M)}=1$, we have
\[
      \dot{\sigma}_1(t_0)=-\int_M \la du_0 \otimes du_0 -\frac{1}{2} |\n u_0|^2 g , h \ra \; da_{t_0} 
         -\frac{\sigma_1(t_0)}{2} \int_{\p M} u_0^2 h(T,T) \; ds_{t_0} = Q_{h}(u_0).
\]
\end{proof}
\begin{proposition} \label{prop:extremal}
If $M$ is a surface with boundary, and $g_0$ is a metric on $M$ with
\[
      \sigma_1(g_0)L_{g_0}(\p M)=\max_g \sigma_1(g) L_g(\p M)
\] 
where the max is over all smooth metrics on $M$. Then there exist independent first eigenfunctions $u_1, \ldots, u_n$ which give a branched conformal minimal immersion 
$u=(u_1, \ldots, u_n)$ of $M$ into the unit ball $B^n$ such that u(M) is a free boundary solution, and up to rescaling of the metric $u$ is an isometry on $\p M$.
\end{proposition}
The outline of the proof of the proposition follows an argument of Nadirashvili (\cite{N}, Theorem 5) and El Soufi and Ilias (\cite{EI2}, Theorem 1.1).

Consider the Hilbert space $\mathcal{H}=L^2(S^2(M)) \times L^2(\p M)$, the space of pairs of $L^2$ symmetric $(0,2)$-tensor fields on $M$ and $L^2$ functions on boundary of $M$.

\begin{lemma} \label{lemma:indefinite}
Assume $g_0$ satisfies the conditions of Proposition \ref{prop:extremal}. For any $(\omega,f) \in \mathcal{H}$ with $\int_{\p M} f \;ds=0$
there exists $u \in V_{g_0}$ with $||u||_{L^2(\p M)}=1$ such that $\la (\omega,\frac{\sigma_1(g_0)}{2}f), (\tau(u),u^2) \ra_{L^2}=0$.
\end{lemma}
\begin{proof}
Let $(\omega,f) \in \mathcal{H}$, and assume that $\int_{\p M} f \; ds=0$. 
Since $C^\infty(S^2(M)) \times C^\infty(M)$ is dense in $L^2(S^2(M)) \times L^2(M))$, we can approximate $(\omega, f)$ arbitrarily closely in $L^2$ by a smooth pair $(h,\tilde{f})$ with $\int_{\p M} \tilde{f} \; ds =0$. We may redefine $h$ in a neighborhood of the boundary to a smooth tensor whose restriction to $\p M$ is equal to the function $\tilde{f}$, and such that the change in the $L^2$ norm is arbitrarily small. In this way, we obtain a smooth sequence $h_i$ with $\int_{\p M} h_i(T,T) \; ds=0$, such that $(h_i, h_i(T,T)) \rightarrow (\omega,f)$ in $L^2$.

Let $g(t)=\frac{L_{g_0}(\p M)}{L_{g_0+th_i}(\p M)} (g_0+ th_i)$. 
Then $g(0)=g_0$, $L_{g(t)}(\p M)=L_{g_0}(\p M)$, and since
\[
      \frac{d}{dt}\Big|_{t=0} L_{g_0+th_i}(\p M)=\int_{\p M} h_i(T,T) \; ds=0
\]
we have $\frac{dg}{dt}\big|_{t=0} =h_i$.
Given any $\varepsilon >0$, by the fundamental theorem of calculus, 
\[
     \int_{-\varepsilon}^0 \dot{\sigma}_1(t) \; dt = \sigma_1(0) - \sigma_1(-\varepsilon) \geq 0
\]
by the assumption on $g_0$. Therefore there exists $t$, $-\varepsilon < t < 0$, such that 
$\dot{\sigma}_1(t)$ exists and $\dot{\sigma}_1(t) \geq 0$. Let $t_j$ be a sequence of points with $t_j<0$ and $t_j \rightarrow 0$, such that $\dot{\sigma}_1(t_j) \geq 0$. Choose $u_j \in V_{g(t_j)}$ with 
$||u_j||_{L^2(\p M)}=1$. Elliptic boundary value estimates (\cite{M}) give bounds on $u_j$ and its derivatives up to the boundary, thus after passing to a subsequence $u_j$ converges in $C^2(M)$ to an eigenfunction $u^{(i)}_- \in V_{g_0}$ with $||u^{(i)}_-||_{L^2(\p M)}=1$.
Since $Q_{h_i}(u_j)=\dot{\sigma}_1(t_j) \geq 0$, it follows that $Q_{h_i}(u^{(i)}_-) \geq 0$. By a similar argument, taking a limit from the right, there exists $u^{(i)}_+ \in V_{g_0}$ with $||u^{(i)}_+||_{L^2(\p M)}=1$, such that $Q_{h_i}(u^{(i)}_+) \leq 0$.

As above, after passing to subsequences, $u^{(i)}_+ \rightarrow u_+$ and $u^{(i)}_- \rightarrow u_-$ in $C^2(M)$, and
\begin{align*}
     -\la (\omega,\frac{\sigma_1(g_0)}{2}f), (\tau(u_+),u_+^2) \ra_{L^2} 
     &= \lim_{i \rightarrow \infty} Q_{h_i}(u^{(i)}_+)\leq 0 \\
     -\la (\omega,\frac{\sigma_1(g_0)}{2}f), (\tau(u_-),u_-^2) \ra_{L^2}  
     &= \lim_{i \rightarrow \infty} Q_{h_i}(u^{(i)}_-)\geq 0.
\end{align*}
\end{proof}
\begin{proof}[Proof of Proposition \ref{prop:extremal}]
Without loss of generality, rescale the metric $g_0$ so that $\sigma_1(g_0)=1$.
Let $K$ be the convex hull in $\mathcal{H}$ of 
\[
        \{ \, (\tau(u), u^2) \, : \, u \in V_{g_0} \}.
\]        
We claim that $(0,1) \in K$.
If $(0,1) \notin K$, then since $K$ is a convex cone which lies in a finite dimensional subspace, the Hahn-Banach theorem implies the existence of $(\omega,f) \in \mathcal{H}$ that separates $(0,1)$ from $K$; in particular such that
\begin{align*}
      \la (\omega,\frac{1}{2}f), (0,1) \ra_{L^2} & >0, \quad \mbox{and} \\
      \la (\omega,\frac{1}{2}f), (\tau(u),u^2) \ra_{L^2} & <0 \quad \mbox{ for all } u \in V_{g_0} \setminus \{0\}.
\end{align*}
Let $\tilde{f}=f-\frac{\int_{\p M} f \; ds}{L_{g_0}(\p M)}$. Then, $\int_{\p M} \tilde{f} \;ds=0$, and
\begin{align*}
      \la (\omega, \frac{1}{2}\tilde{f}), (\tau(u),u^2) \ra_{L^2}
      &= \int_M \la \omega, \tau(u) \ra \; da + \frac{1}{2}\int_{\p M} \tilde{f} u^2 \; ds \\
      &=  \int_M \la \omega, \tau(u) \ra \; da + \frac{1}{2}\int_{\p M} f u^2 \; ds 
            - \frac{\int_{\p M} f}{2L_{g_0} (\p M)} \, \int_{\p M} u^2 \; ds \\
      &= \la (\omega,\frac{1}{2}f), (\tau(u),u^2) \ra_{L^2} 
          -\frac{\int_{\p M} u^2 \; ds}{L_{g_0}(\p M)} \; \la (\omega,\frac{1}{2}f),(0,1) \ra_{L^2} \\
      &< 0.
\end{align*}
This contradicts Lemma \ref{lemma:indefinite}. Therefore, $(0,1) \in K$, and since $K$ is contained in a finite dimensional subspace, there exist independent eigenfunctions 
$u_1, \ldots, u_n \in V_{g_0}$ such that
\begin{align*}
     0 & = \sum_{i=1}^n \tau(u_i) 
     = \sum_{i=1}^n (du_i \otimes du_i -\frac{1}{2} |\n u_i|^2 g_0) \quad \mbox{ on } M \\
     1& = \sum_{i=1}^n u_i^2  \quad \mbox{ on } \p M
\end{align*}
Thus $u=(u_1, \ldots, u_n): M \rightarrow B^n$ is a conformal minimal immersion. Since $u_i$ is a first Steklov eigenfunction and $\sigma_1(g_0)=1$, we have $\fder{u_i}{\eta}=u_i$ on $\p M$ for $i=1, \ldots, n$. Therefore, 
\[
      \left|\fder{u}{\eta}\right|^2= |u|^2=1 \qquad \mbox{on } \p M,
\]
and since $u$ is conformal, $u$ is an isometry on $\p M$.      
\end{proof}

We now consider the question of the existence and regularity of maximizing metrics. The strategy
for obtaining such results involves first proving bounds on the conformal structure and boundary
lengths for metrics which are nearly maximizing. We succeeded in doing this in the special case 
of genus $0$ surfaces in Section \ref{section:properties}. Here we assume that such bounds have been obtained and we proceed to construct a smooth maximizing metric under this assumption. We let $\sigma^*(\gamma,k)$ denote the supremum of $\sigma_1L$ taken over all smooth metrics
on a surface $M$ of genus $\gamma$ with $k\geq 1$ boundary components.

Since Weinstock's theorem handles the case of $\gamma=0$ and $k=1$, we assume that 
$\gamma+k>1$. If we choose a conformal class of metrics on $M$, then there is a unique
(normalized) constant curvature metric $g_0$ in that conformal class with curvature either $K=0$ or $K=-1$ such that the boundary components are geodesics. This may be obtained by
taking a reflection invariant constant curvature metric on the closed surface $\tilde{M}$ gotten by 
doubling $M$. The Euler characteristic of $\tilde{M}$ is $2(2-2\gamma-k)$ (twice that of $M$),
and this is negative except when $\gamma=0$ and $k=2$. Thus the curvature of $g_0$
is $-1$ except in this case, and in this case we normalize $g_0$ to have area $1$ to make it
unique. With this normalization we have a unique constant curvature metric in every conformal
class of metrics on $M$.

We are going to do a canonical smoothing procedure for measures on $\p M$. For $\epsilon>0$
we let $K_\epsilon(x,y)$ for $x,y\in\p M$ denote the heat kernel for the metric (arclength) $ds_0$ induced on $\p M$ by $g_0$. Given a probability measure $\mu$ on $\p M$ we let $\mu_\epsilon$
be the smoothed measure given by
\[ \mu_\epsilon=\left(\int_{\p M}K_\epsilon(x,y)d\mu(y)\right)\ ds_0.
\]
Thus $\mu_\epsilon$ is a probability measure with a smooth density function relative to $ds_0$.

Given $g_0$, $\mu$, and $\epsilon$, we let $\sigma_1(g_0,\mu,\epsilon)$ be the first 
Steklov eigenvalue for a metric $g$ in the conformal class of $g_0$ with boundary measure
given by $\mu_\epsilon$. We have the following result.
\begin{proposition}\label{confclass} Given $g_0$ and $\epsilon>0$, there exists a probability measure $\nu$ on $\p M$ such that
\[ \sigma_1(g_0,\nu,\epsilon)=\sup_\mu \sigma_1(g_0,\mu,\epsilon).
\]
Moreover, we have
\[ \sigma^*(\gamma,k)=\sup_{(g_0,\mu,\epsilon)}\sigma_1(g_0,\mu,\epsilon).
\]
\end{proposition}
\begin{proof} The existence of a maximizer follows from straightforward compactness
arguments since the space of probability measures is compact in the weak* topology and
the map $\mu\to\mu_\epsilon$ is continuous into smooth measures with the $C^l$ topology
for any integer $l$. 

To prove the second assertion observe that if $g$ is any smooth metric with boundary
length $1$, we may let $g_0$ be the associated constant curvature metric and $\mu$ the
arclength measure for $g$. From convergence of the heat equation we have
\[ \lim_{\epsilon\to 0}\sigma_1(g_0,\mu,\epsilon)=\sigma_1(g).
\]
Therefore 
\[ \sigma_1(g)\leq \sup_{(g_0,\mu,\epsilon)}\sigma_1(g_0,\mu,\epsilon),
\]
and the result follows since the reverse inequality is clear by definition.
\end{proof}

We let $\sigma^*(g_0,\epsilon)$ denote the maximum taken over measures $\mu$. In order to vary the conformal class we must understand how things change under diffeomorphisms. To this effect we have the following.
\begin{proposition}\label{diffeo} Given a smooth diffeomorphism $F$ of $M$ we have 
$\sigma_1(F^*g_0,(F^{-1})_\#(\mu),\epsilon)=\sigma_1(g_0,\mu,\epsilon)$. In particular
we have $\sigma^*(F^*g_0,\epsilon)=\sigma^*(g_0,\epsilon)$.
\end{proposition}
\begin{proof} We first claim that $(F^{-1})_\#(\mu_\epsilon)=((F^{-1})_\#\mu)_\epsilon$ for
any measure $\mu$ on $\p M$. This follows from the fact that the heat kernel $\hat{K}$
of the metric $F^*g_0$ is given by $\hat{K}_\epsilon(x,y)=K_\epsilon(Fx,Fy)$. Therefore
\[ \int_{\p M}\hat{K}_\epsilon(x,y)\ d((F^{-1})_\#\mu)(y)=\int_{\p M}K_\epsilon(Fx,\eta)\ d\mu(\eta).
\]
Now we have 
\[ (F^{-1})_\#(\mu_\epsilon)= \int_{\p M}K_\epsilon(Fx,\eta)\ d\mu(\eta)\ F^*(ds_0)
\]
and the result follows since the right hand side is $((F^{-1})_\#\mu)_\epsilon$ from the
previous equation.

Given a function $u$ on $\p M$ with $\int_{\p M}u\ d\mu_\epsilon=0$, we let
$v=u\circ F$, and we observe from above that $\int_{\p M}v\ d[(F^{-1})_\#(\mu)]_\epsilon=0$ where the smoothing is done with respect to $F^*g_0$, and similarly
\[ \int_{\p M}u^2\ d\mu_\epsilon=\int_{\p M}v^2\ d[(F^{-1})_\#(\mu)]_\epsilon.
\]
If $\hat{u}$ is the harmonic extension of $u$ with respect to $g_0$, then $\hat{v}=\hat{u}\circ F$
is the harmonic extension of $v$ with respect to $F^*g_0$, and their Dirichlet integrals coincide.
It follows that the set of Rayleigh quotients which determine $\sigma_1$ is the same for
$(g_0,\mu,\epsilon)$ as for $(F^*g_0,(F^{-1})_\#\mu,\epsilon)$ and the first claim follows.
The second assertion then follows since we are maximizing over the set of measures, and 
$(F^{-1})_\#$ maps this set to itself in a one-to-one and onto fashion.
\end{proof}

We now consider the problem of maximizing over the conformal classes of metrics on $M$. We 
define $\sigma^*(\gamma,k,\epsilon)$  to be the supremum of $\sigma^*(g_0,\epsilon)$ taken
over the metrics $g_0$. We assume here that the geometry of $g_0$ (that is, the conformal class) can be controlled for metrics near the maximum, so that (after applying diffeomorphisms) we can find a maximizing sequence $g_0^{(i)}$ such that the metrics converge in $C^2$ norm. Under this assumption we can find a metric $g_0$ such that
\[ \sigma^*(g_0,\epsilon)=\sigma^*(\gamma,k,\epsilon).
\]

We use the notation $f_\epsilon$ to denote the time $\epsilon$ solution of the heat equation on 
$\p M$ with initial data $f$.  We will now derive the properties of a maximizing measure with 
$\epsilon$ and the conformal class $g_0$ fixed. 
\begin{proposition}\label{epsilon.extremal} Assume $g_0$ is a fixed constant curvature metric and
$\epsilon>0$ is fixed.  Let $\nu$ be a maximizing measure for $g_0$ and $\epsilon$; that is,
$\sigma_1(g_0,\nu,\epsilon)=\sigma^*(g_0,\epsilon)$. There are independent first eigenfunctions $u_1,\ldots, u_n$ such that the map $u=(u_1,\ldots, u_n)$
is a harmonic map into $\mathbb R^n$ with the following properties $(|u|^2)_\epsilon=1$ and $\nabla_T(|u|^2)_\epsilon=0$ for $\nu$-almost all points of $\p M$ where $T$ is the unit tangent vector with respect to $g_0$.

Furthermore, if $g_0$ is chosen so that $ \sigma^*(g_0,\epsilon)=\sigma^*(\gamma,k,\epsilon)$,
then in addition to the above conditions on the map $u$, we may assume that for any
$h$ of compact support the function 
\[ F(t)=\int_M|\nabla_t u|^2\ da_t-\sigma_1(g_0,\epsilon)\int_{\p M}|u|^2\ d(\nu)_{t,\epsilon}
\]
has a critical point at $t=0$ where we have $g_t=g+th$, and the quantities are computed
with respect to $g_t$. Note that $(\nu)_{t,\epsilon}$ denotes the time $\epsilon$ solution of the heat equation with initial data $\nu$ for the canonical metric in the conformal class of $g_t$.
\end{proposition}
\begin{proof} Let $F_t$ be a one parameter family of diffeomorphisms of $\p M$, and we vary the measure $\nu$ by setting $\nu_t=(F_t)_\#(f_t\nu)$ where $f_t>0$ with $f_0=1$ and such
that each $\nu_t$ is a probability measure.

For a family $u_t$ of eigenfunctions for $(g_0,\nu_t)$, we may compute the first variation of 
$\sigma_1(t)$ by differentiating the expression
\[ \int_M|\nabla u_t|^2\ da_0-\sigma_1(t)\int_{\p M}u_t^2\ d(\nu_t)_\epsilon=0.
\]
 Note that the boundary integral can also be written 
 \[ \int_{\p M}u_t^2\ d(\nu_t)_\epsilon=\int_{\p M}(u_t^2)_\epsilon\ d(F_t)_\#(f_t\nu)=\int_{\p M}(u_t^2)_\epsilon\circ F_t\ d(f_t\nu).
 \]
 Thus we can normalize the $L^2$ norm at $t=0$ to be $1$ and
 we find the formal expression for $\sigma'(0)$ 
 \[ \sigma'(0)=-\sigma_1\int_{\p M}\psi(u_0^2)_\epsilon\ d\nu-\sigma_1\int_{\p M}\phi\nabla_T(u_0^2)_\epsilon\ d\nu
 \]
 where we have $\dot{F}_0=\phi T$ on $\p M$ for a smooth function $\phi$, $T$ the unit tangent vector, and $\psi=\dot{f}_0$. Thus we have the 
expression $\sigma'(0)=Q^{(\epsilon)}_{\phi,\psi}(u_0,u_0)$ where
 \[ Q^{(\epsilon)}_{\phi,\psi}(u,u)=-\sigma_1\int_{\p M}\psi(u^2)_\epsilon\ d\nu-\sigma_1\int_{\p M}\phi\nabla_T(u^2)_\epsilon\ d\nu
 \] 
 is a quadratic form on the eigenspace at $t=0$. We also have the requirement 
 $\int_{\p M}\psi\ d\nu=0$.
 Now the proof is similar to the proof of Proposition \ref{prop:extremal} with the
quadratic form $Q_h$ replaced by $Q^{(\epsilon)}_{\phi,\psi}$. The same Hahn-Banach argument implies that there are independent eigenfunctions $u_1,\ldots, u_n$ such that the map
$u=(u_1,\ldots, u_n)$ satisfies $(|u|^2)_\epsilon=1$ and $\nabla_T(|u|^2)_\epsilon=0$ for
$\nu$ almost all points of $\p M$. 

To prove the last statement we consider the deformation $g_t=g+th$ and $\nu_t$ as above. 
We observe that the corresponding quadratic form may be written
\[ Q^{(\epsilon)}_{\phi,\psi,h}(u,u)=-\int_M\langle \tau(u)+T^*(u^2),h\rangle\ da_0+Q^{(\epsilon)}_{\phi,\psi}(u,u)
\]
where $T^*$ is the formal adjoint of the linear operator $T$ which takes smooth compactly
supported tensors $h$ to functions on $\p M$ and is defined by
\[ \left.\frac{d}{dt}\right |_{t=0}\ \int_{\p M}f\ d(\nu)_{t,\epsilon}=\int_{\p M}f\ T(h)\ ds_0
\]
for any smooth function $f$ on $\p M$. We may then apply the Hahn-Banach
theorem in the space $L^2(S^2(M))\times L^2(\p M)\times L^2(\p M)$ to assert that
the point $(0,1,0)$ is in the (finite dimensional) convex cone over the subset consisting
of $(\tau(u)+T^*(u^2),(u^2)_\epsilon,\nabla_T(u^2)_\epsilon)$ for $u$ in the first
eigenspace at $t=0$. It follows that we can find independent eigenfunctions so that
$u=(u_1,\ldots, u_n)$ satisfies the three conditions $(|u|^2)_\epsilon=\nabla_T(|u^2|)_\epsilon=0$
for $\nu$ almost all points of $\p M$, and $F'(0)=0$. 
\end{proof}
\begin{remark} We expect the support of $\nu$ to be a finite set of points since otherwise
the analyticity of $(|u|^2)_\epsilon$ would imply that $(|u|^2)_\epsilon$ is identically $1$
on any component of $\p M$ on which $\nu$ has infinite support. It would then follow that $|u|^2$ is identically $1$ on that component of $\p M$. We do not expect this
to be true for an $\epsilon$-maximizer in a conformal class.
\end{remark}

We have now shown that for a fixed $g_0$ and any sufficiently small $\epsilon>0$ we can find a measure $\nu^{(\epsilon)}$, and a harmonic map $u^{(\epsilon)}$ 
from $M$ to $\mathbb R^{n_\epsilon}$ consisting of independent eigenfunctions such that 
$(|u^{(\epsilon)}|^2)_\epsilon=1$ and $\nabla_T(|u^{(\epsilon)}|^2)_\epsilon=0$ 
$\nu^{(\epsilon)}$-almost everywhere on $\p M$. 

We assume that we have chosen a conformal class that maximizes for a given 
$\epsilon>0$ so that $\sigma_1(g_0,\nu,\epsilon)=\sigma^*(\gamma,k,\epsilon)$. By Theorem \ref{multiplicity} the multiplicity is bounded independent of $\epsilon$. We may extract a sequence $\epsilon_i\to 0$ so that the $n_i$ are the same integer $n$, the background metrics
$g_0^{(i)}$ converge in $C^3$-norm to $g_0$, and the sequence of smooth measures 
$(\nu^{(i)})_{\epsilon_i}$ converges in the weak* topology to a measure $\nu$. The main theorem is the following.
\begin{theorem}\label{regularity} The measure $\nu$ has a smooth nonzero density, and there
is a set of independent first eigenfunctions which define a proper conformal harmonic map 
$u:M\to B^n$ whose image $\Sigma$ is a free boundary surface. Furthermore there is a positive constant $c$ so that $\nu=c\lambda\ ds_0$ on $\p M$ where $\lambda^2=\frac{1}{2}|\nabla u|^2$ is the induced metric on $\Sigma$. If $g$ 
is a smooth metric which has boundary measure $\nu$, then we have $L_g(\p M)=1$ and
\[ \sigma_1(g)=\sigma^*(\gamma,k).
\]
\end{theorem}

\begin{remark}
It follows from the strong maximum principle that $\lambda$ is everywhere nonzero on $\p M$; and on the interior of $M$, the conformal metric is arbitrary. This is in contrast to the case of Laplace eigenvalues on closed surfaces, where there can be points where the density vanishes, and at these point the maximizing metric has conical singularities (though the angles at the conical points are integer multiples of $2\pi$; see \cite{K}, p. 18-19). In the Steklov case, we have that the maximizing metric $g$ is smooth on the boundary, and can be taken smooth in the interior.
\end{remark}

\begin{remark}
The argument uses in an essential way that the eigenvalue is the {\em first} nonzero Steklov eigenvalue, since the analogue of Proposition \ref{pt.mass}, $\sigma_n > 2\pi n$, does not appear to hold in general. Girouard and Polterovich \cite{GP} showed that on a surface of genus zero with one boundary component, the $n$-th Steklov eigenvalue is maximized in the limit by the disjoint union of $n$ identical disks, and for $n=2$ the maximum is not attained. In this example, one could interpret $n$ disjoint disks as $n$ point masses.
\end{remark}

One tool in the proof is the following free boundary minimal surface regularity theorem
whose proof can be found in \cite{GHN}. We use $E(\cdot)$ to denote the Dirichlet integral
of a map.
\begin{proposition}\label{minimal.reg} Suppose $\Omega$ is a smooth bounded domain in
$\mathbb R^n$, and $u:M\to\Omega$ is a conformal harmonic map such that $u\in C^0(\bar{M})\cap H^1(M)$ and such that $u(\p M)\subset \p\Omega$. Assume that for any family of diffeomorphisms $F_t$ of $\Omega$ with $F_0=I$, the function $E(F_t\circ u)$ has a
critical point at $t=0$. It follows that the map $u$ is smooth up to the boundary, and the
normal derivative at the boundary is parallel to the normal vector of $\p \Omega$. If $\Omega$
is convex, then $u$ has no branch point on $\p M$.
\end{proposition}

We may think of $f_\epsilon$ as a weighted average of $f$ taken in an interval of size roughly
$\sqrt{\epsilon}$. The following lemma shows that on this scale the map $u$ of Proposition \ref{epsilon.extremal} ($u=u^{(\epsilon)}$ depends on $\epsilon$, but we suppress the $\epsilon$ for notational purposes) is almost constant,
and so it follows that for any $\epsilon>0$ and for all points $x$ in the support of the extremal measure $\nu^{(\epsilon)}$ we have that $u_\epsilon(x)$ is approximately a unit vector and
is approximately equal to $u(y)$ for all $y$ near $x$. For a point $x\in\p M$ we let 
$I_r(x)\subseteq\p M$ denote the interval of radius $r$ centered at $x$.
\begin{lemma}\label{e-scale} There is a continuous increasing function $\omega(\epsilon)$
with $\omega(0)=0$ with the following properties: For any point $x$ in the support of 
$\nu^{(\epsilon)}$ and for any $y\in I_{\omega(\epsilon)^{-1}\sqrt{\epsilon}}(x)$ we have 
$|u(x)-u(y)|\leq \omega(\epsilon)$, $|u_\epsilon(x)-u(x)|\leq\omega(\epsilon)$, and $||u_\epsilon(x)|-1|\leq\omega(\epsilon)$.
Furthermore we have $|u|\leq c$ for all $x\in \p M$ for a fixed constant $c$ independent of
$\epsilon$.
\end{lemma}

\begin{proof} We choose arclength coordinates on a component of $\p M$, and we observe that
the heat kernel at time $\epsilon$ for $|x-y|$ small has the form
\[ K_\epsilon(x,y)=(4\pi\epsilon)^{-1/2}e^{\frac{-|x-y|^2}{4\epsilon}}+O(e^{-1/\sqrt{\epsilon}}).
\]
Now let $x$ be in the support of $\nu^{(\epsilon)}$, and for simplicity we choose coordinates
so that $x=0$. We assume for convenience that the metric $g_0$ is flat near $\p M$ and we let
$(y,z)$ ($z\geq 0$) be euclidean coordinates in the half-disk $D_r(0)\subseteq M$ centered at $0$. We then scale coordinates by setting $(y,z)=(\sqrt{\epsilon}\eta,\sqrt{\epsilon}\zeta)$ for 
$(y,z)\in D_r(0)$. 
The condition $(|u|^2)_\epsilon(0)=1$ then rescales to imply the inequality 
\[ \int_{\{\eta:\ |\eta|\leq r/\sqrt{\epsilon}\}}|v(\eta,0)|^2e^{-|\eta|^2/4}\ d\eta\leq c
\]
where $v(\eta,\zeta)=u(\sqrt{\epsilon}\eta,\sqrt{\epsilon}\zeta)$. The function $v$ is harmonic and satisfies the boundary condition
\[ \frac{\p v}{\p\zeta}(\eta,0)=-\sigma_1^{(\epsilon)}v(\eta,0)\lambda_\epsilon(\eta),\ \ \lambda_\epsilon(\eta)=
\sqrt{\epsilon}\int_{\p M}K_\epsilon(\sqrt{\epsilon}\eta,x)\ d\nu^{(\epsilon)}(x).
\]
It is easily seen that $\lambda_\epsilon$ and its derivatives are small on an arbitrarily large
interval as $\epsilon$ goes to zero. Elliptic estimates imply that $v$ is bounded on finite
intervals.

To prove the final claim we choose a point $p\in\p M$ at which $|u|$ is maximized, and we choose
a unit vector $\alpha\in\mathbb R^n$ with $\alpha\cdot u(p)=|u(p)|$. From our previous result we
know that if $p$ is in the support of $\nu$ or within distance a constant times $\sqrt{\epsilon}$
from the support of $\nu$, then the bound holds. Therefore we assume that
$p$ lies in the interior of an interval of the complement of the support of $\nu$. Assume that $p$ lies on a boundary
component $\Gamma$ of length $L$, and note that if the distance from $p$ to the support
of $\nu$ is bounded from below by $1/4\min\{1,L\}$, then the result follows from the
$H^1$ boundedness and standard elliptic theory. Thus we may assume that the distance
from $p$ to the support of $\nu$ is less than $1/4\min\{1,L\}$. We choose an arclength coordinate $s$ so that  $0$ corresponds to $p$ and $a=a(\epsilon)>0$ corresponds to the nearest point on the support of $\nu$ to $p$. We then dilate coordinates by letting $x=a^{-1}s$. This produces a harmonic function $h(x,y)$ on a half disk of radius at least $2$ in the upper half plane which has 
bounded Dirichlet integral and satisfies
\[ \frac{\p h}{\p y}(x,0)=-\sigma_1h(x,0)(\mu)_\epsilon
\]
where the point $(1,0)$ is the nearest point to $0$ of the support of the measure $\mu$. We need to prove that $h(0,0)$ is bounded. First observe that if $h=0$ at some point of $D_{3/4}(0,0)$,
then the zero set has length bounded from below in $D_2(0,0)$, and the Poincar\'e 
inequality implies 
\[ \int_{D_2(0,0)}h^2\leq c\int_{D_2(0,0)}|\nabla h|^2\ dxdy.
\]
It then follows from elliptic theory that $h(0,0)$ is bounded since $(\mu)_\epsilon$
is bounded in $D_{1/2}(0,0)$. Thus we may assume that $h>0$ on $D_{3/4}(0,0)$
and we may apply the Harnack inequality to show that $h(0,0)\leq ch(x,y)$ for any
$(x,y)\in D_{1/2}(0,0)$. If $C$ denotes the half circle of radius $1$ centered at
$(1,0)$, then because $h$ is weakly superharmonic we have
\[ 1/\pi\int_Ch\ ds\leq h(1,0)=1.
\]
Therefore we have
\[ h(0,0)\leq 2c\int_{C\cap D_{1/2}(0,0)}h\ ds\leq 2c\int_Ch\ ds\leq 2\pi c.
\]

We finally prove the modulus of continuity estimate. We see that the Dirichlet integral of $u$ on 
$D_r(x)$ is small if $r$ is small since we have
\[ \int_{D_r(x)}|\nabla u|^2\ da_0=\sigma_1\int_{I_r(x)}|u|^2\ d(\nu^{(\epsilon)})_\epsilon
+\int_{\p D_r(x)\cap M}u\cdot \frac{\p u}{\p r}\ ds.
\]
From the boundedness of $u$ this implies
\[ \int_{D_r(x)}|\nabla u|^2\ da_0\leq c(\nu^{(\epsilon)})_\epsilon(I_r(x))+c\int_{\p D_r(x)\cap M}
 |\nabla u|\ ds.
\]
Capacity estimates imply $(\nu^{(\epsilon)})_\epsilon(I_r(x))\leq c|\log(r)|^{-1}$, and by the Courant-Lebesgue lemma, we may choose $r$ so that $\int_{\p D_r(x)\cap M}
 |\nabla u|\ ds$ is bounded by $c|\log(r)|^{-1/2}$. We therefore get the bound 
\[ \int_{D_r(x)}|\nabla u|^2\ da_0\leq c|\log(r)|^{-1/2}.
\]
Using the conformal invariance of the Dirichlet integral we find that the Dirichlet integral
of $v$ is small on arbitrarily large balls. The result now follows from elliptic estimates and
scaling back from $v$ to $u$. This completes the proof.
\end{proof}
The next result tells us that the $u^{(i)}$ become almost conformal in a strong sense
as $i$ tends to infinity. It will play an important role in the regularity proof.
\begin{lemma}\label{almost conformal} The Hopf differential $\tau(u^{(i)})$ tends
to zero on compact subsets of the interior of $M$; in particular the weak limit
is a conformal harmonic map. 
\end{lemma}

\begin{proof}
Let $h$ be a smooth symmetric $(0,2)$ tensor of compact support in $M$. We let 
$g_t^{(i)}=g_0^{(i)}+th$ and we consider the function
\[ F_i(t)=\int_M|\nabla_t u^{(i)}|^2\ da_t^{(i)}-\sigma_1^{(i)}\int_{\p M}|u^{(i)}|^2\ d(\nu^{(i)})_{\epsilon_i}.
\]
By Proposition \ref{epsilon.extremal} we have $F_i'(0)=0$. We let 
$B_i(t)$ denote the boundary term
\[ B_i(t)=\int_{\p M}|u^{(i)}|^2\ d(\nu^{(i)})_{\epsilon_i},
\]
so we have
\[ 0=F_i'(0)=-\int_M\langle \tau(u^{(i)}),h\rangle\ da_0-\sigma_1^{(i)}B_i'(0).
\]
where $\tau(u^{(i)})=du^{(i)}\otimes du^{(i)}-1/2|\nabla u^{(i)}|^2g_0$ is the Hopf differential of 
$u^{(i)}$. Note that
the dependence of $B_i$ on $t$ comes from the fact that the averaging is with respect to the
heat kernel of the canonical metric $g^{(i)}_0(t)$ in the conformal class of $g_t^{(i)}$. To
complete the proof that $u$ is conformal we must show that $B_i'(0)$ tends to $0$ as $i$ 
tends to infinity. Since $h$ is arbitrary we can then conclude that $\tau(u)=0$ and $u$ is conformal. To analyze $B_i'(0)$, we assume that $ds_{0,t}^{(i)}=\lambda_t^{(i)}ds_0^{(i)}$ and we
have
\[ B_i(t)=\int_{\p M}\int_{\p M}|u^{(i)}(x)|^2(K_t)_{\epsilon_i}(x,y)\ d\nu^{(i)}(y)\ \lambda_t^{(i)}(x)\ dx
\]
where $K_t$ denotes the boundary heat kernel for $g_{0,t}^{(i)}$ and $x,y$ are arclength
variables for the boundary metric $ds_0^{(i)}$. Reversing the order of integration we write
this in the form
\[ B_i(t)=\int_{\p M}F_t^{(i)}\ d\nu^{(i)},\ F_t^{(i)}(y)=\int_{\p M}|u^{(i)}(x)|^2(K_t)_{\epsilon_i}(x,y)\lambda_t^{(i)}(x)\ dx.
\]
We will show that $(F_0^{(i)})'(y)$ goes uniformly to $0$ for $y$ in the support of $\nu^{(i)}$.
We choose coordinates so that $y=0$ and we recall that the heat kernel has the form
\[ K_t(x,0)=(4\pi\epsilon_i)^{-1/2}e^{\frac{-d_t(x,0)^2}{4\epsilon_i}}+O(e^{-1/\sqrt{\epsilon_i}})
\]
where for $x$ small $d_t(x,0)=|\int_0^x\lambda_t(s)\ ds|$ and $d_0(x,0)=|x|$. We may
differentiate to obtain
\[ (K_0)'(x,0)=-(4\pi\epsilon_i)^{-1/2}\frac{x\int_0^x(\lambda_0)'(s)\ ds}{2\epsilon_i}e^{\frac{-x^2}{4\epsilon_i}}+O(e^{-1/\sqrt{\epsilon_i}}).
\]
Using the fact that $(\lambda_0)'$ is a smooth function we may write this in the form
\[ (K_0)'(x,0)=-(\lambda_0)'(0)G_{\epsilon_i}(x,0)+O\left(\frac{|x|^3}{\epsilon_i}K_0(x,0)\right)+O(e^{-1/\sqrt{\epsilon_i}})
\]
where
\[ G_{\epsilon_i}(x,0)=\frac{x^2}{2\epsilon_i}K_0(x,0).
\]
Observe that $\int_{\mathbb R}G_{\epsilon_i}(x,0)\ dx=1$, so $G_{\epsilon_i}$ is an approximation
to the unit point mass at $0$ as $\epsilon_i$ goes to $0$. Since $u$ is bounded (Lemma \ref{e-scale}), we therefore have
\[ (F_0^{(i)})'(0)=(\lambda_0)'(0)\int_I|u^{(i)}(x)|^2[K_{\epsilon_i}(x,0)-G_{\epsilon_i}(x,0)]\ dx+E_i
\]
where $I$ is a fixed interval about $0$ and $E_i$ is an error term which goes to zero with $i$.
By Lemma \ref{e-scale} we know that $|u^{(i)}(x)|^2=1+\omega_i$ for 
$x\in I_{\omega_i^{-1}\sqrt{\epsilon_i}}$, and therefore
\[ (F_0^{(i)})'(0)=(\lambda_0)'(0)\int_{-\frac{\sqrt{\epsilon_i}}{\omega_i}}^{\frac{\sqrt{\epsilon_i}}{\omega_i}}[K_{\epsilon_i}(x,0)-G_{\epsilon_i}(x,0)]\ dx+O(\int_{I\setminus I_{\omega_i^{-1}\sqrt{\epsilon_i}}}K_{\epsilon_i}(x,0)+G_{\epsilon_i}(x,0)\ dx)+\tilde{E}_i.
\]
Since $\omega_i$ tends to infinity, it is easy to check that the first two terms go to zero. Since
these bounds are uniform over points in the support of $\nu^{(i)}$ we have shown that $B_i'(0)$
converges to zero and we have completed the proof that $u$ is conformal.
\end{proof}

The following regularity result will be needed in the proof of Theorem \ref{regularity}. We 
assume here that $u\in H^1(M,\mathbb R^n)$ is a harmonic map which satisfies the boundary
condition $\nabla_\eta u=v\nu$ where $\nu$ is a probability measure on $\p M$ and
$v$ is a bounded $\nu$-measurable map to $\mathbb R^n$. This means we have for any
$\varphi\in H^1\cap C^0(\bar{M},\mathbb R^n)$ the condition
\[ \int_M\langle\nabla\varphi,\nabla u\rangle\ da_0=\int_{\p M}\langle\varphi,v\rangle\ d\nu.
\]
We will say that $u$ has radial normal derivative if $u=av$ for $\nu$-almost all points of $\p M$
for a positive $\nu$-measurable function $a$. We see that if $u$ has radial normal derivative,
then we have 
\[ \int_M\langle\nabla\varphi,\nabla u\rangle\ da_0=0
\]
for all $\varphi\in H^1\cap C^0(\bar{M},\mathbb R^n)$ with $\varphi\cdot u=0$ for $\nu$-almost 
all points of $\p M$.
\begin{lemma}\label{weak reg} Assume that $u$ is conformal and harmonic with radial normal derivative in an interval $I$ of $\p M$. Assume that there is a neighborhood $\Omega$ of $I$
in $M$ such that $|u|\geq \lambda$ in $\Omega$ for a positive number $\lambda$, and that
$u/|u|$ is continuous in $\Omega\cup I$. Assume also that $|v|\geq\lambda$ for $\nu$-almost all points of $I$. It follows that $u$ is smooth in the interior of $I$, and $|u|$ is a positive constant 
in $I$.
\end{lemma}
\begin{proof} We introduce spherical coordinates on $\mathbb R^n$ by setting $\rho=|u|$,
and choosing local coordinates $\xi^1,\ldots,\xi^{n-1}$ on $\mathbb S^{n-1}$. The equation
satisfied by $u$ (the Laplace equation) implies the equations
\[ \Delta \xi^i+\sum_{\alpha=1}^2\sum_{j,k=1}^{n-1}\Gamma^i_{jk}(\xi)\frac{\p \xi^j}{\p x^\alpha}\frac{\p \xi^k}{\p x^\alpha}+\sum_{\alpha=1}^2\frac{\p \xi^i}{\p x^\alpha}\frac{\p \rho}{\p x^\alpha}=0
\]
where $\Gamma^i_{jk}$ are the Christoffel symbols for the standard metric on $\mathbb S^{n-1}$.
The conformality condition on $u$ implies
\[ \left(\frac{\p \rho}{\p z}\right)^2=-\rho^2\left\langle\frac{\p \xi}{\p z},\frac{\p\xi}{\p z}\right\rangle
\]
where the inner product is taken with respect the spherical metric and $z=x^1+\sqrt{-1}x^2$.
Since $\rho$ is bounded it follows that $|\nabla\rho|^2\leq c|\nabla\xi|^2$. We also observe
that the $\xi^i$ satisfy the homogeneous Neumann boundary condition, so we do an
even reflection of $\xi$ and $\rho$ across $I$, and reduce the regularity (up to $C^{1,\beta}$) of 
$\xi$ to interior regularity. Thus the map $\xi$ is a continuous $H^1$ map satisfying
\[ |\Delta\xi|\leq c|\nabla\xi|^2.
\] 
Standard regularity arguments then imply that $\xi$ is Lipschitz (see Lemma 3.1 of \cite{Sc1} for a direct argument). It then follows that $\Delta\xi$ is bounded and therefore $\xi$ is
in $W^{2,p}$ for any finite $p$. Thus we conclude that $\xi$ is $C^{1,\beta}$ for any $\beta<1$. 
From conformality it follows that $\rho$ and hence the map $u$ is in $C^{1,\beta}$ for any 
$\beta<1/2$. Since $\nabla_\eta u\cdot\nabla_T u=0$ on $I$ and $\nabla_\eta u$ is parallel to
$u$ it follows that $1/2(|u|^2)_T=u\cdot\nabla_T u=0$ on $I$ where $T$ and $\eta$ are the unit
tangent and normal vectors. Thus $|u|^2$ is constant on $I$, and the higher regularity of $u$ now follows from Proposition \ref{minimal.reg}. 
\end{proof}

The key result which will allow us to prove Theorem \ref{regularity} is a regularity result for the weak* limit $\nu$ of the $\nu^{(i)}$ on $\p M$. By extracting a subsequence we may assume that the maps $u^{(i)}$ converge weakly in $H^1(M)$ to a limiting harmonic map $u$. We will show that this map is nontrivial and is regular up to the boundary, and we will use this to prove
that $\nu$ has a smooth density. 
\begin{proposition}\label{convergence} We may choose a sequence $u^{(i)}$ which converges weakly in $H^1$ to a map $u$, and so that the boundary measures $\nu^{(i)}$ converge weak* to a measure $\nu$. The limiting map $u$ is a nontrivial conformal harmonic map which is regular up to $\p M$, and the measure $\nu$ is a smooth measure.
\end{proposition}

\begin{proof}  
We consider the weak* limit of the measures $u^{(i)}(\nu^{(i)})_{\epsilon_i}$. Since for any
$\zeta\in C^0(\p M)$ we have
\[ \int_{\p M}\zeta u^{(i)}\ d(\nu^{(i)})_{\epsilon_i}=\int_{\p M}(\zeta u^{(i)})_{\epsilon_i}\ d\nu^{(i)}
=\int_{\p M}[(\zeta u^{(i)})_{\epsilon_i}-\zeta(u^{(i)})_{\epsilon_i}]\ d\nu^{(i)}+\int_{\p M}\zeta (u^{(i)})_{\epsilon_i}\ d\nu^{(i)}
\]
and the difference term tends to zero because of the continuity of $\zeta$, it follows that the
weak* limits of $u^{(i)}(\nu^{(i)})_{\epsilon_i}$ and $(u^{(i)})_{\epsilon_i}\nu^{(i)}$ are the
same. Since, by Lemma \ref{e-scale}, the $(u^{(i)})_{\epsilon_i}$ are bounded in length on the support of $\nu^{(i)}$,
it follows that the weak limit is absolutely continuous with respect to $\nu$ and has the form
$\hat{u}\nu$ where $\hat{u}$ is an $\mathbb R^n$ valued $\nu$-measurable function with the
property that $|\hat{u}|\leq 1$ for $\nu$-almost every point of $\p M$. We recall that the weak $H^1$ limit of $u^{(i)}$ is $u$, so it would be reasonable to expect that $\hat{u}=u$
$\nu$-almost everywhere, but this is not clear at this stage. 
 
 We now show that the limit function $\hat{u}$ is strictly nonzero on the support of $\nu$,
 and in fact there is $\delta_0>0$ so that $|\hat{u}|\geq \delta_0$ for $\nu$-almost every point of
 $\p M$.  To control the weak limit on the boundary we show that if there is an interval $I$
 of $\p M$ and a compact convex subset $K$ of $\mathbb R^n$ such that 
 $(u^{(i)})_{\epsilon_i}(x)\in K$ for each $x\in I\cap spt(\nu^{(i)})$, then the weak limit
 $\hat{u}$ has the property that $\hat{u}(x)\in K$ for $\nu$-almost every $x\in I\cap spt(\nu)$. This follows
 because if we are given a linear function $v=a\cdot u+b$ and a non-negative smooth
 function $\zeta$ on $M$ such that $\zeta=0$ on $\p M\setminus I$, we then have
 \[ \lim_{i\to\infty}\int_I\zeta (v^{(i)})_{\epsilon_i}\ d\nu^{(i)}=\int_{\p M}\zeta \hat{v}\ d\nu
 \]
from the observation above and since $\nu$ is the weak* limit of $\nu^{(i)}$. It follows that
an inequality $(v^{(i)})_{\epsilon_i}\geq 0$ on the support of $\nu^{(i)}$ implies the inequality
$\hat{v}\geq 0$ at $\nu$-almost every point of $I$. The statement for compact convex sets follows.

To obtain the lower bound on $\hat{u}$ we use the fact that each component function $u_j^{(i)}$
for $j=1,\ldots, n$ has a bounded number of zeroes on $\p M$ depending only on the topology
of $M$. This follows from the multiplicity bound Theorem \ref{multiplicity}. Thus we may assume
that the zero points are all convergent to a fixed finite set of points in $\p M$. We choose an
interval $I$ in the complement of these points, and we use Lemma \ref{e-scale} to see that for any compact subinterval $I'$ of $I$ and for $i$ sufficiently large, the component functions 
$(u^{(i)}_j)_{\epsilon_i}$ have a fixed sign on $I'\cap spt(\nu^{(i)})$ for $j=1,\ldots, n$ up to
terms which tend to $0$ with $i$. Again from Lemma \ref{e-scale} we have that the images 
$u^{(i)}$ lie in the convex hull $K$ of a small
neighborhood of the portion of $\mathbb S^{n-1}$ lying in an octant (the coordinates each have 
a fixed sign). Such a $K$ omits a fixed neighborhood of the origin of radius $\delta=\delta(n)$.
We thus conclude that $|\hat{u}|\geq \delta$ for $\nu$-almost every point of $\p M$. 

We now show that $u$ is nontrivial by deriving the weak form of the boundary condition satisfied by $u$. We have that $u$ is harmonic on $M$ and satisfies the condition
\[ \int_M\langle \nabla\zeta,\nabla u\rangle\ da_0=\sigma^*(\gamma,k)\int_{\p M}\zeta\cdot\hat{u}\ d\nu
\]
for all $\zeta\in H^1(M,\mathbb R^n)\cap C^0(\bar{M},\mathbb R^n)$. This follows by taking the 
weak* limit for the corresponding equations satisfied by $u^{(i)}$. Since the measure $\hat{u}\nu$
is nonzero, it follows that $u$ is not a constant function. 

It now follows that the support of $\nu$ is $\p M$ since on an open interval $I$ with $\nu(I)=0$
the map $u$ satisfies the homogeneous Neumann boundary condition and is therefore
smooth on $I$ and since $u$ is conformal we would have $\nabla u=0$ on $I$. This
would imply that $u$ is a constant map, a contradiction.

We write $u=(u_1,\ldots, u_n)$, and for each component $u_j$ we define a subset $K_j$ 
of $\p M$ to be the set of points $x$ such that
\[ \liminf_{r\to 0}|D_r(x)|^{-1}\int_{D_r(x)}u_j^2\ da_0=0.
\]
We claim that each $K_j$ is a closed subset of $\p M$. This follows because $K_j$ can be
characterized as the intersection of the closure of the zero set of $u_j$ with $\p M$. To see this
observe that if $x$ is not in the closure of the zero set of $u_j$ then there is an $r>0$ so that
$u_j$ is either positive or negative on $D_r(x)\cap M$. If $u_j$ is positive, then it follows that
for $i$ sufficiently large $u^{(i)}_j$ is also positive on $D_r(x)\cap M$ and therefore we
have $\hat{u}_j\geq 0$ for $\nu$-almost every point of $I_r(x)=\overline{D_r(x)}\cap M$. It follows
from the boundary condition satisfied by $u_j$ that $u_j$ is weakly super-harmonic and
the average of $u_j$ over $D_r(x)$ is monotone decreasing in $r$. Since $x\in K_j$, this average
must tend to $0$ as $r$ tends to $0$. Since $u_j>0$ in $D_r(x)$ we have a contradiction.
Conversely, if $x$ is in the closure of the zero set of $u_j$, then for any $r>0$ and for
$i$ sufficiently large there is a point of the zero set of $u_j^{(i)}$ in $D_{r/2}(x)$ (see the proof of Theorem \ref{k_comp}). It then follows
that the zero set of $u_j^{(i)}$ must intersect $\p D_\rho(x)$ for $r/2\leq\rho\leq r$, and therefore
the Poincar\'e inequality implies
\[ |D_r(x)|^{-1}\int_{D_r(x)}(u^{(i)}_j)^2\ da_0\leq c\int_{D_r(x)}|\nabla u^{(i)}_j|^2\ da_0
\]
for a fixed constant $c$. Since the term on the right is bounded by a fixed constant times
$|\log(r)|^{-1/2}$ (see the proof of Lemma \ref{e-scale}) and for each $r$ the term on the left converges to the average of $u$ over
$D_r(x)$ we conclude that $x\in K_j$. If $u_j<0$ on $D_r(x)$ we apply the same argument
to $-u_j$.

We now let $K$ be the intersection of the $K_j$, and we let $\Omega$ be the nonempty
open subset of $\p M$ which is the complement of $K$. Thus for $x\in \Omega$ we have
\[  \liminf_{r\to 0}|D_r(x)|^{-1}\int_{D_r(x)}|u|^2\ da_0>0.
\]
Furthermore it is easily seen that for any compact subset $C\subseteq \Omega$ there exists
$\delta_0>0$ so that
\[ \liminf_{r\to 0}|D_r(x)|^{-1}\int_{D_r(x)}|u|^2\ da_0\geq \delta_0
\]
for $x\in C$. We also observe that the set $K$ has empty interior since if we had an
interval $I\subseteq K$, then the map $u$ would be smooth and constant on $I$ which
because of the conformality would imply that $u$ is constant, a contradiction.

We now prove a uniform equicontinuity estimate for the angle $\xi_i$ given by
$\xi_i=u^{(i)}/|u^{(i)}|$ on $\Omega$. In particular this gives 
uniform equicontinuity of $u^{(i)}_{\epsilon_i}$ on the portion of the support of $\nu^{(i)}$ that lies in $\Omega$. We note that Lemma \ref{e-scale} shows that 
$(u^{(i)})_{\epsilon_i}$ has magnitude almost $1$ and is close to $u^{(i)}$ at all points of the support of $\nu^{(i)}$, so it is enough to obtain
a continuity estimate on the angle of $\xi_i$, and this we do at all points of $\Omega$. The key to doing this is
to observe that if we have a point $x\in\p M$ and an eigenfunction $v$ with $v(x)=0$,
then the zero set of $v$ must intersect the boundaries of disks $D_r(x)$ up to a fixed 
radius depending on a lower bound on $\sigma_1$ (see the proof of Theorem \ref{k_comp}). 
It then follows from the Poincar\'e inequality (since the zero set of $v$ is large enough) that
\[ |D_r(x)|^{-1}\int_{D_r(x)}v^2\leq c\int_{D_r(x)}|\nabla v|^2\ da_0
\]
for a fixed constant $c$. We can use this together with the lower bound on $u$ to get the
equicontinuity of the angle of $u^{(i)}$. We consider a unit vector $a$ orthogonal
to $(u^{(i)})_{\epsilon_i}$ and apply the previous observation to the function $v=a\cdot u^{(i)}$
which has a zero very near to the point $x$ we are considering. We also have
\[ \int_{D_r(x)}|\nabla v|^2\ da_0=\sigma_1^{(i)}\int_{I_r(x)}v^2\ d(\nu^{(i)})_{\epsilon_i}
+\int_{(\p D_r(x))\cap M}v\frac{\p v}{\p r}\ ds.
\]
Using the bound $(\nu^{(i)})_{\epsilon_i}(I_r)\leq c|\log(r)|^{-1}$ and the Courant-Lebesgue lemma which implies that $r$ may be chosen so that $|v|\leq c|\log(r)|^{-1/2}$ this easily implies the bound
\[ \int_{D_r(x)}|\nabla v|^2\ da_0\leq c|\log(r)|^{-1/2}.
\]
We conclude that for a fixed radius $r>0$ and $i$ large we have 
\[ |D_r(x)|^{-1}\int_{D_r(x)}(a\cdot u^{(i)})^2\leq c|\log(r)|^{-1/2}
\]
for any unit vector $a$ orthogonal to $\xi_i$. On the other hand, because
the weak limit of $u^{(i)}$ is strictly nonzero, for any $x$ there is a unit vector $v$ such that
\[ |D_r(x)|^{-1}\int_{D_r(x)}v\cdot u^{(i)}\geq \delta_0
\]
for a fixed positive constant $\delta_0$ and fixed $r$. These bounds imply the equicontinuity
of the angle since if $x,y$ are close together, then we consider the unit vectors $\xi_i(x)$ and
$\xi_i(y)$. If these vectors are not close together, we can decompose any unit vector
$v$ as $v=v_1+v_2$ where $v_1$ is orthogonal to $\xi_i(x)$ and $v_2$ is orthogonal to $\xi_i(y)$ with
$v_1$ and $v_2$ of bounded length. If we fix a small radius $r$ with $y\in D_{r/2}(x)$, then
we may choose $v$ as above and we contradict the lower bound.
This proves uniform equicontinuity of the angle $\xi_i$ in $\Omega$, and also of the functions 
$(u^{(i)})_{\epsilon_i}$ on the support of $\nu^{(i)}$.

We may now choose a sequence $\epsilon_i\to 0$ so that the sequence $\xi_i$ converges uniformly on $\Omega$ to a function $u'$. Thus $(u^{(i)})_{\epsilon_i}$ converges uniformly to $u'$ on the portion of the support of $\nu$ that is contained in $\Omega$. Since the measures
$(u^{(i)})_{\epsilon_i}\nu^{(i)}$ converge weakly to $\hat{u}\nu$, it follows that $u'=\hat{u}$ for 
$\nu$ almost every point of $\Omega$. 

We have now verified the hypotheses of Lemma \ref{weak reg}, and so we may
conclude that $u$ is smooth up to $\p M$, and the limiting measure $\nu$ is smooth. This
completes the proof.
\end{proof}  

\begin{proof} We now complete the proof of Theorem \ref{regularity}. We let $g$ be a smooth
metric in the conformal class of $g_0$ whose boundary arclength measure is $\nu$, so we have
$\sigma_1(g_0,\nu)=\sigma_1(g)$. If we can show that $\sigma_1(g)=\sigma^*(\gamma,k)$,
the result then follows from Proposition \ref{prop:extremal}. To see this we use the variational
characterization of $\sigma_1$
\[ \sigma_1(g_0,\nu)=\inf\left\{E(\varphi,g_0):\ \varphi\in H^1(M)\cap C^0(\bar{M}),\ \int_{\p M}\varphi^2\ d\nu=1,\ \int_{\p M}\varphi\ d\nu=0\right\}.
\]
We choose any $\varphi\in H^1(M)\cap C^0(\bar{M})$ with average $0$ and $L^2$ norm $1$
with respect to $\nu$. It then follows from the weak* convergence of $\nu^{(i)}$ to $\nu$ and the $C^2(\bar{M})$ convergence of $g_0^{(i)}$ to $g_0$ that if we let $m^{(i)}=\int_{\p M}\varphi\ d\nu^{(i)}$ and we define $\varphi^{(i)}$ by
\[ \varphi^{(i)}=\left(\int_{\p M}(\varphi-m^{(i)})^2\ d\nu^{(i)}\right)^{-1/2}(\varphi-m^{(i)}),
\]
we have $\varphi^{(i)}$ converging in $H^1(M)\cap C^0(\bar{M})$ to $\varphi$. Therefore
we have
\[ \sigma^*(\gamma,k)=\lim_{i\to\infty}\sigma_1(g_0^{(i)},\nu^{(i)})\leq \lim_{i\to\infty}E(\varphi^{(i)},g_0^{(i)})=E(\varphi).
\]
Since $\varphi$ was arbitrary we conclude that $\sigma^*(\gamma,k)\leq \sigma_1(g)$,
and since $g$ is a smooth metric we have equality. This completes the proof of 
Theorem \ref{regularity}.
\end{proof}

We now combine the above result with those of Section \ref{section:properties} to establish the main existence and regularity theorem of this paper.
\begin{theorem}\label{main} Let $M$ be either an oriented surface of genus $0$ with $k\geq 2$
boundary components or a M\"obius band. There exists on $M$ a smooth metric $g$ which
maximizes $\sigma_1L$ over all metrics on $M$. Moreover there is a branched conformal minimal 
immersion $\varphi:(M,g)\to B^n$ for some $n\geq 3$ by first eigenfunctions so that $\varphi$ is a 
$\sigma$-homothety from $g$ to the induced metric $\varphi^*(\delta)$ where $\delta$ is
the euclidean metric on $B^n$.
\end{theorem}

\begin{proof} The result for the M\"obius band follows from combining Proposition \ref{two_comp}
with Theorem \ref{regularity}. In order to apply Proposition \ref{two_comp} we must check
that the supremum of $\sigma_1 L$ for metrics on the M\"obius band is strictly larger than $2\pi$.
This follows from the fact that for the critical M\"obius band the value of $\sigma_1 L$ is $2\sqrt{3}\pi$ (see Section \ref{section:band}). It thus follows that
the conformal structure is controlled for any metric with $\sigma_1 L$ near the supremum, and the existence and regularity then follows from Theorem \ref{regularity}.

Now assume that $M$ is a smooth surface of genus $0$ with $k$ boundary components. We
prove the result inductively on $k$. First for $k=2$ it follows from Proposition \ref{two_comp}
and Theorem \ref{regularity} as for the M\"obius band where we use the fact that $\sigma^*(0,2)$
is at least as large as $\sigma_1 L$ for the critical catenoid and this is greater than $2\pi$.

Now assume that we have proven Theorem \ref{main} for $k$ boundary components
and let $M$ be a surface of genus $0$ with $k+1$ boundary components. From Proposition \ref{sigma*_inc} it then follows that $\sigma^*(0,k+1)>\sigma^*(0,k)$, and thus from
Theorem \ref{k_comp} we see that the conformal structure is controlled for metrics with
$\sigma_1 L$ near $\sigma^*(0,k+1)$. The existence and regularity then follows from
Theorem \ref{regularity}.
\end{proof}

\section{Uniqueness of the critical catenoid} \label{section:unique_cc}

In this section we will show that the critical catenoid is the only minimal annulus in $B^n$
which is a free boundary solution with the coordinate functions being first eigenfunctions.
Recall that the critical catenoid is the unique portion of a suitably scaled catenoid which defines a free boundary surface in $B^3$ (see section 3 of \cite{FS}). 
To this end, suppose $\sig=\vp(M)$ is a free boundary solution in $B^n$ where $M$ is
the Riemann surface $[-T,T]\times S^1$ and $\vp$ is a conformal harmonic map. We denote
the coordinates on $M$ by $(t,\theta)$, and we consider the conformal Killing vector
field $X=\frac{\p\vp}{\p\theta}$ defined along $\sig$. We will show that $X$ is the restriction
of a Killing vector field of $\R^n$, and hence $\sig$ is a surface of revolution which must
be the critical catenoid. We observe that since $X$ is a conformal Killing vector field it satisfies
the conditions
\[ D_{e_1}X\cdot e_2=-D_{e_2}X\cdot e_1,\ \ D_{e_1}X\cdot e_1=D_{e_2}X\cdot e_2
\]
for any orthonormal basis $e_1,e_2$ of the tangent space. 

We will have need to consider nontangential vector fields satisfying similar conditions. We call 
a (not necessarily tangential) vector field $V$ a {\it conformal vector field} if 
\[ D_{e_1}V\cdot e_2=-D_{e_2}V\cdot e_1,\ \  D_{e_1}V\cdot e_1=D_{e_2}V\cdot e_2
\]
for any oriented orthonormal basis $e_1,e_2$ of the tangent space of $\sig$. A consequence is that if $v$ is any unit vector in the tangent space the expression $D_vV\cdot v$ is constant. This may be
seen by writing $v=\cos\theta e_1+\sin\theta e_2$ and calculating
\[ D_vV\cdot v=\cos^2\theta\ D_{e_1}V\cdot e_1+\sin^2\theta\ D_{e_2}V\cdot e_2
+\sin\theta\cos\theta(D_{e_1}V\cdot e_2+D_{e_2}V\cdot e_1)
\]
and this is equal to $D_{e_1}V\cdot e_1$ for any choice of $v$.

The following result applies generally to free boundary solutions $\sig=\vp(M)$ in $B^n$ for any surface  $M$. For vector fields $V,\ W$ defined along $\sig$ (but not necessarily tangent to $\sig$) and tangent to $S^{n-1}$ along $\p\sig$ we consider the quadratic form
\[ Q(V,W)=\int_\sig \langle D V,DW\rangle\ da-\int_{\p\sig}V\cdot W\ ds.
\]
The following lemma relates the second variations of energy and area.
\begin{lemma} \label{lemma:area-energy}
If $\vp_s$ is a family of maps from $M$ to $B^n$ with $\dot{\vp}=V$, then
\[ Q(V,V)=\frac{1}{2}\frac{d^2}{ds^2} E(\vp_s)\ \ \mbox{at}\ \ s=0.
\]
If $V$ is a conformal vector field, then $Q(V,V)=S(V^\perp,V^\perp)$.
\end{lemma}
\begin{proof} By direct calculation we have
\[ \frac{1}{2}\ddot{E}=\int_\sig (\|DV\|^2+D\vp\cdot D\ddot{\vp})\ da.
\]
Integrating the second term by parts using the harmonicity of $\vp$ and the free boundary 
condition we have
\[ \frac{1}{2}\ddot{E}=\int_\sig \|DV\|^2\ da+\int_{\p\sig}x\cdot \ddot{\vp}\ ds.
\]
Since $\vp_t(x)$ is a curve on $S^{n-1}$ for fixed $x\in\p\sig$, the normal component of the
acceleration is the second fundamental form of $S^{n-1}$ in the direction $V(x)$; thus we have
\[ \frac{1}{2}\ddot{E}=\int_\sig \|DV\|^2\ da-\int_{\p\sig}\|V\|^2\ ds=Q(V,V).
\]

Now assume that $V$ is a conformal vector field and consider a variation $\vp_s$
with $\dot{\vp}=V$. We work in local conformal coordinates $(t,\theta)$ on $M$ and we let 
\[ g_{11}=\|\vp_t\|^2,\ \ g_{22}=\|\vp_\theta\|^2,\ \ \mbox{and}\ \ g_{12}=\vp_t\cdot\vp_\theta.
\]
We then have $E=\int_M (g_{11}+g_{22})\ dtd\theta$ and $A=\int_M\sqrt{g_{11}g_{22}-g_{12}^2}\ dtd\theta$. The condition that $V$ is conformal implies that at $s=0$
we have $\dot{g}_{11}=\dot{g}_{22}$ and $\dot{g}_{12}=0$. We know that the second variation of area is given by $\ddot{A}=S(V^\perp,V^\perp)$. At $s=0$ we have $g_{11}=g_{22}=\lambda$
and $g_{12}=0$. We compute 
\[ \dot{A}=\frac{1}{2}\int_M (g_{11}g_{22}-g_{12}^2)^{-1/2}(\dot{g}_{11}g_{22}+g_{11}\dot{g}_{22}-2g_{12}\dot{g}_{12})\ dtd\theta.
\]
Taking a second derivative and setting $s=0$ we obtain
\[ \ddot{A}=\frac{1}{2}\int_M\big[-\frac{1}{2}\lambda^{-3}(\lambda\dot{g}_{11}+\lambda\dot{g}_{22})^2
+2\lambda^{-1}\dot{g}_{11}\dot{g}_{22}+(\ddot{g}_{11}+\ddot{g}_{22})\big]\ dtd\theta.
\]
The conditions on $V$ imply that the first two terms cancel and we have
\[ \ddot{A}=\frac{1}{2}\int_M(\ddot{g}_{11}+\ddot{g}_{22})\ dtd\theta=\frac{1}{2}\ddot{E},
\]
and therefore $Q(V,V)=S(V^\perp,V^\perp)$ as claimed.

\end{proof}

We now specialize to the annulus case, and assume that $\sig=\vp(M)$ is a free boundary solution in $B^n$ where $M$ is the Riemann surface $[-T,T]\times S^1$ and $\vp$ is a conformal harmonic map, and denote the coordinates on $M$ by $(t,\theta)$.
We observe the following properties of the vector field $X=\fder{\vp}{\theta}$.
\begin{lemma}  \label{lemma:null}
The vector field $X$ is harmonic as a vector valued function on $\sig$.
Moreover $X$ is in the nullspace of $Q$ in the sense that $Q(X,Y)=0$ for any vector
field $Y$ along $\sig$ which is tangent to $S^{n-1}$ along $\p\sig$.
\end{lemma}
\begin{proof} Since $\vp$ is harmonic and $X=\frac{\p\vp}{\p\theta}$ it follows that
$X$ is harmonic. Thus we may integrate by parts to write
\[ Q(X,Y)=\int_{\p\sig}(D_xX\cdot Y-X\cdot Y)\ ds.
\]
We write $Y=Y^t+Y^\perp$ as the sum of vectors tangential and normal to $\sig$. Since $Y$
is tangent to $S^{n-1}$ it follows that both $Y^t$ and $Y^\perp$ are also tangent to $S^{n-1}$. Since
$X$ is perpendicular to $x$, we have the second fundamental form term $D_xX\cdot Y^\perp=0$, and thus the first term becomes  $D_xX\cdot Y^t$, and since $X$ is conformal Killing this is equal
to $-D_{Y^t}X\cdot x$. This term is the second fundamental form of $S^{n-1}$ in the directions
$Y^t$ and $X$, and thus is equal to $X\cdot Y^t$. Since $X$ is tangential to $\sig$ this is
equal to $X\cdot Y$, and thus $Q(X,Y)=0$ as claimed.
\end{proof}

For the next two lemmas we assume that $\Sigma$ is a free boundary minimal surface in $B^3$ with unit normal $\nu$. The following result, which will not be used, relates the Laplacian of a conformal vector field $V$ to the Jacobi operator.
\begin{lemma} Assume that $V$ is a conformal vector field. If we let $\psi=V\cdot\nu$,
then we have $\Delta V=(\Delta\psi+|A|^2\psi)\nu$. In particular, if $\psi$ is a Jacobi field
then $V$ is harmonic.
\end{lemma}
\begin{proof} We do the calculation in a local orthonormal basis $e_1,e_2$ which is parallel
at a point; thus $D_{e_i}e_j=h_{ij}\nu$ at the point. We first compute the tangential component
$\Delta V\cdot e_j$. We have
\[ D_{e_1}V=(D_{e_1}V\cdot e_1)\ e_1+(D_{e_1}V\cdot e_2)\ e_2+(D_{e_1}V\cdot \nu)\ \nu.
\]
Therefore
\[ D_{e_1}D_{e_1}V\cdot e_j=D_{e_1}(D_{e_1}V\cdot e_1)\ \delta_{1j}+D_{e_1}(D_{e_1}V\cdot e_2)\ \delta_{2j}-(D_{e_1}V\cdot \nu)\ h_{1j}.
\]
Similarly,
\[ D_{e_2}D_{e_2}V\cdot e_j=D_{e_2}(D_{e_2}V\cdot e_1)\ \delta_{1j}+D_{e_2}(D_{e_2}V\cdot e_2)\ \delta_{2j}-(D_{e_2}V\cdot \nu)\ h_{2j}.
\]
Using the conformal condition on $V$ we have
\[ D_{e_1}(D_{e_1}V\cdot e_1)+D_{e_2}(D_{e_2}V\cdot e_1)=D_{e_1}(D_{e_2}V\cdot e_2)
-D_{e_2}(D_{e_1}V\cdot e_2).
\]
Now we have 
\[ D_{e_1}(D_{e_2}V\cdot e_2)=D_{e_1}D_{e_2}V\cdot e_2+D_{e_2}V\cdot D_{e_1}e_2=
D_{e_2}D_{e_1}V\cdot e_2+h_{12}D_{e_2}V\cdot \nu.
\]
This implies
\[ D_{e_1}(D_{e_2}V\cdot e_2)=D_{e_2}(D_{e_1}V\cdot e_2)-h_{22}D_{e_1}V\cdot\nu
+h_{12}D_{e_2}V\cdot \nu.
\]
Thus we have
\[ D_{e_1}(D_{e_1}V\cdot e_1)+D_{e_2}(D_{e_2}V\cdot e_1)=-h_{22}D_{e_1}V\cdot\nu
+h_{12}D_{e_2}V\cdot \nu.
\]
Thus we have
\[ \Delta V\cdot e_1=-h_{22}D_{e_1}V\cdot\nu +h_{12}D_{e_2}V\cdot \nu-(D_{e_1}V\cdot\nu)h_{11}
-(D_{e_2}V\cdot \nu)h_{12}=0.
 \]
 Similarly we have $\Delta V\cdot e_2=0$, and we have shown that $\Delta V$ is a normal
 vector field.
 
 We calculate $\Delta V\cdot\nu$,
 \[ D_{e_1}D_{e_1}V\cdot\nu=(D_{e_1}V\cdot e_1)h_{11}+(D_{e_1}V\cdot e_2)h_{12}+
 D_{e_1}(D_{e_1}V\cdot\nu),
 \]
 and
 \[ D_{e_2}D_{e_2}V\cdot\nu=(D_{e_2}V\cdot e_1)h_{12}+(D_{e_2}V\cdot e_2)h_{22}+
 D_{e_2}(D_{e_2}V\cdot\nu).
 \]
 Summing these and using the conformal condition on $V$ and minimality we have
 \[ \Delta V\cdot\nu=D_{e_1}(D_{e_1}V\cdot\nu)+ D_{e_2}(D_{e_2}V\cdot\nu).
 \]
 Now $D_{e_i}V\cdot\nu=\sum_{j=1}^2(V\cdot e_j)h_{ij}+D_{e_i}\psi$ where $\psi=V\cdot\nu$.
Therefore, using the Codazzi equations and minimality we have
 \[ \Delta V\cdot\nu=\sum_{i,j=1}^2D_{e_i}(V\cdot e_j)h_{ij}+\Delta \psi.
 \]
 Now we have $\sum_{i,j=1}^2(D_{e_i}V\cdot e_j)h_{ij}=0$ by the conformal condition and
 minimality, so we get
 \[  \Delta V\cdot\nu=\Delta \psi+|A|^2\psi
 \]
 as claimed.
 
\end{proof}

Let $\mC$ denote the linear span of the functions $\{\nu_1,\nu_2,\nu_3,x\cdot\nu\}$ and we
observe the following.
\begin{lemma}  \label{lemma:space}
If $\sig$ is a free boundary solution in $B^3$ which is not a plane disk, then $\mC$ is
a four dimensional vector space of functions on $\sig$.
\end{lemma}

\begin{proof} 
If there is  a linear relation, then there would be a $v\in S^2$ and
numbers $a,b$, not both zero, such that $(av+bx)\cdot\nu\equiv 0$ on $\sig$. Thus on the boundary of $\sig$ we would have $av\cdot\nu\equiv 0$. This implies that either $a=0$ or $v$ lies in the tangent plane to $\sig$ along each component of $\p\sig$. In the latter case, the tangent plane must be constant along each component of $\p\sig$. This is because the position vector $x$
lies in the tangent plane, and $x$ can be parallel to $v$ at only a finite number of points. It follows that the two-vector $x\wedge v$ represents the tangent plane $T_x\p\sig$ at all but a finite number of points. If $T$ is the unit tangent, we have  $D_T(x\wedge v)=T\wedge v$, and this is parallel to $x\wedge v$. It follows that the tangent plane is constant along each component of $\p\sig$,
and hence each boundary component lies in a $2$-plane through the origin. It follows from
uniqueness for the Cauchy problem that the surface is a plane disk contrary to our assumption. Therefore we must have $a=0$ and hence $x\cdot\nu\equiv 0$ on $\sig$. This implies $\sig$ is a cone and hence again a plane disk since $\sig$ is smooth.
\end{proof}

We will need the following existence theorem for conformal vector fields for annular free boundary solutions.
\begin{proposition}  \label{prop:dbar}
Assume that $\sig$ is an annular free boundary solution in $B^3$. 
There is subspace $\mC_1$ of $\mC$ of dimension at least three such that for 
all $\psi\in\mC_1$ there is a tangential vector field $Y^t$ with $x\cdot Y^t=0$ on $\p\sig$ such that the vector field $Y=Y^t+\psi\ \nu$ is conformal. Furthermore we have $Q(Y,Y)=S(\psi \nu,\psi \nu)$.
For brevity of notation we denote $S(\psi \nu,\psi \nu)$ by $S(\psi,\psi)$.
\end{proposition} 

\begin{proof} 
For any function $\psi$ the equations for $Y^t$ which dictate the
condition that $Y=Y^t+\psi\ \nu$ is conformal are
\[ D_{\p_t}Y^t\cdot \p_\theta+ D_{\p_\theta}Y^t\cdot \p_t=2\psi h_{12}\ \mbox{and}\ D_{\p_t}Y^t\cdot \p_t- D_{\p_\theta}Y^t\cdot \p_\theta=2\psi h_{11}
\]
where $\p_t,\p_\theta$ denote the coordinate basis and $h_{ij}$ the second fundamental form
of $\sig$ in this basis. If we write $Y^t=u\vp_t+v\vp_\theta$ we have $u=|\vp_t|^{-2}Y^t\cdot\vp_t$
and $v=|\vp_\theta|^{-2}Y^t\cdot\vp_\theta$. Setting $\lambda=|\vp_t|^2=|\vp_\theta|^2$ we then have
\[ u_t=\lambda^{-2}\big[\lambda (D_{\p_t}Y^t\cdot \p_t+Y^t\cdot\vp_{tt})-2(\vp_t\cdot\vp_{tt})(Y^t\cdot \p_t)\big].
\]
We have $Y^t\cdot\ \vp_{tt}=\lambda^{-1}[(Y^t\cdot\vp_t)(\vp_{tt}\cdot\vp_t)+(Y^t\cdot\vp_\theta)(\vp_{tt}\cdot\vp_\theta)]$ and therefore
\[ u_t=\lambda^{-1}D_{\p_t}Y^t\cdot \p_t+\lambda^{-2}[(Y^t\cdot\vp_\theta)(\vp_{tt}\cdot\vp_\theta)
-(Y^t\cdot\vp_t)(\vp_{tt}\cdot\vp_t)].
\]
Similarly we have 
\[  v_\theta=\lambda^{-1}D_{\p_\theta}Y^t\cdot \p_\theta+\lambda^{-2}[(Y^t\cdot\vp_t)(\vp_{\theta\theta}\cdot\vp_t)-(Y^t\cdot\vp_\theta)(\vp_{\theta\theta}\cdot\vp_\theta)]
\]
Using the fact that $\vp$ is harmonic we obtain
\[ u_t-v_\theta=\lambda^{-1} (D_{\p_t}Y^t\cdot \p_t-D_{\p_\theta}Y^t\cdot\p_\theta)=2\lambda^{-1}
\psi h_{11}.
\]
We can similarly check that
\[ u_\theta+v_t=2\lambda^{-1}\psi h_{12}.
\]
Thus if we set $f=u+\sqrt{-1}v$ and $z=t+\sqrt{-1}\theta$, the equations become
\[ \frac{\p f}{\p\z}=k
\]
where $k=\lambda^{-1}\psi (h_{11}+\sqrt{-1}h_{12})$. We impose the boundary condition $\Re f=0$ on $\p M$ which is the condition that $Y^t$ be tangent to $S^2$ along $\p\sig$. 
If solvable the solution is unique up to a pure imaginary constant (corresponding to the vector
field which is a real multiple of $\frac{\p\vp}{\p\theta}$). The adjoint boundary value problem
corresponds to the operator $-\frac{\p}{\p z}$ with boundary condition $\Im f=0$. This has 
kernel the real constants, and so by the Fredholm alternative our problem is solvable
if and only if $\int_M \Re k\ dtd\theta=0$. 

We now define 
\[ \mC_1=\left\{\psi\in\mC: \int_M \Re(\lambda^{-1}\psi (h_{11}+\sqrt{-1}h_{12}))\ dtd\theta=0 \right\}
\]
which is a subspace of dimension at least three. The final statement follows from Lemma \ref{lemma:area-energy}.
\end{proof}

We are now in a position to prove the main theorem of the section.
\begin{theorem} \label{thm:uniqueness}
If $\sig$ is a free boundary minimal surface in $B^n$ which is homeomorphic
to the annulus and such that the coordinate functions are first eigenfunctions, then $n=3$ and $\sig$
is congruent to the critical catenoid.
\end{theorem}
\begin{proof} 
The multiplicity bound of Theorem \ref{multiplicity} implies that $n=3$.
Let $X=\frac{\p\vp}{\p\theta}$ be the conformal Killing vector field associated with rotations of the annulus. We first consider the case in which $\int_{\p\sig}X\ ds=0$. Since
$Q(X,X)=0$ and $\sigma_1=1$ it follows that the components of $X$ are first eigenfunctions.
In this case we can complete the proof by observing that $X$ must satisfy the Steklov boundary
condition
\[ \frac{\p X}{\p t}=X\lambda
\]
where $\lambda=|\varphi_t|=|\varphi_\theta|$ is the induced conformal metric. Since we have
$\frac{\p \varphi}{\p t}=\varphi\lambda$ we have
\[ \frac{\p^2\varphi}{\p t\p\theta}\cdot\frac{\p\varphi}{\p t}=\frac{1}{2}\frac{\p\lambda^2}{\p\theta}=0.
\] 
It follows that $\lambda$ is constant on each component of $\p M$, and therefore $M$ with the
metric induced from $\varphi$ is $\sigma$-homothetic to a flat annulus, and the rotationally
symmetric analysis of \cite{FS} implies that $\Sigma$ is the critical catenoid since it is the unique
free boundary conformal immersion by first eigenfunctions in the rotationally symmetric case. 

Now let's assume that $\int_{\p\sig}X\ ds\neq 0$. 
By Proposition \ref{prop:dbar}, for any $\psi\in\mC_1$ there is a conformal vector field whose normal component is $\psi \nu$, and which is unique up to addition of a real multiple of $X$. We denote by $Y(\psi)$ the
unique conformal vector field  with $Y(\psi)\cdot\nu=\psi$ and with $(\int_{\p\sig}Y(\psi)\ ds)\cdot (\int_{\p\sig}X\ ds)=0$. The map from $\psi$ to $Y(\psi)$ is linear. We consider the
vector space 
\[ \V=\{Y(\psi)+cX:\ \psi\in\mC_1,\ c\in\R\}, 
\]
and we observe that $\V$ is at least four dimensional since any nonzero vector field $Y(\psi)$ has a nontrivial normal component while
$X$ is tangential to $\sig$.

We define a linear transformation $T:\V\to\R^3$ by $T(V)=\int_{\p\sig}\ V\ ds$, and observe
that for dimensional reasons $T$ has a nontrivial nullspace. Thus there is a $\psi\in\mC_1$
and $c\in\R$ such that $\int_{\p\sig}\ (Y(\psi)+cX)\ ds=0$ with $V=Y(\psi)+cX\neq 0$. From the definition of $\mC_1$
there is a vector $v\in S^2$ and real numbers $a,b$ so that $\psi=(av+bx)\cdot\nu$. From
Lemmas \ref{lemma:null} and \ref{lemma:area-energy} we have $Q(Y(\psi)+cX,Y(\psi)+cX)=Q(Y(\psi),Y(\psi))=S(\psi,\psi)$. Now we
observe that $S(\psi,\psi)=a^2S(v\cdot\nu,v\cdot\nu)+2abS(x\cdot\nu,v\cdot\nu)+b^2S(x\cdot\nu,x\cdot\nu)$. Since both $x\cdot\nu$ and $v\cdot\nu$ are Jacobi fields and $x\cdot\nu=0$ on $\p\sig$,
it follows from integration by parts that $S(x\cdot\nu,v\cdot\nu)=S(x\cdot\nu,x\cdot\nu)=0$ and
so $S(\psi,\psi)=a^2S(v\cdot\nu,v\cdot\nu)$, and 
by Theorem \ref{theorem:index} we see that $Q(Y(\psi)+cX,Y(\psi)+cX)\leq 0$ and is strictly negative unless $a=0$. Since $\sigma_1=1$, we must have 
$Q(Y(\psi)+cX,Y(\psi)+cX)\geq 0$. Therefore $a=0$, and $\psi=bx\cdot\nu$. It follows that
the components of $V=bY(x\cdot\nu)+cX$ are eigenfunctions. 

We now compute the derivative in the $x$-direction along $\p\sig$. We have 
$Y(x\cdot\nu)=Y^t+(x\cdot\nu)\ \nu$. Now we have $D_x V=V$ and we observe that $D_xX$ and
$D_xY^t$ are both tangent to $\sig$ because the second fundamental form is diagonal along
$\p\sig$ and both $Y^t$ and $X$ are parallel to the unit tangent $T$ to $\p\sig$. Therefore
if we take the derivative and the dot product with $\nu$ we have $(D_x(x\cdot\nu)\ \nu)\cdot\nu=0$
if $b\neq 0$. This implies that $D_x(x\cdot\nu)=0$ along $\p\sig$. Since $x\cdot\nu$ is also zero along $\p\sig$ and $x\cdot\nu$ is a solution of the Jacobi equation it follows from uniqueness for the Cauchy problem that $x\cdot\nu\equiv 0$ on $\sig$. This contradiction shows that $b=0$, and so it follows that $V=cX$ on $\sig$
and so the components of $X$ are first eigenfunctions. We are now in the situation discussed in the first paragraph of this proof and it follows that $\sig$ is the critical catenoid.
\end{proof}

Combining this with the results of the previous sections we are now ready to prove the sharp upper bound for the annulus. Recall that 
$\sigma^*(\gamma,k)=\sup_g \,\sigma_1\,L$
where the supremum is over all smooth metrics on a surface of genus $\gamma$ with $k$ boundary components, and that by Weinstock's \cite{W} result  $\sigma^*(0,1)=2\pi$. The next result identifies $\sigma^*(0,2)$.
\begin{theorem}
For any metric on the annulus $M$ we have 
\[
         \sigma_1L\leq (\sigma_1L)_{cc}
\]
with equality if and only if $M$ is $\sigma$-homothetic to the critical catenoid.
In particular,
\[
     \sigma^*(0,2)=(\sigma_1L)_{cc}\approx 4\pi/1.2.
\]
\end{theorem}

\begin{proof}
By Theorem \ref{main} there exists on $M$ a smooth metric $g$ which maximizes $\sigma_1L$ over all metrics on $M$. Moreover there is a branched conformal minimal immersion $\varphi:(M,g)\to B^n$ for some $n\geq 3$ by first eigenfunctions so that $\varphi$ is a $\sigma$-homothety from $g$ to the induced metric $\varphi^*(\delta)$ where $\delta$ is the euclidean metric on $B^n$.
By the uniqueness result Theorem \ref{thm:uniqueness} above, this immersion is congruent to the critical catenoid.
\end{proof}

\section{Uniqueness of the critical M\"{o}bius band} \label{section:band}

In this section we show that there is a free boundary minimal embedding of the M\"obius band into $B^4$ by first Steklov eigenfunctions, and that it is the unique free boundary minimal M\"obius band in $B^n$ such that the coordinate functions are first eigenfunctions. Finally, combining this with the results of the previous sections this implies that the critical M\"obius band uniquely maximizes $\sigma_1L$ over all smooth metrics on the M\"obius band. 

We think of the M\"obius band $M$ as $\mathbb R\times S^1$ with the identification
$(t,\theta)\approx (-t,\theta+\pi)$. 

\begin{proposition} \label{prop:cmb}
There is a minimal embedding of the M\"obius band $M$ into $\mathbb R^4$
given by
\[ 
       \varphi(t,\theta)
       =(2\sinh t\cos\theta,2\sinh t\sin\theta,\cosh 2t\cos2\theta,\cosh 2t\sin2\theta)
\]
For a unique choice of $T_0$ the restriction of $\varphi$ to $[-T_0,T_0]\times S^1$ defines
a proper embedding into a ball by first Steklov eigenfunctions. 
We may rescale the radius of the ball to $1$ to get the {\em critical M\"obius band}. 
Explicitly $T_0$ is the unique positive solution of
$\coth t=2\tanh 2t$. Moreover, the maximum of $\sigma_1L$ over all rotationally symmetric metrics on the M\"obius band is uniquely achieved (up to $\sigma$-homothety) by the critical M\"obius band, and is equal to $(\sigma_1L)_{cmb}=2\pi \sqrt{3}$.
\end{proposition}

\begin{proof}
To prove this, following the approach of \cite{FS} Section 3 for rotationally symmetric metrics on the annulus, we do explicit analysis using separation of variables to compute the eigenvalues and eigenfunctions of the Dirichlet-to-Neumann map for rotationally symmetric metrics on the M\"obius band.
Consider the product $[-T,T] \times S^1$ with the identification $(t,\theta) \approx (-t,\theta+\pi)$, and with metric of the form $g=f^2(t)(dt^2 + d\theta^2)$ for a positive function $f$ such that $f(-t)=f(t)$. The outward unit normal vector at a boundary point $(T,\theta)$ is given by $\eta=f(T)^{-1}\der{t}$. To compute the Dirichlet-to-Neumann spectrum, as in Section 3 of \cite{FS}, we separate variables and look for harmonic functions of the form $u(t,\theta)=\alpha(t)\beta(\theta)$, but here additionally satisfying the symmetry condition $u(t,\theta)=u(-t,\theta+\pi)$. We obtain solutions for each nonnegative integer $n$ given by linear combinations of $\sinh(nt)\sin(n\theta)$ and $\sinh(nt)\cos(n\theta)$ when $n$ is odd, and $\cosh(nt)\sin(n\theta)$ and $\cosh(nt)\cos(n\theta)$ when $n$ is even. For $n=0$ the solutions are constants.

In order to be an eigenfunction for the Dirichlet-to-Neumann map we must have $u_\eta=\lambda u$ on the boundary, or $f(T)^{-1}u_t=\lambda u$ at the boundary point $(T,\theta)$. For $n=0$ we have $u(t,\theta)=a$ and $\lambda=0$. For $n \geq 1$ odd the eigenfunctions have $\alpha(t)=a \sinh (nt)$ and the condition is
\[
              n f(T)^{-1} \cosh(nT)=\lambda \sinh(nT).
\]
Therefore $\lambda=n f(T)^{-1} \coth(nT)$. For $n \geq 1$ even the eigenfunctions have $\alpha(t)=a \cosh (nt)$ and the condition is
\[
              n f(T)^{-1} \sinh(nT)=\lambda  \cosh(nT).
\]
Therefore $\lambda=n f(T)^{-1} \tanh(nT)$. Note that both $n f(T)^{-1} \coth(nT)$ and $n f(T)^{-1} \coth(nT)$ are increasing functions of $n$. Thus if we want to find the smallest nonzero eigenvalue $\sigma_1$ of the Dirichlet-to-Neumann map we need only consider $n=1,\,2$. We must have $\sigma_1=\min \{ f(T)^{-1} \coth (T), 2f(T)^{-1} \tanh (2T)\}$. If we fix the boundary length $2\pi f(T)$, we see that $2f(T)^{-1} \tanh(2T)$ is an increasing function of $T$ and $f(T)^{-1}\coth(T)$ is a decreasing function of $T$. It follows that if we fix the boundary length, then $\sigma_1 L$ is maximized for $T=T_0$ where $T_0$ is the positive solution of $\coth(T)=2\tanh (2T)$. Therefore, the maximum of $\sigma_1L$ over all rotationally symmetric metrics on the M\"obius band is $2\pi \coth T_0=2\pi \sqrt{3}$.

The map $\varphi: [-T_0,T_0] \times S^1 \rightarrow \mathbb{R}^4$ given by
\[
\varphi(t,\theta)
       =(2\sinh t\cos\theta,2\sinh t\sin\theta,\cosh 2t\cos2\theta,\cosh 2t\sin2\theta)
\]
is a proper conformal map into a ball by first Steklov eigenfunctions, and gives a free boundary minimal embedding into that ball.          
\end{proof}

We now show that the critical M\"obius band is the unique free boundary minimal M\"obius band in $B^n$ such that the coordinate functions are first eigenfunctions. 
First we need the following analogs of Lemma \ref{lemma:space} and Proposition \ref{prop:dbar} from section \ref{section:unique_cc}. 

\begin{lemma}
Suppose $\Sigma$ is a free boundary minimal surface in $B^n$ which is not a plane disk. Let 
\[
      \mathcal{C}=\{ E_1^\perp, \ldots, E_n^\perp, x^\perp \}
\]
where $E_1, \ldots, E_n$ are the standard basis vectors in $\mathbb{R}^n$, $x$ is the position vector, and $v^\perp$ denotes the component of $v$ normal to $\Sigma$. Then $\mathcal{C}$ is an $(n+1)$-dimensional space of vector fields along $\Sigma$.
\end{lemma}

\begin{proof}
If there is a linear relation, then there would be a $v\in S^{n-1}$ and numbers $a,b$, not both zero, such that $av^\perp+bx^\perp\equiv 0$ on $\sig$. 
If $a=0$, then $x^\perp \equiv 0$ on $\Sigma$ which would imply that $\sig$ was a cone and hence  a plane
disk since $\sig$ is smooth, a contradiction. Therefore, $a \neq 0$, and so $v^\perp=-\frac{b}{a} x^\perp$. Since $x^\perp=0$ on $\p \Sigma$ this implies that $v^\perp=0$ on $\p \Sigma$. Therefore $v$ lies in the tangent space to $\Sigma$ at each point of $\p \Sigma$. Since $x$ also lies in the tangent plane and is independent of $v$ at all but finitely many points, the $2$-vector $x\wedge v$
represents the tangent plane when they are independent. If $T$ is a unit tangent to $\p \Sigma$, we have $D_T(x \wedge v)=T\wedge v$ and this is parallel to $x\wedge v$. Therefore the tangent plane $T_x\Sigma$ is constant along each component of $\p\sig$, and each component lies in a $2$-plane through the origin. It follows from uniqueness for the Cauchy problem that 
$\Sigma$ is a plane disk contrary to our assumption. 
\end{proof}

We now specialize to the case where $\Sigma$ is the M\"obius band. That is, suppose $\sig=\vp(M)$ is a free boundary solution in $B^n$ where $M=[-T,T]\times S^1$ with the identification
$(t,\theta)\approx (-t,\theta+\pi)$, and $\vp$ is a conformal harmonic map. The following existence theorem for conformal vector fields on the M\"obius band is analogous to Proposition \ref{prop:dbar} for the annulus, and the proof is similar.

\begin{proposition} \label{prop:dbar2}
Assume that $\sig=\vp(M)$ is a free boundary minimal M\"obius band in $B^n$. 
There is subspace $\mC_1$ of $\mC$ of dimension at least $n$ such that for 
all $V \in \mC_1$ there is a tangential vector field $Y^t$ with $x\cdot Y^t=0$ on $\p\sig$ such that the vector field $Y=Y^t+V$ is conformal. Furthermore we have $Q(Y,Y)=S(V,V)$.
\end{proposition}

\begin{proof}
We lift to the oriented double cover $\widetilde{M}=[-T,T] \times S^1$ and look for a vector field $Y^t=u\vp_t+v \vp_\theta$ that is invariant $Y^t(t,\theta)=Y^t(-t,\theta+\pi)$ and hence descends to the quotient $M$. Since the lifted map $\vp$ is invariant we have $\vp_t(t,\theta)=-\vp_t(-t,\theta+\pi)$, $\vp_\theta(t,\theta)=\vp_\theta(-t,\theta+\pi)$ and so $Y^t$ is invariant if 
\begin{equation} \label{uv}
      u(-t,\theta+\pi)=-u(t,\theta) 
      \mbox{ and }
      v(-t,\theta+\pi)=v(t,\theta).
\end{equation}
The equations for $Y^t$ which dictate the condition that $Y=Y^t+V$ is conformal are
\[
       D_{\p_t} Y^t \cdot \p_\theta + D_{\p_\theta} Y^t \cdot \p_t =2h_{12} \cdot V
       \mbox{ and }
       D_{\p_t} Y^t \cdot \p_t - D_{\p_\theta} Y^t \cdot \p_\theta =2h_{11} \cdot V
\]
where $\p_t$, $\p_\theta$ denote the coordinate basis and $h_{ij}$ the vector-valued second fundamental form of $\Sigma$ in this basis, and as in the proof of Proposition \ref{prop:dbar} we have 
\[ 
      u_t-v_\theta=2\lambda^{-1} h_{11} \cdot V
      \mbox{ and }
      u_\theta+v_t=2\lambda^{-1} h_{12} \cdot V.
\]
Thus if we set $f=u+\sqrt{-1}v$ and $z=t+\sqrt{-1}\theta$, the equations become
\[ \frac{\p f}{\p\z}=k
\]
where $k=\lambda^{-1}(h_{11}\cdot V+\sqrt{-1}h_{12}\cdot V)$. We impose the boundary condition $\Re f=0$ on $\p \widetilde{M}$ which is the condition that $Y^t$ be tangent to $S^{n-1}$ along $\p\sig$.  Furthermore, by (\ref{uv}), we require that $f(-t,\theta+\pi)=-\bar{f}(t,\theta)$. If solvable the solution is unique up to a pure imaginary constant (corresponding to the vector field which is a real multiple of $\fder{\vp}{\theta}$).

Consider the operator
\[
       L: \mathcal{F}_1 \rightarrow L^2(M,\mathbb{C})
\]
defined by 
\[
       Lf=\fder{f}{\bar{z}}
\]
on the domain
\[
     \mathcal{F}_1=\{ f \in H^1(\widetilde{M},\mathbb{C}) : 
     f(-t,\theta+\pi)=-\bar{f}(t,\theta) \mbox{ on } \widetilde{M}, \; \Re f=0 \mbox{ on } \p \widetilde{M} \}.
\]       
Since $\widetilde{M}$ is compact with boundary, and $L$ is an elliptic operator with elliptic boundary condition, $L$ is a Fredholm operator on the domain $\mathcal{F}_1$. Then,
\begin{align*}
      \la Lf,g \ra &= \Re \int_{\widetilde{M}} \fder{f}{\bar{z}} \bar{g} \; dt d\theta \\
      &= \Re \left[ - \int_{\widetilde{M}} f \overline{\fder{g}{z}} \; dt d\theta 
                            + \frac{1}{2} \int_{\p \widetilde{M}} f \bar{g} \; d\theta \right] \\
      &= -\left\la f, \fder{g}{z} \right\ra + \frac{1}{2} \Re \int_{\p \tilde{M}} f \bar{g} \, d\theta.
\end{align*}
Therefore the $L^2$-adjoint $L^*$ of $L$ is defined on the domain
\[
      \mathcal{F}_2=\{ g \in H^1(\widetilde{M}, \mathbb{C}) : 
      g(-t, \theta+\pi)=\bar{g}(t,\theta) \mbox{ on } \widetilde{M}, \; \Im g=0 \mbox{ on } \p \widetilde{M} \}
\]
and is given by
\[
      L^* g = - \fder{g}{z}.
\]
This has kernel the real constants, and so by the Fredholm alternative our problem is solvable if and only if $\int_{\widetilde{M}} \Re k\ dtd\theta=0$. 

We now define 
\[ 
      \mC_1=\left\{ V  \in \mC:  \int_{\widetilde{M}} 
      \Re(\lambda^{-1}(h_{11} \cdot V+\sqrt{-1}h_{12}\cdot V))\ dtd\theta=0 \right\}
\]
which is a subspace of dimension at least $n$. The final statement follows from Lemma \ref{lemma:area-energy}.
\end{proof}

We now  prove the uniqueness result for the critical M\"obius band. The proof is almost identical to the annulus case Theorem \ref{thm:uniqueness}, however we include the details for completeness. 

\begin{theorem} \label{theorem:mobius}
Assume that $\Sigma$ is a free boundary minimal M\"obius band in $B^n$ such that the coordinate functions are first eigenfunctions. Then $n=4$ and  $\Sigma$ is the critical M\"obius band.
\end{theorem}

\begin{proof}
Let $X=\frac{\p\vp}{\p\theta}$ be the conformal Killing vector field associated with rotations of the M\"obius band. We first consider the case that $\int_{\p\sig}X\ ds=0$. Since $Q(X,X)=0$ and $\sigma_1=1$ it follows that the components of $X$ are first eigenfunctions. Therefore on $\p M$
we have
\[ \frac{\p X}{\p t}=\frac{\p^2 \varphi}{\p t\p\theta}=X\lambda
\]
where $\lambda=|\varphi_t|=|\varphi_\theta|$ is the induced conformal metric. Taking the
dot product with $\varphi_t=\varphi\lambda$ we obtain
\[ \frac{\p^2 \varphi}{\p t\theta}\cdot\frac{\p\varphi}{\p t}=\frac{1}{2}\frac{\p\lambda^2}{\p t}=0.
\]
It follows that $M$ with the induced metric is $\sigma$-homothetic to a rotationally symmetric
flat M\"obius band, and therefore by Proposition \ref{prop:cmb}, $\Sigma$ must be the critical 
M\"obius band. 

Now let's assume that $\int_{\p\sig}X\ ds\neq 0$. By Proposition \ref{prop:dbar2}, for any $V\in\mC_1$ there is a conformal vector field whose normal component is $V$, and which is unique up to addition of a real multiple of $X$. We denote by $Y(V)$ the unique conformal vector field with $Y(V)^\perp=V$ and with $(\int_{\p\sig}Y(V)\ ds)\cdot (\int_{\p\sig}X\ ds)=0$. The map from $V$ to $Y(V)$ is linear. We consider the
vector space 
\[ 
     \V=\{ Y(V)+cX:\ V \in \mC_1, \ c \in \R \}, 
\]
and we observe that $\V$ is at least $(n+1)$-dimensional since any nonzero vector field $Y(V)$ has a nontrivial normal component while $X$ is tangential to $\sig$.

We define a linear transformation $T:\V\to\R^n$ by $T(W)=\int_{\p\sig}\ W\ ds$, and observe that for dimensional reasons $T$ has a nontrivial nullspace. Thus there is a $V \in \mC_1$ and $c \in \R$ such that $\int_{\p\sig}\ (Y(V)+cX)\ ds=0$ with $Y(V)+cX\neq 0$. From the definition of $\mC_1$ there is a vector $v\in S^{n-1}$ and real numbers $a,b$ so that $V=av^\perp+bx^\perp$. From Lemmas \ref{lemma:null} and \ref{lemma:area-energy} we have 
$Q(Y(V)+cX,Y(V)+cX)=Q(Y(V),Y(V))=S(V,V)$. 
Now we observe that $S(V,V)=a^2S(v^\perp,v^\perp)+2abS(x^\perp,v^\perp)+b^2S(x^\perp,x^\perp)$. Since both $x^\perp$ and $v^\perp$ are Jacobi fields and $x^\perp=0$ on $\p\sig$,
it follows from integration by parts that $S(x^\perp,v^\perp)=S(x^\perp,x^\perp)=0$ and
so $S(V,V)=a^2S(v^\perp,v^\perp)$, and  by Theorem \ref{theorem:index} we see that $Q(Y(V)+cX,Y(V)+cX)\leq 0$ and is strictly negative unless $a=0$. Since $\sigma_1=1$, we must have 
$Q(Y(V)+cX,Y(V)+cX)\geq 0$. Therefore $Q(Y(V)+cX,Y(V)+cX) = 0$, which implies that $a=0$ so $V=bx^\perp$, and the components of $W=bY(x^\perp)+cX$ are first eigenfunctions. 

We now observe that the normal component of the derivative $D_xW$ along $\p\sig$ is
$0$ since $D_xW=W$ along $\p\Sigma$ and $W^\perp=0$ along $\p \Sigma$.
We have $Y(x^\perp)=Y^t+x^\perp$, and so
\[
     0=(D_xW)^\perp=b(D_x Y^t)^\perp + b (D_x x^\perp)^\perp + c (D_x X)^\perp.
 \]
We observe that $D_xX$ and $D_xY^t$ are both tangent to $\sig$ because the second fundamental form is diagonal in the basis $\{x,T\}$ along $\p\sig$ and both $Y^t$ and $X$ are parallel to the unit tangent $T$ to $\p\sig$. Therefore $b(D_x x^\perp)^\perp=0$, and if $b\neq 0$
we have $(D_x x^\perp)^\perp=0$. But $(D_x x^\perp)^t=0$ on $\p \Sigma$ since this is a second fundmental form term and $x^\perp=0$ on $\p \Sigma$. Therefore $D_x x^\perp=0$ along $\p\sig$. Since $x^\perp$ is also zero along $\p\sig$ and $x^\perp$ is a solution of the Jacobi equation it follows from uniqueness for the Cauchy problem for the Jacobi operator that $x^\perp\equiv 0$ on $\sig$. This contradiction shows that $b=0$, and so it follows that $X=c^{-1}W$ is a first eigenfunction. It now follows that $\sig$ must be the critical M\"obius band
by the argument given in the first paragraph of this proof.
\end{proof}

We now show that the critical M\"obius band uniquely maximizes $\sigma_1L$ over all smooth metrics on the M\"obius band. 

\begin{theorem}
For any metric on the M\"obius band $M$ we have 
\[
        \sigma_1L\leq (\sigma_1L)_{cmb} =2\pi\sqrt{3}
\]
with equality if and only if $M$ is $\sigma$-homothetic to the critical M\"obius band. 
\end{theorem}

\begin{proof}
By Theorem \ref{main} there exists on $M$ a smooth metric $g$ that maximizes $\sigma_1L$ over all metrics on $M$. Moreover there is a branched conformal minimal immersion $\varphi:(M,g)\to B^n$ for some $n\geq 3$ by first eigenfunctions so that $\varphi$ is a $\sigma$-homothety from $g$ to the induced metric $\varphi^*(\delta)$ where $\delta$ is the euclidean metric on $B^n$.
Then by the uniqueness result Theorem \ref{theorem:mobius} above we have $n=4$ and this immersion is congruent to the critical M\"obius band.
\end{proof}

\section{The asymptotic behavior as $k$ goes to infinity} \label{section:asymptotics}

In this section we discuss the limit of the extremal surfaces which were constructed in Section \ref{section:extremal}. Thus we will be considering free boundary minimal surfaces $\Sigma_k$ in 
$B^3$ which are of genus zero with $k$ boundary components. We first derive an important
result concerning the geometry of such surfaces.
\begin{proposition} Assume that  $\Sigma_k$ is a branched minimal immersion in $B^3$ which satisfies the free boundary condition and has genus zero with $k>1$ boundary components. Assume also that the coordinate functions are first Steklov eigenfunctions. It then follows
that $\Sigma_k$ is embedded, does not contain the origin $0$, and is star-shaped in the sense that
a ray from $0$ hits $\Sigma_k$ at most once. Furthermore $\Sigma_k$ is a stable minimal
surface with area bounded by $4\pi$. 
\end{proposition}
\begin{proof} From the nodal domain theorem it follows that the gradient of any first eigenfunction
$u$ is nonzero on the zero set of $u$. By assumption, for any unit vector $v$ in $\mathbb R^3$, the function $u=x\cdot v$ is a first eigenfunction. This implies that any plane through the origin
intersects $\Sigma_k$ transversally, or to put it another way, the (affine) tangent plane at each point of $\Sigma_k$ does not contain $0$. It follows that $\Sigma_k$ cannot contain the origin
since otherwise its tangent plane would violate the condition.

Let $\nu$ be a choice of unit normal vector chosen so that $x\cdot \nu$ is positive at some point
of $\Sigma_k$. We now claim that $x\cdot\nu$ is positive everywhere. To see this observe that
if $x\cdot\nu=0$ at a point $x\in\Sigma_k$ then the tangent plane contains the line between
the origin and $x$. It follows that the tangent plane contains the origin and this violates the
transversality condition.

It also follows that $\Sigma_k$ is free of branch points since, if $x\in\Sigma_k$ were a branch point, then any linear function vanishing at $x$ would have a critical point there, a contradiction.

We note that the zero set of a first Steklov eigenfunction $u$ (of a genus $0$ surface $M$) which is transverse to a boundary component $\Gamma$ must intersect it in either no points or two points. To see this
observe that if the zero set of  $u$ separated $\Gamma$ into at least four arcs on which $u$ alternates in sign, then the arcs on which $u$ is positive must lie in a single connected component of the positive set of $u$ since there are two nodal domains. Thus we can join two points $p$
and $q$ in separate arcs on which $u>0$ by a path in $M$ on which $u>0$. This path together with one of the arcs of $\Gamma$ between $p$ and $q$ then separates $M$ into two connected
components (since $M$ has genus $0$). But there is a negative arc of $\Gamma$ in each of the components and this contradicts the fact that the negative set of $u$ is connected. 

We now show that each boundary component of $\Sigma_k$ is an embedded curve. Assume we have parametrized $\Sigma_k$ by an immersion $\varphi$ from a domain surface
$M$ into the ball. We show that $\varphi$ is an embedding on each boundary component
of $M$. Suppose to the contrary that we had a boundary component  $\Gamma$ and distinct points $p$ and $q$ on $\Gamma$ such $\varphi(p)=\varphi(q)=x_0$,
then we can choose a linear function $l(x)$ vanishing at $x_0$ and at another chosen point
$x_1$ of $\varphi(\Gamma)$. The function $l\circ\varphi$ is then a first Steklov eigenfunction
vanishing at more than two points of $\Gamma$, a contradiction. 

We now show that $\Sigma_k$ is star-shaped and therefore embedded. To see this we let
$\varphi$ be a parametrizing immersion from a surface $M$. Since $M$ has genus $0$,
we may take $M$ to be a subset of $S^2$ whose complement consists of $k$ disks. For
each boundary curve $\Gamma$ of $M$ we have shown that $\varphi$ is an embedding of
$\Gamma$ into $S^2$. Since we have chosen a unit normal $\nu$ for $\Sigma_k$ by the
requirement that $x\cdot\nu>0$, we may choose the disk $D_\Gamma$ bounded by 
$\varphi(\Gamma)$ such that $\nu$ is the outward pointing unit normal to $D_\Gamma$. We now extend $\varphi$ to be an immersion of $S^2$ by filling each disk with a diffeomorphism  to 
$D_\Gamma$. We may then smooth out the $90^0$ corner along $\Gamma$ to obtain an
immersion of $S^2$ into $B^3$ with $x\cdot\nu>0$ everywhere. It follows that the map $f$ from 
$S^2$ to $S^2$ given by $f(p)=\varphi(p)/|\varphi(p)|$ is a local diffeomorphism at each point.
Therefore $f$ is a global diffeomorphism and $\Sigma_k$ is star-shaped and hence embedded.

Since $x\cdot\nu$ is a positive Jacobi field, we know (see \cite{FCS}) that 
$\Sigma_k$ is a stable minimal surface.
Finally, by the coarse upper bound of Theorem \ref{theorem:yy} we have $\sigma_1 L(\p \Sigma_k) \leq 8\pi$. 
Since $\sigma_1=1$, this implies that 
$L(\partial\Sigma_k)\leq 8\pi$. But $2A(\Sigma_k)=L(\p \Sigma_k)$ (see \cite{FS} Theorem 5.4) and this shows that the area of $\Sigma_k$ is at most $4\pi$. 
\end{proof}

We are now ready to prove the main convergence theorem.
\begin{theorem} After a suitable rotation of each $\Sigma_k$, the sequence $\Sigma_k$ converges in $C^3$ norm on compact subsets of $B^3$ to the disk $\{z=0\}$ taken with multiplicity two. Furthermore we have $\lim_{k\to\infty}A(\Sigma_k)=2\pi$ and $\lim_{k\to\infty}L(\partial\Sigma_k)=4\pi$.
\end{theorem}
\begin{proof} Since $\Sigma_k$ is stable, by the curvature estimates \cite{Sc} we have
a uniform bound on the second fundamental form of $\Sigma_k$ on each compact subset
of $B^3$; in fact, at all points a fixed geodesic distance from $\p \Sigma_k$. Since the area is also bounded we have a sequence $k'$ such that $\Sigma_{k'}$
converges in $C^3$ norm to a smooth minimal surface $\Sigma_\infty$ possibly with multiplicity (see for example \cite{CM}, Proposition 7.14 and its proof).

We show that $\Sigma_\infty$ is a disk containing the origin and the multiplicity is two. Suppose
the (integer) multiplicity is $m\geq 1$. We first show that the area of $\Sigma_{k'}$ converges to
$mA(\Sigma_\infty)$. To see this we observe that the statement follows from the $C^3$ 
convergence on compact subsets together with uniform bounds on the area near the boundary.
Specifically we can show that $A(\Sigma_k\cap (B_1\setminus B_{1-\delta}))\leq c\delta$ for 
a fixed constant $c$. This follows from approximate monotonicity in balls around boundary
points together with the global upper bound on the area. Precisely we have for $v\in\partial\Sigma_k$ and $r<1/4$ the bound $A(\Sigma_k\cap B_r(v))\leq cr^2$. The bound on the
annular region then follows by covering the boundary with $N\approx c/\delta$ balls of radius
$\delta$ (possible since $L(\partial\Sigma_k)$ is uniformly bounded). The union of the corresponding radius $2\delta$ balls then covers $\Sigma_k\cap(B_1-B_{1-\delta})$ and yields the area bound. We can prove the monotonicity of the weighted area ratio $e^{cr}r^{-2}A(B_r(v))$
for $0<r\leq 1/4$ by using an appropriate vector field which is tangent to the sphere along
the boundary. For example we can take the conformal vector field
\[ Y=(x\cdot v)x-\frac{1+|x|^2}{2}v.
\]
For any unit vector $e$ we have $\nabla_e Y\cdot e=x\cdot v$, and an easy
estimate implies $|Y-(x-v)|\leq c|x-v|^2$. We apply the first variation formula on 
$\Sigma_k\cap B_r(v)$ to obtain
\[ 2\int_{\Sigma_k\cap B_r(v)}x\cdot v\ da=\int_{\Sigma\cap\partial B_r(v)}Y\cdot\eta\ ds
\]
where $\eta$ is the unit conormal vector and there is no contribution along $\partial B^3$
since $\eta$ is orthogonal to $Y$ at those points. The information we have about $Y$ then
implies the differential inequality 
\[ (2-cr)A(\Sigma_k\cap B_r(v))\leq r\frac{d}{dr}A(\Sigma_k\cap B_r(v)).
\]
This implies the desired monotonicity statement.

To complete the proof that $A(\Sigma_{k'})\to mA(\Sigma_\infty)$ we need to show that
for $k$ large, all points of $\Sigma_k$ which are near the boundary of $B^3$ are also
near $\p \Sigma_k$ and thus have small area. In particular we must show that for $\epsilon>0$
there is a $\delta>0$ and $N$ such that for $k\geq N$
\[ \{x\in\Sigma_k:\ |x|\geq 1-\delta\}\subseteq \{x\in\Sigma_k:\ d(x,\p\Sigma_k)<\epsilon\}.
\]
This follows from curvature estimates by an indirect argument. Indeed, suppose there is 
$\epsilon>0$ such that the conclusion fails for a sequence $\delta_i\to 0$ and $k_i\to\infty$.
Then we would have a sequence of points $x_i\in \Sigma_{k_i}$ which are at least a distance
$\epsilon$ from $\p\Sigma_{k_i}$ and which converge to $\p B^3$. The curvature estimates 
imply that a fixed neighborhood of $x_i$ in $\Sigma_{k_i}$ has bounded curvature and
area, and therefore has a subsequence which converges in $C^2$ norm to a minimal surface
in $B^3$. This violates the maximum principle since these surfaces have an interior point
(the limit of $x_i$) which lies on $\p B^3$. Combining this result with the boundary
monotonicity we then conclude that $\lim_{i'}A(\Sigma_{i'})=mA(\Sigma_\infty)$.

Now we show that the multiplicity cannot be $1$. Indeed, if $m=1$, then $\Sigma_\infty$
is a smooth embedded free boundary solution inside the ball. At each boundary point $x$
of $\Sigma_\infty$ there is a density which is an integer multiple of $1/2$ and the tangent cone at any boundary point is a half-plane with integer multiplicity. This follows from the curvature estimate which holds for a blow-up sequence at a fixed distance from the boundary and implies that
the rescaled surface is locally, away from the boundary, a union of graphs over the tangent plane. By Allard-type minimal surface regularity theorems (see \cite{GJ}), a boundary point $x$ is a smooth point if and only if the density at $x$ is equal to $1/2$. We use the argument of \cite{FL} to first show that $\Sigma_\infty$ is connected. Indeed if we have two connected components $\Sigma'$ and $\Sigma''$, then we can find nearest points $x'\in\bar{\Sigma'}$ and $x''\in\bar{\Sigma''}$. By the maximum principle both points lie on $\p B^3$, and it follows that the density of $\Sigma'$ at $x'$ is $1/2$ as is the density of $\Sigma''$ at $x''$. By the boundary maximum principle (see \cite{FL}) we get a contradiction. Therefore we conclude that $\Sigma_\infty$ is connected. We do not know how to
rule out the existence of a free boundary surface with infinite topology in general, but in this
case since the metrics on $\Sigma_{k'}$ are converging to that of $\Sigma_\infty$, we may use
the argument of Proposition \ref{sigma*_inc} to punch a hole in $\Sigma_\infty$ and strictly increase 
$\sigma_1L$. Because of the convergence, when we punch a corresponding hole in 
$\Sigma_{k'}$ for $k'$ large we get a sequence of genus $0$ surfaces with $\sigma_1L$ converging to a number greater than $\lim_k\sigma^*(0,k)$ a contradiction. To be more precise the increase
we obtain on $\sigma_1L$ depends on the number $\epsilon$ in the proof of Proposition \ref{sigma*_inc}, and if the $\Sigma_k$ are converging to a smooth limiting surface, then
this number is uniform as we see directly from the argument. Therefore the surfaces cannot
be converging and we must have $m\geq 2$.

Since $m\geq 2$, it must be true that $\Sigma_\infty$ contains the origin since otherwise the distance from the origin to $\Sigma_k$ would be bounded from below, and the ray from the
nearest point would intersect $\Sigma_k$ at $m$ points contradicting the star-shaped property.
It also follows that $m=2$ since otherwise one of the rays from the origin orthogonal to the plane
$T_0\Sigma_\infty$ would intersect $\Sigma_k$ in at least two points. Now it must be true
that $x\cdot\nu$ is identically zero on $\Sigma_\infty$ and hence $\Sigma_\infty$ is a cone
(and therefore a plane since it is smooth). This follows because if $x\cdot\nu>0$ at some
point, then the ray from the origin to $x$ intersects $\Sigma_\infty$ transversally at $x$. This
implies that the ray intersects $\Sigma_{k'}$ in two points for $k'$ large, a contradiction. Therefore
$m=2$ and $\Sigma_\infty$ is a plane.

Finally, since the subsequential limit is unique up to rotation, we may rotate each $\Sigma_k$
so that the sequence $\Sigma_k$ converges to $\{z=0\}$ with multiplicity two. 
\end{proof}

Combining the previous theorem with the results of the previous section we have proven the main theorem of the section.
\begin{theorem} The sequence $\sigma^*(0,k)$ is strictly increasing in $k$
and converges to $4\pi$ as $k$ tends to infinity. For each $k$ a maximizing metric is achieved
by a free boundary minimal surface $\Sigma_k$ in $B^3$ of area less than $2\pi$. The limit of these minimal surfaces as $k$ tends to infinity is a double disk.
\end{theorem}
By a blow-up analysis it is possible to say a bit more about the surfaces $\Sigma_k$ near the
boundary of $B^3$.
\begin{remark}  For large $k$, $\Sigma_k$ is approximately a pair of nearby parallel plane disks joined by $k$ boundary bridges each of which is approximately a scaled down version of
half of the catenoid gotten by dividing with a plane containing the axis.
\end{remark}

\bibliographystyle{plain}

\end{document}